\documentclass[10pt]{amsart}
\usepackage{amsmath,amssymb,amsthm}
\usepackage{float}
\usepackage{tikz}
\usetikzlibrary{matrix, fit}
\usetikzlibrary{backgrounds}
\usepackage{braket}
\usepackage{tikz-cd}
\usetikzlibrary{matrix}
\usetikzlibrary{cd}
\usepackage{subfig}
\usepackage[colorlinks,hyperindex,linkcolor=blue]{hyperref}
\hypersetup{
pdfauthor = {J. Martinez Aguinaga, A. del Pino Gomez},
 pdftitle= {Convex integration with avoidance and the classification of hyperbolic (4,6) distributions},
 pdfsubject = {Subriemannian Geometry, tangent distributions, h-principle}} 

\newcommand{\R}{{\mathbb{R}}}

\newcommand{\NP}{{\mathbb{P}}}

\newcommand{\SA}{{\mathcal{A}}}
\newcommand{\SC}{{\mathcal{C}}}

\newcommand{\SF}{{\mathcal{F}}}

\newcommand{\SH}{{\mathcal{H}}}

\newcommand{\SL}{{\mathcal{L}}}
\newcommand{\SM}{{\mathcal{M}}}
\newcommand{\SR}{{\mathcal{R}}}
\renewcommand{\SS}{{\mathcal{S}}}
\newcommand{\ST}{{\mathcal{T}}}
\newcommand{\SU}{{\mathcal{U}}}
\newcommand{\SV}{{\mathcal{V}}}

\newcommand{\rank}{{{\operatorname{rank}}}}
\newcommand{\image}{{\operatorname{image}}}
\newcommand{\Id}{{\operatorname{Id}}}
\newcommand{\id}{{\operatorname{id}}}

\newcommand{\Op}{{\mathcal{O}p}}

\newcommand{\GL}{{\operatorname{GL}}}
\newcommand{\Hom}{{\operatorname{Hom}}}

\newcommand{\hyp}{{\operatorname{hyp}}}
\newcommand{\ellip}{{\operatorname{ell}}}
\newcommand{\cont}{{{\operatorname{cont}}}}
\newcommand{\step}{{{\operatorname{step2}}}}

\newcommand{\Gr}{{\operatorname{Gr}}}
\newcommand{\Dist}{{\operatorname{Dist}}}

\newcommand{\Pf}{{\operatorname{Pf}}}

\renewcommand{\Pr}{{{\operatorname{Pr}}}}
\newcommand{\Conv}{{\operatorname{Conv}}}
\newcommand{\Avoid}{{{\operatorname{Avoid}}}}
\newcommand{\Thin}{{{\operatorname{Thin}}}}

\newcommand{\Sym}{{\operatorname{Sym}}}

\newcommand{\HConf}{{\operatorname{H-Conf}}}
\newcommand{\barHConf}[1][]{{\overline{\operatorname{H-Conf}_{#1}}}}
\newcommand{\sbarHConf}[1][]{{\overline{\operatorname{H-Conf}^*_{#1}}}}
\newcommand{\barSA}{{\overline{\mathcal{A}}}}

\newcommand{\formConfig}{{{\oplus^{n-k}\, T^*M}}}

\newtheorem{proposition}{Proposition}[section]
\newtheorem{theorem}[proposition]{Theorem}
\newtheorem{definition}[proposition]{Definition}
\newtheorem{lemma}[proposition]{Lemma}

\newtheorem{corollary}[proposition]{Corollary}
\newtheorem{remark}[proposition]{Remark}
\newtheorem{question}[proposition]{Question}
\newtheorem{example}[proposition]{Example}

\title{Convex integration with avoidance and hyperbolic (4,6) distributions}
\subjclass[2020]{Primary: 57R57. Secondary: 57C17.}

\author{Javier Mart\'inez-Aguinaga}
\address{Universidad Complutense de Madrid, 
		Departamento de \'{A}lgebra, Geometr\'{i}a y Topolog\'{i}a, Facultad de Matem\'{a}ticas, and Instituto de Ciencias Matem\'{a}ticas CSIC-UAM-UC3M-UCM, C. Nicol\'{a}s Cabrera, 13-15, 28049 Madrid, Spain}
\email{frmart02@ucm.es}

\author{\'Alvaro del Pino}
\address{Utrecht University, Department of Mathematics, Budapestlaan 6, 3584 Utrecht, The Netherlands}
\email{a.delpinogomez@uu.nl}

\begin{document}

\begin{abstract}
This paper tackles the classification, up to homotopy, of tangent distributions satisfying various non-involutivity conditions. All of our results build on Gromov's convex integration. For completeness, we first prove that that the full $h$-principle holds for step-2 bracket-generating distributions. This follows from classic convex integration, no refinements of the theory are needed. The classification of $(3,5)$ and $(3,6)$ distributions follows as a particular case.\\

We then move on to our main example: A complete h-principle for hyperbolic (4,6) distributions. Even though the associated differential relation fails to be ample along some principal subspaces, we implement an ``avoidance trick'' to ensure that these are avoided during convex integration. Using this trick we provide the first example of a differential relation that is ample in coordinate directions but not in all directions, answering a question of Eliashberg and Mishachev.\\

This so-called ``avoidance trick'' is part of a general avoidance framework, which is the main contribution of this article. Given any differential relation, the framework attempts to produce an associated object called an ``avoidance template''. If this process is successful, we say that the relation is ``ample up to avoidance'' and we prove that convex integration applies. The example of hyperbolic (4,6) distributions shows that our framework is capable of addressing differential relations beyond the applicability of classic convex integration.
\end{abstract}
\maketitle
\setcounter{tocdepth}{1}
\tableofcontents

\section{Introduction} \label{sec:introduction}

\subsection{An overview of convex integration} \label{ssec:introConvexIntegration}

Convex integration appeared first in the work of J. Nash on $C^1$ isometric immersions/embeddings \cite{Na}. Broadly speaking, the idea is that a short immersion can be corrected, one codirection at a time, by introducing oscillations that increase its length. This process can be iterated in such a way that, after infinitely many corrections at progressively smaller scales, one obtains an isometric map that is only $C^1$.

In \cite{Gr73}, M. Gromov turned the ideas of Nash into a scheme capable of constructing and classifying solutions of more general differential relations\footnote{The crucial observation of Gromov is that the arguments presented in \cite{Na} and \cite{Sm,Hi} can be understood as integration processes. Recent work of M. Theillière  \cite{T1,T2} connects Gromov's approach to these earlier literature, showing that, for certain differential relations, one can perform convex integration \emph{without integrating}, relying instead on explicit corrugations. This yields solutions with self-similarity properties.}. The implementation is rather involved (particularly for differential relations of order higher than one), but the rough idea remains the same: We start with an arbitrary section $f$, which we correct one derivative at a time, inductively in the order of the derivatives. Namely, given a locally defined codirection $\lambda$ and an order $k$, we add oscillations to $f$ in order to adjust its pure derivative of order $k$ along $\lambda$. Once we iterate over all orders and, for each order, over a well-chosen collection of codirections, we will have corrected all the derivatives of $f$, yielding a solution of our differential relation $\SR$. There are several subtleties one must deal with:

\subsubsection{Errors}

On a given step, we correct the pure derivative of order $k$ along some $\lambda$. As we do so, we may introduce errors in all other derivatives, potentially destroying what we had achieved in previous steps. Therefore, a key part of the argument is proving that oscillations along $\lambda$ can be added at the expense of adding arbitrarily small errors in all other derivatives of order at most $k$.

This leads us to restrict our attention to open relations, because the errors will then be absorbed by openness. Note that, under this assumption of openness, we do not need to introduce infinitely many corrections at different scales anymore. This differs from the isometric immersion case.

\subsubsection{The formal datum}

Another key point is that we need to make sense of what ``correcting'' is. Indeed, at each step we must study the space of all possible oscillations of $f$ along $\lambda$ and select one that is closer to being a solution of $\SR$. In order to do this, our initial data will not be $f$ but a pair $(f,F)$, where $F$ is a formal solution of $\SR$ (i.e. a choice of Taylor polynomial solving $\SR$ at each point). The formal datum $F$ guides the convex integration process: At each step we add oscillations to both $f$ and $F$ so that their derivatives along $\lambda$ agree. The process terminates when we produce a holonomic pair $(g,G)$ (i.e. $G$ is the Taylor polynomial of $g$ at all points and, since $G$ is a formal solution, $g$ is thus a solution).

\subsubsection{Ampleness}

The argument we are outlining only works if, at each step, we can find suitable oscillations for $f$ and $F$. The way to do this is to consider $\Pr_{\lambda,F}$: the space consisting of all Taylor polynomials that differ from $F$ only in the direction of $\lambda$ (and in order $k$); we call this the \emph{principal subspace} associated to $F$ and $\lambda$. Inside of $\Pr_{\lambda,F}$ we can find $\SR_{\lambda,F}$, the subset of Taylor polynomials that are still solutions of $\SR$. Our oscillations will be chosen within this subset.

Crucially, we know that $\SR_{\lambda,F}$ is non-empty, because it contains $F$. It is also open by assumption. To carry out the proof we also require that it is \emph{ample}: this means that the connected component $\tilde\SR_{\lambda,F} \subset \SR_{\lambda,F}$ containing $F$ has the whole of $\Pr_{\lambda,F}$ as its convex hull. The geometric way of interpreting this condition is that the space of admissible order-$k$ derivatives along $\lambda$ is large and, upon integration, can be used to approximate any Taylor polynomial of order one less.

A relation $\SR$ is said to be ample if each $\SR_{\lambda,F}$ is ample.

\subsubsection{Ampleness in coordinate directions}

We are then interested in open and ample relations $\SR$. In practice, openness is readily checked, but ampleness takes some effort: A priori, it is a condition that has to be verified for each formal solution $F$, each order $k$, and each codirection $\lambda$. However, as is apparent from the explanations above, one need not study all $\lambda$, but only sufficiently many of them (finitely many per chart) to correct all derivatives. A pair $(\SR,\{\lambda_i\})$ consisting of a relation $\SR$ and a suitable collection of codirections $\{\lambda_i\}$ is said to be \emph{ample in coordinate directions} if this weaker condition holds.

In \cite[p. 171]{EM}, Y. Eliashberg and N. Mishachev posed the following question:
\begin{question} \label{quest:EM}
Is there a (geometrically meaningful) differential relation $\SR$ that is ample in coordinate directions but not ample?
\end{question}
The present article provides the first such examples. We construct them using a convex integration scheme that we call \emph{convex integration up to avoidance}. This is our main contribution and we introduce it next.

\subsection{Statement of the main result} \label{ssec:resultsConvexIntegration}

This article extends the applicability of convex integration to open relations that may not be ample nor ample in coordinate directions. The relevant (weaker) condition that they must satisfy instead is called \emph{ampleness up to avoidance} (Definition \ref{def:main}). Our main theorem reads:
\begin{theorem} \label{thm:main}
The full $C^0$-close $h$-principle holds for differential relations that are open and ample up to avoidance.
\end{theorem}
This result is restated in slightly more generality in Theorem \ref{thm:mainRestate}, Section \ref{sec:avoidance}. We emphasise that we do not require our differential relations to be of first order.

Ampleness up to avoidance effectively allows us to take the relation of interest $\SR$ and a formal solution $F: M \to \SR$, and find a smaller relation $\SR(F) \subset \SR$ that is now ample along coordinate directions and has $F$ as a formal solution. In particular we can now answer Question \ref{quest:EM}:
\begin{corollary}
There are open relations $\SR$ such that:
\begin{itemize}
\item $\SR$ is ample up to avoidance and Diff-invariant.
\item $\SR$ is not ample nor ample in coordinate directions.
\item Each relation $\SR(F)$ is ample in coordinate directions, but not necessarily ample nor Diff-invariant.
\end{itemize}
\end{corollary}
A concrete example is given in Theorem \ref{thm:46} below, which is our main application. To get there and introduce the rest of our applications, we go into the theory of tangent distributions.

\subsection{An overview of tangent distributions} \label{ssec:distributionsIntro}

A distribution $\xi$ on a manifold $M$ is a subbundle of $TM$. Since the Lie bracket is a (naturally defined) first order operator acting on vector fields, we can apply it to the sections $\Gamma(\xi)$ of $\xi$. This defines for us a sequence of modules:
\[ \Gamma^1(\xi) \subset \Gamma^2(\xi) \subset \Gamma^3(\xi) \subset \cdots \]
\[ \Gamma^1(\xi) := \Gamma(\xi), \qquad \Gamma^{i+1}(\xi) := [\Gamma^1(\xi),\Gamma^i(\xi)], \]
that we call the \emph{Lie flag}. For simplicity, we will always assume that $\xi$ is \emph{regular}, i.e. there is a distribution $\xi_i$ such that $\Gamma^i(\xi) = \Gamma(\xi_i)$. The rank of $\xi_i$ is then a measurement of the non-involutivity of $\xi$. 

\subsubsection{The nilpotentisation} \label{sssec:nilpotentisation}

One may then observe that the flag of distributions
\[ \xi_1 = \xi \subset \xi_2 \subset \xi_3 \subset \cdots \]
stabilises: i.e. there exists some smallest $i_0$ such that $\xi_i = \xi_{i_0}$ for all $i \geq i_0$. This means that $\Gamma^{i_0}(\xi)$ is involutive and thus $\xi_{i_0}$ is the tangent bundle of a foliation $\SF$ on $M$. We say that $\xi$ \textbf{bracket-generates} $\SF$. When $\xi_{i_0} = TM$, we say that $\xi$ is \textbf{bracket-generating} and we call the number $i_0$ the \textbf{step}.

We define the \textbf{nilpotentisation} $\SL(\xi)$ of $\xi$ as the graded vector bundle
\[ \xi_1 \oplus \xi_2/\xi_1 \oplus \cdots \oplus \xi_i/\xi_{i-1} \oplus \cdots \oplus \xi_{i_0}/\xi_{i_0-1}. \]
One can then observe that the composition
\[ \Gamma^j(\xi) \times \Gamma^i(\xi) \longrightarrow \Gamma^{i+j}(\xi) \longrightarrow  \Gamma^{i+j}(\xi)/\Gamma^{i+j-1}(\xi)\]
of the Lie bracket with the projection is $C^\infty$-linear. In particular, it descends to a bilinear map
\[ \Omega_{i,j}(\xi): \xi_j/\xi_{j-1} \times \xi_i/\xi_{i-1} \longrightarrow \xi_{i+j}/\xi_{i+j-1} \]
that we call the (i,j)-\textbf{curvature}. All the curvatures together endow $\SL(\xi)$ with a fibrewise Lie bracket compatible with the grading\footnote{We will say that $\SL(\xi)$ is a bundle of positively graded Lie algebras. Note that these algebras need not be modelled on a single graded Lie algebra. I.e. the bundle is not locally trivial, as the Lie bracket is allowed to vary smoothly from fibre to fibre.}.

Note that $\SL(\xi)$ (regarded as a graded vector bundle) is the graded version of $\xi_{i_0}$ (regarded as a vector bundle filtered by the $\xi_i$). In particular, there is a vector bundle isomorphism between the two that is unique up to homotopy. A concrete way of defining such an isomorphism is by selecting a metric on $TM$.

\subsubsection{Formal distributions}

$\SL(\xi)$ captures the non-involutivity of $\xi$ in a more refined manner than the Lie flag. We can think of it as a partial formal datum associated to $\xi$. Indeed, by construction, $\SL(\xi)$ is uniquely determined by the $(i_0-1)$-jet of $\xi$ at each individual point. Because we are interested in differential relations that depend only on the curvatures, we are happy to forget the full jet and focus on $\SL(\xi)$ instead; this motivates the upcoming definitions.

In general, given a foliation $\SF$, we will say that a \emph{formal $\SF$-generating distribution} is a positively graded Lie algebra bundle structure on $T\SF$ such that the degree-$1$ part is a generating set. The space of formal $\SF$-generating distributions is denoted by $\Dist^f(\SF)$; we topologise it using the (weak) $C^\infty$-topology. In particular, in families, each of the graded pieces and the bracket vary smoothly and therefore the rank of each graded piece remains constant.

We denote by $\Dist(\SF)$ the space of regular distributions that are contained in and generate by Lie brackets $\SF$. We similarly topologise it using the $C^\infty$-topology, turning the nilpotentisation procedure described above into a continuous inclusion:
\[ \SL: \Dist(\SF) \longrightarrow \Dist^f(\SF). \]
As we pointed out before, $\SL$ is only defined up to homotopy, but a concrete and consistent choice for all distributions at once can be made by choosing a metric on $TM$.

\subsubsection{Formal distributions with constraints}

The inclusion $\SL$ becomes more interesting once we introduce some natural differential constraints. Fix a $\GL$-invariant open $\SU$ in the space of positively graded Lie algebras of dimension $\rank(\SF)$. Recall that we are interested in distributions whose first layer bracket-generates the rest. This implies that the smallest $\SU$ we want to look at consists of those Lie algebras generated by their degree-$1$ part.

We will say that $F \in \Dist^f(\SF)$ is a \textbf{formal $\SU$-distribution} if $F(p) \in \SU$ for all $p \in M$. Note that an identification of $\SF_p$ with $\R^{\rank(\SF)}$ is needed for this to make sense, but the concrete choice we make is irrelevant due to $\GL$-invariance. The subspace of all such $F$ is denoted by $\Dist^f(\SF,\SU)$. This process effectively lifts $\SU$ to a Diff-invariant differential relation $\SR_\SU$ contained in the space of $(i_0-1)$-jets of distributions. Its solutions (i.e. those distributions whose nilpotentisation takes values in $\SU$) will be denoted by $\Dist(\SF,\SU)$.

The nilpotentisation map can be regarded then as an inclusion
\[ \SL_\SU: \Dist(\SF,\SU) \longrightarrow \Dist^f(\SF,\SU), \]
that we sometimes call the \textbf{scanning map}. The main question in the topological study of distributions reads:
\begin{question} \label{quest:hPrinciple}
Fix $\SU$. Is $\SL_\SU$ a weak homotopy equivalence (for any foliated manifold and relative to boundary conditions)? 
\end{question}
A positive answer to this question is often phrased by saying that the differential relation $\SR_\SU$ satisfies the full $h$-principle.

\subsubsection{Two results of Gromov} \label{sssec:Gromov}

We remind the reader that Gromov's method of flexible sheaves \cite{Gr71} applies to open and Diff-invariant differential relations to provide a full $h$-principle over \emph{open manifolds}. All the \emph{non-foliated} examples $\Dist(TM,\SU)$ described above fit within this scheme due to the openness and $\GL$-invariance of $\SU$, as long as $M$ is open.

Another well-known remark of Gromov says that the foliated case can be regarded as a parametric version of the standard case. More precisely, the following claims are equivalent:
\begin{itemize}
\item The full $h$-principle holds for $\Dist(T\SF,\SU)$, for all foliations $\SF$ of rank $n$.
\item The full $h$-principle holds for $\Dist(TM,\SU)$, for all manifolds $M$ of dimension $n$.
\end{itemize}

Due to these observations, we will tackle Question \ref{quest:hPrinciple} for the case $T\SF = TM$, where $M$ is a closed manifold.

\subsection{Applications in the study of distributions}

We now review what is known about Question \ref{quest:hPrinciple} for various choices of $\SU$. We will explain what the contributions of this article are as we go along. Our main application is Theorem \ref{thm:46} below.

\subsubsection{$h$-Principle for step $2$}

Let $M$ be a smooth manifold. One expects the answer to Question \ref{quest:hPrinciple} to be positive if the differential constraints we introduce are rather weak (i.e. if $\SU$ is large). As we stated above, the weakest assumption we are interested in is that $\SU$ consists of Lie algebras generated by their first layer, so $\Dist(TM,\SU)$ is the space of bracket-generating distributions in $M$.

Under this weak assumption we prove:
\begin{theorem}\label{thm:step2}
Let $M$ be a smooth manifold of dimension at least $4$. The complete $C^0$-close $h$-principle holds for bracket-generating distributions of step $2$ in $M$.
\end{theorem}
The result is sharp, since the $3$-dimensional case corresponds to contact structures, which are known not to abide by the full $h$-principle \cite{Ben}. The proof is presented in Section \ref{sec:flexBG} and is, in fact, a routinary application of convex integration. 

\begin{remark} \label{rem:AIM}
Theorem \ref{thm:step2} partially answers an open question raised during the workshop on Engel Structures held in April 2017 at AIM (American Institute of Mathematics, San Jose, California). Concretely, \cite[Problem 6.2]{AIM} asks whether any parallelizable $n$-manifold admits a $k$-plane field $\xi \subset TM$ with maximal growth vector. This question is further refined to ask whether any formal distribution of maximal growth admits a holonomic representative up to homotopy.

For step $2$, our result answers the question and its refinement positively and goes a bit beyond. Indeed, for $k>3$ we provide a full classification in terms of formal data and not just a existence statement. On the other hand, we do not tackle the higher step case. This is left as an interesting open question.
\end{remark}

\subsubsection{Maximal non-involutivity}

Theorem \ref{thm:step2} says that being bracket-generating is a very flexible condition (in dimension 4 onwards). As such, we would like to consider more restrictive assumptions on $\SU$. Our guiding example is Contact Topology, the study of contact structures. These are distributions whose nilpotentisation is non-degenerate, in the sense that the first curvature is a non-degenerate two-form. This is equivalent to the fact that a contact structure has as many non-trivial Lie brackets as possible. This non-degeneracy is, ultimately, responsible for the contact scanning map not being a homotopy equivalence in general \cite{Ben}, even though partial flexibility results do hold \cite{El89}. In the last few years we have seen spectacular progress in our understanding of higher-dimensional contact structures \cite{BEM,Mur}.

We will henceforth focus on distributions presenting a similar flavour of non-degeneracy. We will call this \textbf{maximal non-involutivity}; the precise meaning of this will be explained for each dimension and rank as we go along.

\subsubsection{Even-contact structures}

In even dimensions, a hyperplane field is maximally non-involutive if its curvature has corank $1$ (i.e. it has a 1-dimensional kernel). Such distributions are called even-contact structures. For them, Question \ref{quest:hPrinciple} was answered positively by McDuff \cite{McD}, proving that they are (topologically) much more flexible than contact structures. However, interesting questions about them from a geometric perspective remain open \cite{Pia}.

\subsubsection{Dimensions 3 and 4}

In dimension $3$, a regular bracket-generating distribution is necessarily a contact structure; we have already mentioned that the $h$-principle fails for them. In dimension $4$, a corank-$1$ regular bracket-generating distribution is an even-contact structure.

The remaining case in dimension $4$ corresponds to rank $2$. In this situation, a maximally non-involutive distribution is a regular bracket-generating distribution of step $3$; these are called Engel structures. Various results have appeared in the last few years regarding their classification \cite{Vo,CPPP,CPP,PV} and the classification of their submanifolds \cite{Gei,PP,CP} but a definite answer to Question \ref{quest:hPrinciple} is still open.

\subsubsection{Dimension 5}

In dimension $5$, maximally non-involutive hyperplanes are contact structures.

Rank-$3$ distributions are maximally non-involutive if they are of step $2$. In particular, as a corollary of Theorem \ref{thm:step2}, we have:
\begin{theorem} \label{thm:35and36}
Let $M$ be $5$ or $6$ dimensional. The complete $C^0$-close $h$-principle holds for maximally non-involutive rank-$3$ distributions.
\end{theorem}

Maximally non-involutive distributions of rank-$2$ are the so-called (2,3,5) distributions of Cartan \cite{Car}, which have been classified only in open manifolds \cite{DH}. If we replace maximal non-involutivity by some concrete closed growth-vector condition, there are other interesting classes of distributions (e.g. Goursat structures) whose classification is open as well.

\subsubsection{Dimension 6}

Our main application concerns rank $4$ distributions in $6$-dimensional manifolds. It turns out that maximally non-involutive $(4,6)$-distributions come in two families, \emph{elliptic} and \emph{hyperbolic}. The statement reads:
\begin{theorem} \label{thm:46}
Let $M$ be a $6$-dimensional manifold. The complete $C^0$-close $h$-principle holds for rank-$4$ distributions of hyperbolic type.
\end{theorem}
The proof can be found in Section \ref{sec:46} and it is a consequence of our main result Theorem \ref{thm:main}. We emphasise that this result requires ampleness up to avoidance and is beyond the scope of classic convex integration.

\begin{remark} \label{rem:ellipticFail}
We conjecture that the answer to Question \ref{quest:hPrinciple} is negative for elliptic $(4,6)$ distributions. In Corollary \ref{cor:ellipticity} we will show that ampleness fails badly for the differential relation that defines them.
\end{remark}

The remaining cases are: Corank-$1$ (which are even-contact structures), rank-$3$ (classified by Theorem \ref{thm:35and36}) and rank-$2$ (the so-called $(2,3,5,6)$ structures, for which nothing is known).

\subsubsection{Where is rigidity?}

This article began as our attempt to pinpoint those pairs $(k,n)$, consisting of a rank $k$ and a dimension $n$, for which the relation defining maximally non-involutive distributions is not ample. Our goal was to narrow down the (open and Diff-invariant) classes of distributions that may display rigid behaviours (as contact structures do). Our previous discussion provides a list of candidates in dimensions up to $6$. Proving rigidity and/or constructing suitable overtwisted classes in each case are interesting open questions.

Beyond dimension 6, we propose, jointly with I. Zelenko, the following concrete conjecture. Consider maximally non-involutive distributions of corank-$2$. These are always of step $2$, with the exception of dimension $4$ (the Engel case, which we leave out of the discussion). Then:
\begin{itemize}
\item In odd rank $(2l+1,2l+3)$, maximal non-involutivity means that the top-wedge of the pencil of curvatures has maximal linear span (i.e. of dimension $l+1$). The differential relation is then the complement of a singularity of codimension $l$. We expect flexibility to hold due to (classic) ampleness.
\item In even rank $(4l,4l+2)$ we see a elliptic versus hyperbolic dichotomy (just like for $(4,6)$). We expect flexibility to hold in the hyperbolic case, but avoidance to be necessary to prove it. Elliptic distributions are good candidates for rigidity.
\item In even rank $(4l+2,4l+4)$ maximal non-involutivity is equivalent to hyperbolicity. We expect flexibility based on avoidance.
\end{itemize}
We intend to address this conjecture in future work. 

\begin{remark} \label{rem:Adachi}
Recently (and independently from us) J. Adachi has also tackled the classification of non-involutive distributions using (classic) convex integration \cite{Ada2}.

He looks at distributions that are of \emph{odd rank} and satisfy a certain non-involutivity condition that we call \emph{Adachi non-integrability}. This notion is weaker than standard maximal non-involutivity (as defined in the first item above) and, in fact, it is not an intrinsic property of the distribution, but a property of a preferred coframe. In particular, our aforementioned conjecture does not follow from his statements.

On the other hand, Adachi non-integrability is stronger than being of step 2. The two notions coincide for (3,5) and (3,6). In particular, his claims should recover our step-$2$ result (Theorem \ref{thm:step2}) for odd rank. Theorem \ref{thm:35and36} would then be a particular case.

Our main result about hyperbolic (4,6) distributions (Theorem \ref{thm:46}) falls outside of the scope of his paper.
\end{remark}

\subsection{Structure of the paper}

In Sections \ref{sec:jetSpace} and \ref{sec:ampleness} we review jet spaces and convex integration. Even though the material covered is classic, we hope readers will see it as a helpful starting point into the theory of convex integration. In particular, Section \ref{sec:ampleness} provides a comparison of the various flavours of ampleness that have appeared in the literature.

Sections \ref{sec:avoidance} and \ref{sec:mainProof} introduce avoidance (pre-)templates (Definition \ref{def:avoidanceTemplate}), ampleness up to avoidance (Definition \ref{def:main}), and contain the proof of our main result Theorem \ref{thm:main}.

In Section \ref{sec:exactForms} we explain how avoidance may be used to prove the $h$-principle for exact, linearly-independent differential forms. We see this as a particularly easy example of how the theory is used and we include it for the sake of exposition. We note that classic convex integration without avoidance can already be used to prove such an $h$-principle.

In Section \ref{sec:1dimBundle} we apply avoidance to relations involving functions into $\R$. We prove that, under mild assumptions, our theory produces empty avoidance pretemplates, so the relations are not ample up to avoidance. The concrete case of functions without critical points is spelled out in Subsection \ref{ssec:functions}; note that this relation indeed does not abide by the $h$-principle, due to Morse theory. Other examples of this type can be found in Lemma \ref{lem:contact} (contact structures) and Corollary \ref{cor:ellipticity} (elliptic (4,6) structures).

In Section \ref{sec:flexBG} we study step-2 distributions and we prove Theorems \ref{thm:step2} and \ref{thm:35and36}.

In Section \ref{sec:maximalNonInvolutivity} we review the theory of rank-4 distributions in dimension 6. This sets the stage for our main application in Section \ref{sec:46}: We show that the relation $\SR^\hyp$ defining hyperbolic (4,6) distributions is ample up to avoidance. The $h$-principle for $\SR^\hyp$ follows then from Theorem \ref{thm:main}.

\textbf{Acknowledgements:} The authors are thankful to I. Zelenko for his interest in this project and his willingness to share his insights into the geometry of distributions. They are also grateful to F. Presas for his continued interest and support. During the development of this work the first author was supported by the ``Programa Predoctoral de Formaci\'on de Personal Investigador No Doctor'' scheme funded by the Basque department of education (``Departamento de Educaci\'on del Gobierno Vasco''). Similarly, the second author was funded by the NWO grant 016.Veni.192.013; this grant also funded the visits of the first author to Utrecht.

\section{Jets and relations} \label{sec:jetSpace}

We now introduce jet spaces (Subsection \ref{ssec:jetSpace}). We put particular emphasis in the geometry behind principal directions and subspaces (Subsection \ref{ssec:principalSubspaces}). This will allow us to discuss differential relations and over-relations and study them ``one direction at a time'' (Subsection \ref{ssec:overRelations}). In Subsection \ref{ssec:foliatedRelations} we introduce foliated analogues of these concepts; these will be used to phrase parametric statements in later Sections.

\subsection{Jet spaces} \label{ssec:jetSpace}

Given a smooth fibre bundle $X \rightarrow M$, we denote by $J^r(X)$ its associated space of $r$-jets. We have projections from the space of $r$-jets to the space of $r'$-jets, $r' < r$:
\[ \pi^r_{r'}: J^r(X) \rightarrow J^{r'}(X).\]
Furthermore, we write $\pi^r_M: J^r(X) \to M$ for the base projection. The projection $\pi^r_{r-1}$ is an affine fibration. Given a section $f: M \to X$, we write $j^rf: M \to J^r(X)$ for its $r$-jet, which is a holonomic section of jet space. 


\subsubsection{Jet spaces in local coordinates} \label{sssec:jetsCoords}

If we work in a local chart of $X$, we can identify $M$ with a vector space $V$ and the fibres of $X$ with a vector space $W$. Doing so allows us to identify, locally:
\[ J^r(X) \supset J^r(V \times W) \cong V \times W \times \Hom(V,W) \times \Sym^2(V,W) \times \cdots \times \Sym^r(V,W), \]
by sending a Taylor polynomial at a given point in $V$ (an element of the right-hand side) to the jet it represents in $J^r(X)$. Here $\Sym^r(V,W)$ denotes the space of symmetric tensors of order $r$ (i.e. homogeneous polynomials) with entries in $V$ and values in $W$. In particular, we are identifying the (affine) fibres of $\pi^r_{r-1}$ with their underlying vector space $\Sym^r(V,W)$.

\subsection{Principal subspaces} \label{ssec:principalSubspaces}

The following notion formalises the idea of two $r$-jets that agree except along a pure derivative of order $r$:
\begin{definition} \label{def:perpJet}
Given a hyperplane $\tau\subset T_pM$, we say that two sections $f,g: M \to X$ have the same $\bot(\tau,r)$-jet at $p \in M$ if
\[ D_p|_\tau j^{r-1}f =  D_p|_\tau j^{r-1}g, \]
where $D_p|_\tau$ means taking the differential at $p$ and restricting it to $\tau$.
\end{definition}
When $\tau$ is a hyperplane field, the $\bot(\tau,r)$-jets form a bundle, which we denote by 
\[ J^{\bot(\tau,r)}(X). \]
There are affine fibrations:
\[ \pi^r_{\bot(\tau,r)} : J^r(X) \rightarrow J^{\bot(\tau,r)}(X), \]
\[ \pi^{\bot(\tau,r)}_{r-1}: J^{\bot(\tau,r)}(X) \rightarrow J^{r-1}(X). \]

In practice, the hyperplane field $\tau$ may be defined only over an open subset $U$ of $M$, but we will still write $J^{\bot(\tau,r)}(X)$ instead of $J^{\bot(\tau,r)}(X|_U)$. Given a section $f: M \to X$, we write 
\[ j^{\bot(\tau,r)}f: M \to J^{\bot(\tau,r)}(X) \]
for the corresponding section of $\bot(\tau,r)$-jets. A section of this form is said to be \emph{holonomic}.

\begin{definition} \label{def:principalSubspaces}
The fibers of $\pi^r_{\bot(\tau,r)}$ are said to be the \textbf{principal subspaces} associated to $\tau$ (and $r$). They are all affine subspaces parallel to one another. Given $z \in J^r(X)$, we write
\[ \Pr_{\tau,z} := (\pi^r_{\bot(\tau,r)})^{-1}\left(\pi^r_{\bot(\tau,r)}(z)\right) \]
for the principal subspace that contains it. 
\end{definition}

Instead of talking about hyperplanes, it is often convenient to talk about covectors $\lambda \in T^*M$. We then write $\bot(\lambda,r) := \bot(\ker(\lambda),r)$. When $\lambda$ is a global 1-form, the bundle $J^{\bot(\lambda,r)}(X)$ is only defined in the open set $\{\lambda \neq 0\}$. However, we can define $\Pr_{\lambda,z}$ everywhere by setting $\Pr_{\lambda,z} = \{z\}$ if $\lambda = 0$.

In the context of convex integration, we will attempt to add to $f: M \to X$ oscillations of order $r$ along codirections $\lambda$; this will amount to pushing $j^rf$ along $\Pr_{\lambda,j^rf}$.
 
\subsubsection{Principal subspaces in coordinates} \label{sssec:principalSubspacesCoords}

As in Subsubsection \ref{sssec:jetsCoords}, we use vector spaces $V$ and $W$ as the local models for $M$ and the fibre of $X \to M$, respectively. 

\begin{lemma} \label{lem:principalImage}
Consider a codirection $\lambda \in T_0^*V$ and an element $z \in J^r(V \times W)$ lying over $0 \in V$. The principal subspace $\Pr_{\lambda,z}$ is the image of the affine map:
\begin{equation*}
	\begin{array}{rccl}
		W  & \longrightarrow  &  J^r(V \times W) \\
		w & \longmapsto & z + \lambda^{\otimes r} \otimes w.
	\end{array}
\end{equation*}
\end{lemma}
In particular, the vector subspace underlying $\Pr_{\lambda,z}$ is 
\[ \{\lambda^{\otimes r} \otimes w \,\mid\, w \in W\} \subset \Sym^r(V,W), \]
which we call the \textbf{principal direction} $\Pr_\lambda$. Do note that the map from Lemma \ref{lem:principalImage} is defined for all $\lambda$, and in fact varies smoothly with $\lambda$, but is an affine isomorphism between $W$ and the principal subspace if and only if $\lambda \neq 0$. An element in $\Pr_\lambda$ is said to be \textbf{principal} or \textbf{pure}. We recall:
\begin{lemma} \label{lem:principalBasis}
$\Sym^r(V,W)$ admits a basis consisting of principal elements.
\end{lemma}
The Lemma says that any two elements $F,G \in J^r(X)$ lying over the same $H \in J^{r-1}(X)$ differ by a finite sequence of modifications along principal subspaces of order $r$. See Figure \ref{fig:principalCone}.

\begin{figure}[ht]
		\includegraphics[scale=0.55]{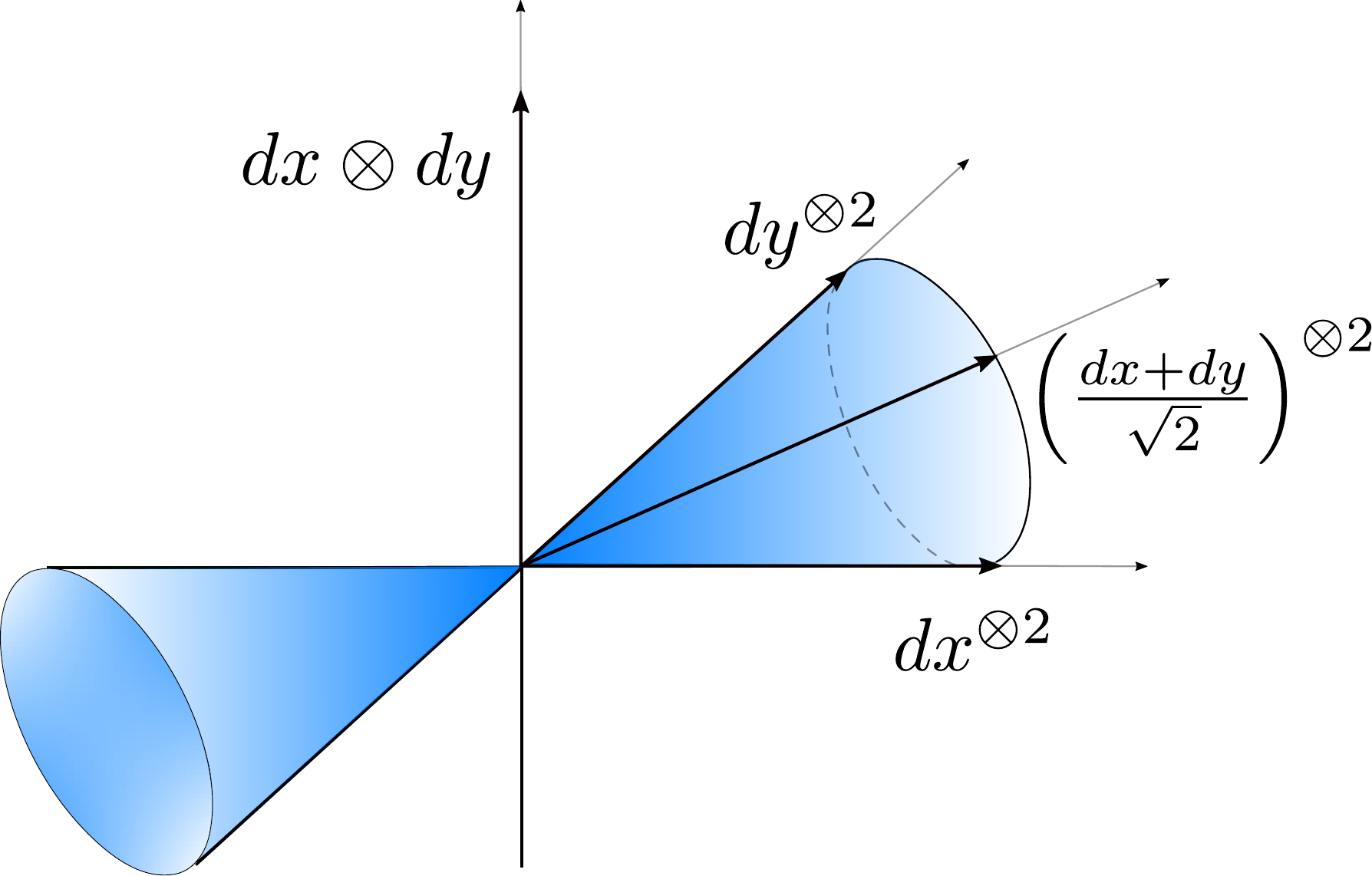}
		\centering
		\caption{The principal cone in $J^2(\R^2,\R)$. We use coordinates $(x,y)$ in the base $\R^2$. We identify the fibre of $J^2 \to J^1$ with the symmetric bilinear maps $\Sym^2(\R^2,\R)$; this is what we depict. The pure directions are then of the form $\alpha^{\otimes 2}$, with $\alpha \in T^*\R^2$. They are shown forming a blue cone, on the right hand side. We single out three principal vectors: $dx \otimes dx$, $dy \otimes dy$, and $\frac{dx+dy}{\sqrt{2}} \otimes \frac{dx+dy}{\sqrt{2}}$. The cone linearly spans the whole fibre, since these three vectors form a principal basis. The vectors contained in the left-hand-side cone are not principal, as they are of the form $-\alpha^{\otimes 2}$.} \label{fig:principalCone}
\end{figure}

\subsubsection{Principal paths} \label{sssec:principalPaths}

We have formalised the idea that two jets differ from one another along a single pure derivative by saying that they have same underlying $\perp(\tau,r)$-jet. We can similarly define the notion of two jets differing by a finite sequence of changes along pure derivatives:
\begin{definition}
Fix $z \in J^{r-1}(X)$. A \textbf{principal path} over $z$ is a sequence 
\[ \{F_i \in (\pi^r_{r-1})^{-1}(z)\}_{i=0,\cdots,I} \]
such that $F_{i+1}-F_i$ is principal. We say that $I$ is the \textbf{(principal) length} of the path.
\end{definition}
Do note that, unless $F_i = F_{i+1}$, the pair $(F_i,F_{i+1})$ uniquely determines the principal subspace $\Pr_{\lambda,F_i} = \Pr_{\lambda,F_{i+1}}$ containing both.

Fix $z \in J^{r-1}(X)$ and set $p = \pi^{r-1}_M(z) \in M$. According to Lemma \ref{lem:principalBasis}, we can fix an ordered collection of hyperplanes $\{\tau_i \subset T_pM\}_{i=0,\cdots,I}$ such that the corresponding principal directions span the fibre $(\pi^r_{r-1})^{-1}(z) \cong \Sym^r(V,W)$. If this collection is minimal, we say that it is a \textbf{(principal) basis}; in this case $I = \dim(\Sym^r(V,\R))$. It follows that any two elements $F,G \in (\pi^r_{r-1})^{-1}(z)$ can be connected by a principal path of length exactly $I$. Choosing a basis uniquely determines a principal path between $F$ and $G$. See Figure \ref{fig:pureDecomposition}.

\begin{figure}[ht]
		\includegraphics[scale=1.1]{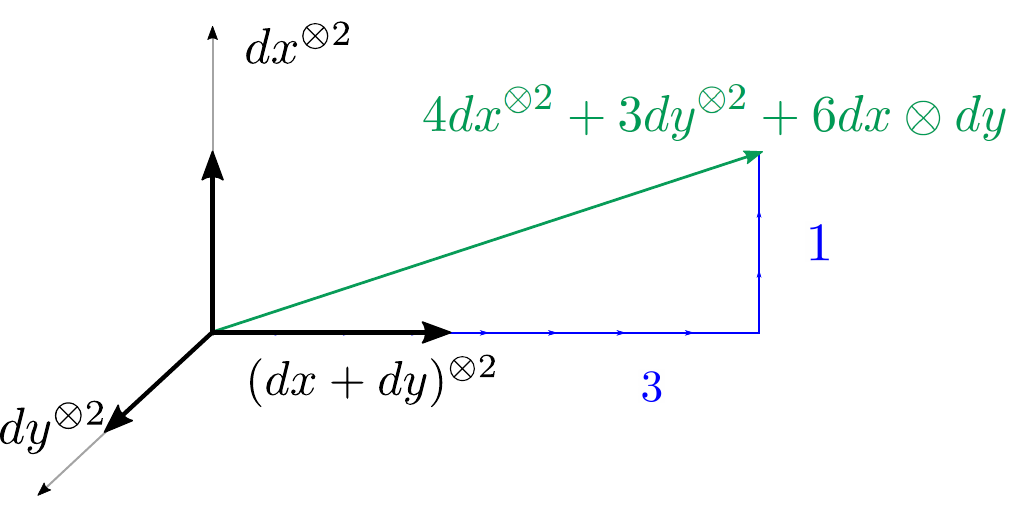}
		\centering
		\caption{Any direction in the fibre $J^r \to J^{r-1}$ can be written in a unique manner once a principal basis has been fixed. In this example, as in Figure \ref{fig:principalCone}, we work in $J^2(\R^2,\R)$. Its fibre is identified with $\Sym^2(\R^2,\R)$ and we fix $\{(dx+dy) \otimes (dx+dy),dx \otimes dx, dy \otimes dy\}$ as principal basis. This provides a preferred principal path between any two given vectors. In the figure we show a vector in green and how it connects to the origin using a path (which in this concrete case is of length $2$).} \label{fig:pureDecomposition}
\end{figure}

\subsection{Over-relations} \label{ssec:overRelations}

We are interested in finding and classifying solutions of differential relations. More generally, we define:
\begin{definition} \label{def:overRelation}
Let $X \to M$ be a fibre bundle. An \textbf{over-relation} of order $r$ is a smooth manifold $\SR$ endowed with a smooth map 
\[ \iota_\SR: \SR \to J^r(X), \]
that we sometimes call the \textbf{anchor}. If $\iota_\SR$ is an inclusion, we will say that $\SR$ is a \textbf{differential relation}. The over-relation $\SR$ is said to be \textbf{open} if the map $\iota_\SR$ is a submersion.

A \textbf{formal solution} of $\SR$ is a section $F: M \to \SR$ of $\pi^r_M \circ \iota_\SR$. A formal solution is a genuine \textbf{solution} if $\iota_\SR(F)$ is holonomic.
\end{definition}
Observe that the map $\pi^r_M \circ \iota_\SR: \SR \to M$ need not be a fibration if $\SR$ is open. It is, however, a \emph{microfibration}, meaning that the homotopy lifting property holds for small times\footnote{More generally, Gromov defines open over-relations as those where $\iota_\SR$ is a microfibration \cite[p. 175]{Gr86} but not necessarily submersive. Such generality is unnecessary for our purposes.}. A trivial but key observation is the following:
\begin{lemma}
Let $\iota_\SR: \SR \to J^r(X)$ be an open over-relation and let $\SR' \subset \SR$ be an open subset. Then $\iota_\SR|_{\SR'}$ is an open over-relation as well.
\end{lemma}

The main motivating example for our usage of over-relations is the following:
\begin{example} \label{ex:lowerOrderRelation}
Given an over-relation $\iota_\SR: \SR \to J^r(X)$ of order $r$, the projection 
\[ \pi^r_{r'} \circ \iota_\SR: \SR \to J^{r'}(X) \]
is an over-relation of order $r' < r$. That is, over-relations are crucial if we want to construct solutions inductively on $r$.

Observe that openess of $\iota_\SR$ implies openness of $\SR$, since the maps $\pi^r_{r'}$ are submersive.
\end{example}

We need to understand how over-relations $\iota_\SR: \SR \to J^r(X)$ relate to our idea of introducing oscillations along a given principal subspace. Given $z \in J^r(E)$, we write 
\[ \SR_{\tau,z} := \iota_\SR^{-1}(\Pr_{\tau,z}) \]
for the restriction of $\SR$ to the principal subspace containing $z$. Given $F \in \SR$, we write 
\[ \Pr_{\tau,F} := \Pr_{\tau,\iota_\SR(F)} \qquad \SR_{\tau,F} := \iota_\SR^{-1}(\Pr_{\tau,F}). \]
We use analogous notation when dealing with covectors $\lambda$ instead of hyperplanes $\tau$.

\subsection{The foliated setting} \label{ssec:foliatedRelations}

One can generalise all the discussion up to this point to differential relations that vary with respect to a parameter. The language of foliations is convenient for this purpose.

We fix a foliated manifold $(N,\SF)$ and a bundle $Y \to N$. We write $J^r(Y;\SF)$ for the bundle of leafwise $r$-jets (i.e. equivalence classes of sections of $Y$ up to $r$-order tangency along the leaves of $\SF$). Note that $J^r(Y;\SF)$ is the disjoint union of all the individual $J^r(Y|_L)$, with $L$ ranging over the leaves of $\SF$, endowed with the natural smooth structure.

Apart from the usual projections among these bundles for varying $r$, we have a forgetful map
\[ \pi^\SF: J^r(Y) \to J^r(Y;\SF), \]
that just remembers the leafwise jets. In particular, if $L$ is a leaf of $\SF$ we obtain a projection map $J^r(Y)|_L \to J^r(Y|_L)$.

A section of $J^r(Y;\SF)$ is holonomic if its restriction to each leaf is holonomic. A section $F$ of $J^r(Y)$ is leafwise holonomic if the corresponding $\pi^\SF \circ F|_L$ are holonomic; this is weaker than $F$ itself being holonomic.

\begin{definition}
A \textbf{foliated over-relation} is a map 
\[ \iota_\SS: \SS \to J^r(Y;\SF) \]
where $\SS$ is smooth manifold. We say it is \textbf{open} if it is submersive.
\end{definition}
We can restrict $\SS$ to a leaf $L$ of $\SF$ and yield a (standard) over-relation $\SS|_L \to J^r(Y|_L)$.

\subsubsection{Parametric lifts of over-relations} \label{sssec:parametricLift}

The most important example of foliated over-relation is the following:
\begin{definition} \label{def:parametricLift}
Let $X \to M$ be a bundle, $\iota_\SR: \SR \to J^r(Y)$ an over-relation, and $K$ a compact manifold serving as parameter space. Set $N := M \times K$ and write $\pi_M$ and $\pi_K$ for the projections mapping to $M$ and $K$, respectively. Endow $N$ with the foliation $\SF$ consisting of the fibres of $\pi_K$. Set $Y := X \times K = \pi_M^*X$ and $\SS := \SR \times K = \pi_M^*\SR$.

The \textbf{parametric lift} of $\SR$ to $M \times K$ is the foliated over-relation 
\[ \pi_M^*\iota_\SR:\quad \SS \to J^r(Y;\SF). \]
\end{definition}
Do note that the leaves of $\SF$ are copies $M \times \{k\}$ of $M$ and, identifying both using $\pi_M$, we have that $\SS$ restricts to $\SR$ along each $M \times \{k\}$. Families $(F_k)_{k \in K}$ of formal solutions of $M \to \SR$ are then equivalent to formal solutions $F: N \to \SS$. The family $(F_k)_{k \in K}$ consists of holonomic sections if and only if $F$ is (leafwise) holonomic.

We remark:
\begin{lemma}
The parametric lift of an open over-relation is open.
\end{lemma}

\subsubsection{Non-foliated preimage} \label{sssec:forgetfulLift}

Any $\iota_\SS: \SS \to J^r(Y;\SF)$ defines an over-relation in $J^r(Y)$ by pullback. This will be relevant later on, because it will allow us to rephrase statements about $\SS$ to statements about the pullback (therefore reducing the foliated theory to the non-foliated one).
\begin{definition}
Let $\iota_\SS: \SS \to J^r(Y;\SF)$ be a foliated over-relation. Its \textbf{non-foliated preimage} is the over-relation
\[ \SS^* := \{(F,z) \in \SS \times J^r(Y) \,\mid\, \iota_\SS(F) = \pi^\SF(z) \}, \]
with anchor map $\iota_\SS^*: \SS^* \to J^r(Y)$ defined by the expression $\iota_\SS^*(F,z) := z$.
\end{definition}
I.e. $\SS^*$ is the pullback of $J^r(Y) \to J^r(Y;\SF)$ to $\SS$. It follows that:
\begin{lemma}
The non-foliated preimage of an open, foliated over-relation is open.
\end{lemma}

\section{Ampleness and convex integration} \label{sec:ampleness}

In this Section we recall some key ideas behind convex integration. Our goal is not to be comprehensive, but rather to fix notation and discuss its different incarnations, as introduced by Gromov \cite{Gr86}. We also borrow from Spring \cite{Spring} and Eliashberg-Mishachev \cite{EM}.

We first recall the three standard flavours of ampleness; each of them is the basis of a concrete implementation of convex integration. Classic ampleness is explained in Subsection \ref{ssec:ampleness1}. Ampleness along principal frames (often called ampleness in coordinate directions) is explained in Subsection \ref{ssec:ampleness2}. Ampleness in the sense of convex hull extensions appears in Subsection \ref{ssec:ampleness3}.

We then compare them in Subsection \ref{ssec:comparison}. This will clarify how ampleness up to avoidance (to appear in Section \ref{sec:avoidance}) fits within this greater context.

\subsection{Ampleness in affine spaces} \label{ssec:affineAmpleness}

We define ampleness for subsets of affine spaces first. We adapt it to relations in jet spaces in upcoming Subsections.
\begin{definition} \label{def:ampleAffine}
Let $A$ be an affine space. Then:
\begin{itemize}
\item Let $B \subset A$ be a subset. Given $b \in B$, we write $B_b$ for the path-component containing it. We say that $B$ is \textbf{ample} if the convex hull $\Conv(B,b) := \Conv(B_b)$ of each $B_b \subset B$ is the whole of $A$.
\item Let $C$ be a topological space and $\iota: C \to A$ be a continuous map. The map $\iota$ is \textbf{ample} if $\Conv(C,c) := \Conv(\iota(C_c)) = A$ for each $c \in C$. 
\end{itemize}
Furthermore, we say that ampleness holds trivially if for each $c \in C$ either $\iota(C_C)=\emptyset$ or $\iota(C_C)=A$.
\end{definition}

A particularly relevant case in the examples to come is the following:
\begin{example} \label{def:thinSet}
A stratified subset $\Sigma \subset \R^n$ of codimension at least $2$ is said to be \emph{thin}. Its complement is ample. See Figure \ref{fig:thinSet}.
\end{example}

\begin{figure}[ht]
		\includegraphics[scale=0.08]{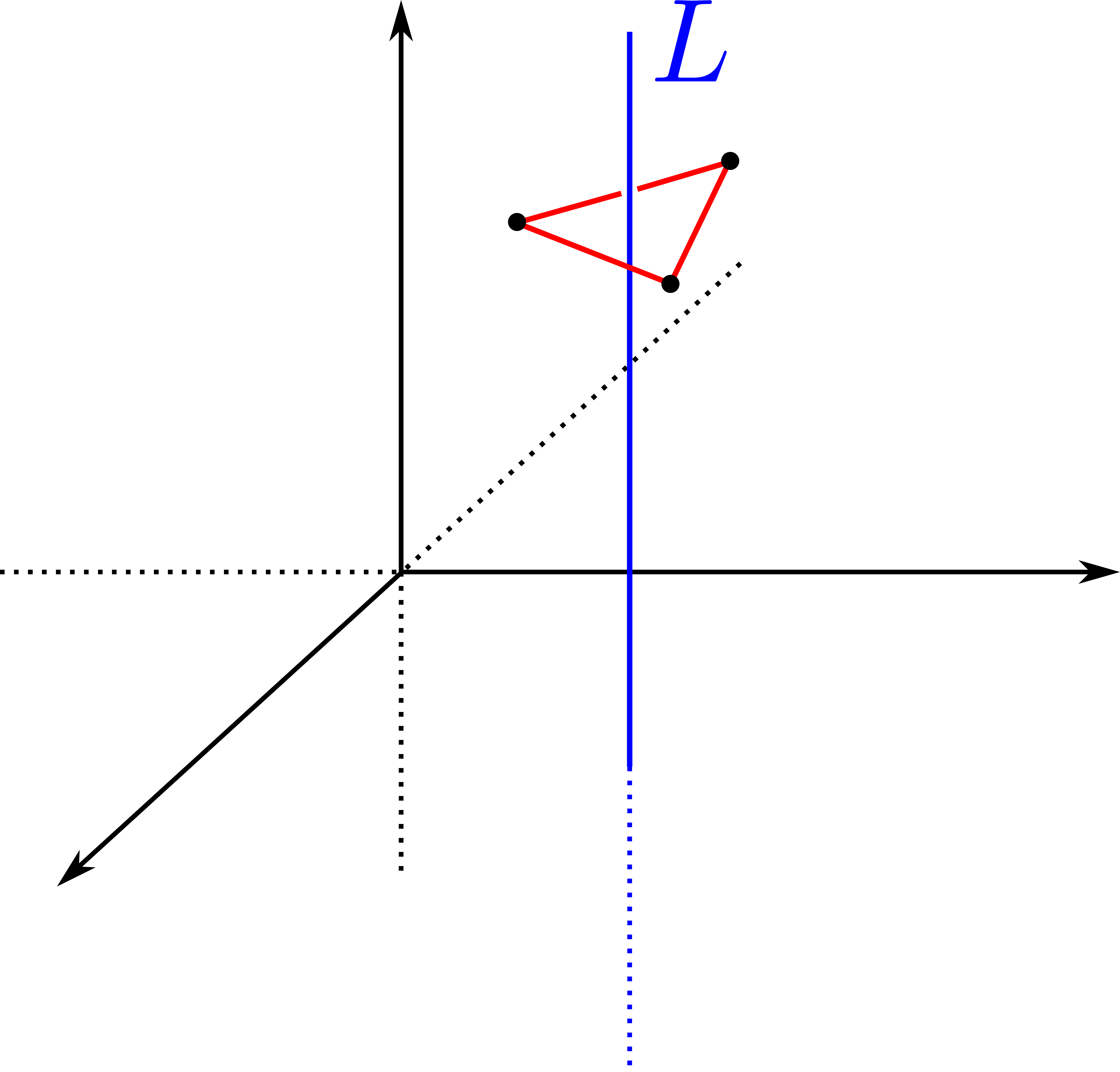}
		\centering
		\caption{Example of a thin set $L\subset\R^3$. Any point in $L$ is a convex combination of points in the complement. One such example is shown in the image, where three black points in $\R^3\setminus L$ convexly generate a point in $L$ (surrounded by red lines). \label{fig:thinSet}}
\end{figure}

Not all ample subsets have thin complements. The following example shows an ample subset whose complement has codimension one:
\begin{example} \label{ex:symmetricMatrices}
The subset of $\R^3$ defined by
\[\SH^- := \lbrace (x,y,z)\in\R^3 \,\mid\, xy-z^2 < 0 \rbrace\]
is the outer-component of a cone. It is ample and thus any point $p\in\R^3$ can be expressed as a convex combination of points in $\SH^-$. The remaining part of the complement of the cone
\[\SH^+ := \lbrace (x,y,z)\in\R^3 \,\mid\, xy-z^2 > 0 \rbrace\]
consists of two components, neither of which is ample.
\end{example}
The set $\SH^-$ will reappear later on in our study of hyperbolic $(4,6)$ distributions. $\SH^+$ appears in the study of elliptic $(4,6)$ distributions. See Section \ref{sec:46}.

\subsection{Ampleness in all principal directions} \label{ssec:ampleness1}

We now define the most commonly used notion of ampleness for differential relations. It is also the most restrictive one.
\begin{definition} \label{def:ampleness1}
Fix a bundle $X \to M$ and an over-relation $\iota_\SR: \SR \to J^r(X)$. Let $\lambda \in T_p^*M$ be a covector. We say that $\iota_\SR$ is
\begin{itemize}
\item \textbf{ample along the principal direction determined by $\lambda$} if, for every $F \in \SR$ projecting to $p$, the map $\iota_\SR: \SR_{\lambda,F} \to \Pr_{\lambda,F}$ is ample.
\item \textbf{ample in all principal directions} if the over-relations $(\pi^r_{r'} \circ \iota_\SR)_{r' = 1,\cdots,r}$ are ample along all non-zero covectors $\lambda$.
\end{itemize}
\end{definition}
Being the most commonly used flavour, we sometimes just say that $\iota_\SR$ is \textbf{ample}. Gromov's convex integration is usually stated as:
\begin{theorem} \label{thm:convexIntegration1}
The complete $C^0$-close $h$-principle holds for any open over-relation that is ample in all principal directions.
\end{theorem}
This result was first proven, only for first order, in \cite[Corollary 1.3.2]{Gr73}. The statement for all orders appeared later in \cite[Section 2.4, p. 180]{Gr86}. The first order case is treated as well in \cite[Part 4]{EM}.

\subsubsection{The foliated setting} \label{sssec:foliatedAmpleness1}

Fix a bundle $Y \to (N,\SF)$ and a foliated over-relation $\iota_\SS: \SS \to J^r(Y;\SF)$. We say that $\SS$ is \textbf{ample along all foliated principal directions} if, for each leaf $L$, the restriction $\SS|_L$ satisfies Definition \ref{def:ampleness1}.

By construction, the ampleness of the non-foliated preimage $\SS^* \to J^r(Y)$ of $\SS$ can be read purely along $\SF$:
\begin{lemma} \label{lem:reductionToFoliation}
Fix a leaf $L$, a point $p \in L$, a formal datum $z \in J^r(Y)|_p$, and a codirection $\lambda \in T_p^*N$. Write $z' \in J^r(Y|_L)$ for the leafwise jet of $z$ and $\lambda'$ for the restriction $\lambda|_L \in T_p^*L$.

The following conditions are equivalent:
\begin{itemize}
\item $\SS^*$ is ample along the principal subspace $\Pr_{\lambda,z} \subset J^r(Y)$.
\item $\SS|_L$ is ample along the principal subspace $\Pr_{\lambda',z'} \subset J^r(Y|_L)$.
\end{itemize}
\end{lemma}

It immediately follows that:
\begin{corollary} \label{cor:convexIntegration1}
The complete $C^0$-close $h$-principle holds for any open, foliated over-relation that is ample in all foliated principal directions.
\end{corollary}

A particularly beautiful consequence of these statements is the following. Suppose $K$ is a compact manifold, $X \to M$ is a bundle, and $\iota_\SR: \SR \to J^r(X)$ is an open over-relation that is ample along all principal directions. Then, the parametric lift $\SR \times K$ is, by definition, ample along all foliated principal directions. Applying Lemma \ref{lem:reductionToFoliation} we deduce that $(\SR \times K)^* \to J^r(X \times K)$ is ample along all principal directions. It follows that, in order to prove Theorem \ref{thm:convexIntegration1} for arbitrary parameters (and relatively in parameter and domain), it is sufficient to prove the non-parametric version (relatively in domain). Indeed: the $K$-parametric statement for $\SR$ is just the $0$-parametric statement for $\SR \times K$.

\subsection{Ampleness along principal frames} \label{ssec:ampleness2}

As we pointed out in the Introduction, we do not need ampleness in all directions, since it is sufficient to be able to proceed over a base in the space of derivatives. This motivates the following definitions.
\begin{definition} \label{def:hyperplaneField}
A \textbf{locally-defined hyperplane field} is a pair $(U,\tau)$ consisting of an open set $U \subset M$ and a germ of hyperplane field $\tau$ along the closure $\bar{U}$.

The hyperplane field $(U,\tau)$ is \textbf{integrable} if $\tau$ integrates to a codimension-$1$ foliation.
\end{definition}
Our hyperplane fields will live on charts and therefore they will always be locally-defined. The condition that $\tau$ is a germ along the closure $\bar{U}$ is included to make some of our later statements cleaner. The reader can think of $\bar{U}$ as being some closed ball in $M$. Often, we just write $\tau$ and we leave $U$ implicit; we say that $U$ is the \emph{support} of $\tau$. 

\begin{definition} \label{def:principalCover}
A \textbf{principal frame} of order $r$ is a collection $C$ of locally-defined hyperplane fields satisfying:
\begin{itemize}
\item All of the fields in $C$ are integrable and have the same support $U$.
\item $C$ is a principal basis in each of the fibres of $\pi^r_{r-1}$ lying over $U$.
\end{itemize}
A \textbf{principal cover} of order $r$ is a collection $\SC$ of principal frames of order $r$ whose supports cover $M$.

A principal direction/subspace defined by a hyperplane in $\SC$/$C$ will be called a $\SC$/$C$-principal direction/subspace.
\end{definition}

The second flavour of ampleness reads:
\begin{definition} \label{def:ampleness2}
An over-relation $\iota_\SR: \SR \to J^r(X)$ is \textbf{ample along principal frames in order $r$} if each point $p \in M$ admits an $r$-order principal frame $C$ with support $U \ni p$ such that $\SR$ is ample along all $C$-principal directions.

The over-relation $\iota_\SR$ is \textbf{ample along the principal cover} $\SC$ if $\SR$ is ample along all $\SC$-principal directions.

The over-relation $\iota_\SR$ is \textbf{ample along principal frames} if the relations $(\pi^r_{r'} \circ \iota_\SR)_{r'=1,\cdots,r}$ are ample along principal frames in their respective order.
\end{definition}

If $r=1$ (or if $r>1$ but $\dim(M)=1$), a principal cover can be obtained by picking a covering of $M$ and in each chart setting $\{\tau_i = \ker(dx_i)\}$, where $\{dx_i\}$ is the coordinate coframe. For this reason, when one deals with first order jets, ampleness along a principal cover is also called \emph{ampleness along coordinate directions}; see \cite[Definition 1.2.6]{Gr73} and \cite[p. 167]{EM}.

\begin{theorem} \label{thm:convexIntegration2}
The complete $C^0$-close $h$-principle holds for any open over-relation that is ample along principal frames.
\end{theorem}
For first order, this is the main result in \cite[Theorem 1.3.1]{Gr73}; it appears in \cite[p. 172]{EM} as well. For arbitrary order, it follows from \cite[p. 179, Principal Stability Theorem C]{Gr86}. An alternate implementation for arbitrary order appeared in the master thesis \cite{Den}; it avoids convex hull extensions and adapts instead the idea from \cite{Gr73}.

\subsubsection{The foliated setting}

Consider a bundle $Y \to (N,\SF)$ and a foliated over-relation $\SS \to J^r(Y;\SF)$. It is possible to adapt Definition \ref{def:ampleness2} to the foliated setting by relying on principal covers that consist of leafwise hyperplane fields. This is ultimately unnecessary for us, so we leave the details to the reader. The same comment applies to the next section.

\subsection{Ampleness in the sense of convex hull extensions} \label{ssec:ampleness3}

If an (over)-relation is ample along principal frames, all formal solutions can be made holonomic, one derivative at a time, by introducing oscillations along the codirections given by the frames. However, one can imagine a situation where different formal solutions need oscillations along different principal frames or even oscillations along collections of codirections that do not form a frame at all.

In order to formalise this idea, we introduce the concept of convex hull extensions\footnote{Our definition differs slightly from the one in \cite[p. 177]{Gr86}. The reason is that we want our open over-relations to be manifolds that submerse onto jet space (instead of more general microfibrations with domain a (quasi)topological space). Assuming $\SR$ itself was a manifold, Gromov's convex hull extension would still yield instead a topological space with conical singularities. The upcoming convex integration statements are unaffected by this change.}:
\begin{definition} \label{def:convexHullExtension}
Let $\iota_\SR: \SR \to J^r(X)$ be an over-relation. Its \textbf{convex hull extension} is the set
\[ \Conv(\SR) \,:=\, \{(F,\lambda,z) \in \SR \times_M T^*M \times_M J^r(X) \,\mid\, z \in \Conv(\SR_{\lambda,F},F)\} \]
with anchor map $(F,\lambda,z) \mapsto z$.
\end{definition}
We observe:
\begin{lemma}
Suppose $\SR$ is open. Then, $\Conv(\SR)$ is an open over-relation. In particular, its underlying space is a smooth manifold.
\end{lemma}
\begin{proof}
Let $W \to J^r(X)$ be the pullback of the vertical tangent space of $X$; i.e. the subspace of $TX$ consisting of vectors tangent to the fibres of $X \to M$. Using Lemma \ref{lem:principalImage} we observe that the space
\[ A = \{(F,\lambda,z) \in \SR \times_M T^*M \times_M J^r(X) \,\mid\, z \in \Pr_{\lambda,F}\} \]
is a smooth fibre bundle over $\SR \times_M T_p^*M$ with affine fibre isomorphic to $W$. Using the Lemma once more, we see that the anchor map $A \to J^r(X)$ given by $(F,\lambda,z) \mapsto z$ can equivalently be written as 
\[ (F,\lambda,z = \iota_\SR(F) + \lambda^{\otimes r} \otimes w) \,\mapsto\, \iota_\SR(F) + \lambda^{\otimes r} \otimes w, \]
which is a submersion because $\SR$ itself was. The proof is complete noting that $\Conv(\SR)$ is an open subset of $A$, due to the openness of $\SR$.
\end{proof}

We write $\Conv^l(\SR)$ for the $l$-fold convex hull extension of $\SR$. An element in $\Conv^l(\SR)$ is then an element $F \in \SR$, together with a principal path of length $l$ starting at $\iota_\SR(F)$. A section of $\Conv^l(\SR)$ is thus a smoothly-varying choice of principal path at each point. Do note that the hyperplanes associated to such paths vary smoothly, but need not be integrable; for this reason, it is convenient to restrict our attention to the following nice subclass of sections:
\begin{definition}
A section $(F,\lambda_1,z_1,\cdots,\lambda_l,z_l): M \to \Conv^l(\SR)$ is said to be \textbf{integrable} if the $\lambda_i$ are integrable.
\end{definition}
Do note that the $\lambda_i$ are allowed to vanish and thus we speak of integrability in the locus $\{\lambda_i \neq 0\}$.

The following definition corresponds to the idea of being able to connect, using convex hull extensions, a formal datum $F$ to the holonomic section corresponding to its zero order part $\pi^r_0 \circ \iota_\SR \circ F$.
\begin{definition}
A formal solution $F: M \to \SR$ is (integrably) \textbf{short} if, for some $l$, there is a (integrable) holonomic solution $G: M \to \Conv^l(\SR)$ of the form $(F,\cdots, j^r(\pi^r_0 \circ \iota_\SR \circ F))$.

Assume that $F$ is holonomic in a neighbourhood of a closed subset $M' \subset M$. Then, $F$ is \textbf{short relative to $M'$} if the codirections $\{\lambda_i\}_{i=1}^l$ associated to $G = (F,\lambda_1,z_1,\cdots,\lambda_l,z_l)$ can be chosen to be zero over $\Op(M')$.
\end{definition}
Do note that if a solution is short it is already a solution of $\pi^r_{r-1} \circ \iota_\SR$. That we need such an assumption is not surprising, since the convex hull extension machinery works purely in order $r$. As such, in order to provide a full flexibility statement, we need to consider convex hull extensions for each $\pi^r_{r'} \circ \iota_\SR$, $r'=1,\cdots,r$.

The last flavour of ampleness comes in two slightly different incarnations:
\begin{definition} \label{def:ampleness3}
An over-relation $\iota_\SR: \SR \to J^r(X)$ is \textbf{ample in the sense of (integrable) convex-hull extensions} if the following property holds: Fix
\begin{itemize}
\item An order $r'=1,\cdots,r$,
\item A compact manifold $K$,
\item A $K$-family of formal solutions $F: M \times K \to \SR \times K$ that is holonomic of order $r'-1$.
\end{itemize}
Then, the family $F$ is (integrably) short for $\pi^r_{r'} \circ \iota_\SR$, relative to the regions in which it is already $r'$-holonomic.
\end{definition}

Convex integration, in full generality, reads:
\begin{theorem} \label{thm:convexIntegration3}
The complete $C^0$-close $h$-principle holds for any open over-relation that is ample in the sense of (integrable) convex hull extensions.
\end{theorem}
This result, assuming integrability, is presented in detail in \cite[p. 123, Theorem 8.4]{Spring}. The statement, without integrability, was already implicit in \cite[p. 179, Principal Stability Theorem C]{Gr86}. The integrability hypothesis is restrictive; we discuss it further in the next Subsection.

\subsection{A comparison of the different incarnations of ampleness} \label{ssec:comparison}

As stated earlier, classic ampleness (ampleness in all principal directions) is the most restrictive of the notions we have introduced. Indeed, it is immediate that Theorem \ref{thm:convexIntegration2} implies Theorem \ref{thm:convexIntegration1}. Furthermore, Theorem \ref{thm:convexIntegration3} implies both: Ampleness along a frame says that we can connect any formal solution $F$ to $j^r(\pi^r_0 \circ \iota_\SR \circ F)$ using the given principal frames, proving integrable shortness of $F$. This works for families and relatively as well.

It is obvious that ampleness in the sense of integrable convex hull extensions is more restrictive than the version without integrability. In particular, ampleness in the sense of convex hull extensions is the most general of the four definitions given.

As we wrote above, Theorem \ref{thm:convexIntegration3}, without integrability, is contained in Gromov's text implicitly; it can be deduced from \cite[p. 179, Principal Stability Theorem C]{Gr86}. In \cite{Spring}, Spring works always under integrability assumptions; this allows him to directly invoke $1$-dimensional convex integration in the foliation charts associated to the integrable hyperplane fields appearing in the definition of integrable shortness.

The key claim that Gromov uses to drop integrability, see \cite[p. 177]{Gr86}, is that any \emph{continuous} hyperplane field can be piecewise approximated by foliation charts. This can then be used to approximate any section of $\Conv^l(\SR)$ by an integrable one (at the expense of increasing $l$). We interpret this as an \emph{$h$-principle without homotopical assumptions} (see also Remark \ref{rem:hPrincipleWOHA}) saying that there is a weak equivalence between the space of sections and the subspace of integrable ones. The authors of the present paper have not checked this claim in detail. In fact, it is not important for our results:
\begin{remark} \label{rem:ampleness3Integrability}
None of the results from this paper rely on Definition \ref{def:ampleness3} or Theorem \ref{thm:convexIntegration3}. Our arguments reduce the $h$-principle for relations that are ample up to avoidance to the $h$-principle for relations that are ample along a principal frame (Theorem \ref{thm:convexIntegration2}). Nonetheless, in Corollary \ref{cor:ampleness3Integrability} we prove that a relation that is ample up to avoidance is ample in the sense of integrable convex hull extensions.
\end{remark}

\subsubsection{Computability of ampleness in all principal directions}

Ampleness in all principal directions is the most restrictive but easiest to check of the four incarnations. The reason is that it is \emph{pointwise} in nature: we just go through each fibre of $J^r(X)$ checking ampleness, one principal direction at a time. In practice, one often deals with Diff-invariant relations described as the complement of some fibrewise (semi-)algebraic condition (which we call the \emph{singularity}). It is then sufficient to check a single fibre of $J^r(X)$ and a single codirection; the problem boils down then to checking the intersection of the singularity with each principal subspace. In practice, this can be already quite involved\footnote{In \cite{MT}, P. Massot and M. Theillière prove that convex integration can be used to prove holonomic approximation in spaces of 1-jets. This is a beautiful application of classic convex integration in which checking ampleness is highly non-trivial.} unless the relation is relatively simple. 

Classic ampleness turns out to be limited in its applications. In practice, we only encounter it if all formal solutions $F \in \SR$ present some form of symmetry guaranteeing that they sit equally nicely with respect to all codirections $\lambda \in T^*M$; we will see this in examples in Section \ref{sec:flexBG}. The relation defining hyperbolic $(4,6)$ distributions, despite being Diff-invariant, does not satisfy this. We expect most differential relations of codimension-$1$ not to satisfy it.

\subsubsection{Computability of ampleness along principal frames}

Ampleness along principal frames turns out to be not so different from ampleness in all directions. The two are equivalent if we assume Diff-invariance. In terms of computability, once a concrete principal cover $\SC$ is given, checking $\SC$-ampleness is, by definition, easier than checking it in all directions. 

\subsubsection{Computability of ampleness in the sense of convex hull extensions}

Ampleness in the sense of convex hull extensions is incredibly general, but notoriously difficult to check. The reason is that it is not a pointwise condition: A formal solution $F$ is short if we can connect it, using a smooth family of principal paths $G$, to its underlying holonomic section; both $F$ and $G$ are global objects.

Suppose we want to construct $G$ and thus prove that $F$ is short. Convex integration is local in nature, so we try to find a suitable cover of $M$ to proceed. Given a point $p \in M$, we may be able to define $G(p)$ and then extend it locally, by openess, to some open neighbourhood $U_p$. Finding $G(p)$ is a pointwise process and does not need ampleness in all principal \emph{directions}; it is sufficient to find a suitable sequence of ample principal \emph{subspaces} starting from $F$. We do this for all $p$ and we extract a cover $\{U_i\}$ with a section $G_i$ defined over $U_i$. In order to patch these up, we start with $G_1$ and we glue it with $F$ using a cut-off close to $\partial U_1$. The problem now is that the principal subspaces that behaved nicely with respect to $F$ need not behave nicely with respect to $G_1$. In particular, $G_2$ may not help us at all in the overlap $U_1 \cap U_2$.

Furthermore, unlike the previous two flavours, ampleness in the sense of convex hull extensions is not readily parametric. To deal with families one has to prove that the family in question, as a whole, is short.

\emph{Ampleness up to avoidance} is designed to deal with these considerations and make the aforementioned sketch of argument work. It is also computable pointwise, as we explain in Subsection \ref{sssec:computabilityAvoidance}.

\section{Avoidance} \label{sec:avoidance}

Let $\iota_\SR: \SR \to J^r(X)$ be an open over-relation. Our goal is to construct a so-called \emph{avoidance template} $\SA$ associated to $\SR$; if we succeed in constructing $\SA$, we will say that $\SR$ is \emph{ample up to avoidance}. Our main Theorem \ref{thm:main}, whose proof we postpone to the next Section, says that this is a sufficient condition for the $h$-principle to hold. 

Templates (and more general objects called pre-templates) are introduced in Subsections \ref{ssec:avoidanceTemplates} and \ref{ssec:liftingTemplate}. These definitions require us to introduce some auxiliary notation about configurations of hyperplanes; this is done in Subsection \ref{ssec:hyperplanes}. In Subsection \ref{ssec:removing} we present some simple constructions of pre-templates. These constructions can yield empty pre-templates when $\SR$ is very far from being ample; this is explained in Subsection \ref{ssec:trivialPretemplates}.

\subsection{Configurations of hyperplanes} \label{ssec:hyperplanes}

Given a positive integer $a$ and a vector space $V$, we write
\[ \HConf_a(V) := \{ (H_1,\cdots,H_a) \in (\NP V^*)^a \,\mid\, H_i \neq H_j, \text{ for all $i \neq j$}\}/\Sigma_a. \]
I.e. the smooth, non-compact manifold consisting of all unordered configurations $[H_1,\cdots,H_a]$ of $a$ distinct hyperplanes in $V$. Its non-compactness is due to collisions (i.e. any sequence in which $H_i$ approaches $H_j$ has no convergent subsequence). In order to consider collections of arbitrary finite cardinality, we consider the union:
\[  \HConf(V) := \coprod_{a=0}^\infty \HConf_a(V), \]
where $\HConf_0(V) := \{\emptyset\}$ is the space containing only the empty configuration.

Given two configurations $\Xi, \Xi' \in \HConf(V)$ we will write $\Xi' \subset \Xi$ if every hyperplane in the former is contained in the latter.

\subsubsection{Repetitions}

In practice, we will deal with ordered collections of hyperplanes that may have repetitions. Concretely, these correspond to points in the closed manifold
\[  \barHConf[a](V) := (\NP V^*)^a. \]
Consider the open dense subset $\sbarHConf[a](V) \subset \barHConf[a](V)$ consisting of those collections with no repetitions. Its complement is an algebraic subvariety. By construction, we have a quotient map
\[ \pi: \sbarHConf[a](V) \longrightarrow \HConf_a(V) \]
whose fibres are isomorphic to the symmetric group $\Sigma_a$. As before, we write 
\[ \barHConf(V)  := \coprod_{a=0}^\infty \barHConf[a](V), \qquad 
   \sbarHConf(V) := \coprod_{a=0}^\infty \sbarHConf[a](V), \]
where $\barHConf[0](V)$ and $\sbarHConf[0](V)$ are the singleton set $\{\emptyset\}$.

\subsubsection{Bundles of configurations}

Fix a manifold $M$. We write 
\[ \HConf(TM) \to M \]
for the smooth fibre bundle with fibre $\HConf(T_pM)$ at a given $p \in M$. Similarly, we write $\barHConf(TM)$ and $\sbarHConf(TM)$. By construction, we have a quotient map 
\[ \sbarHConf(TM) \longrightarrow \HConf(TM) \]
given by the fibrewise action of the symmetric groups.

\subsection{Avoidance templates and ampleness} \label{ssec:avoidanceTemplates}

Fix a bundle $X \to M$, an over-relation $\iota_\SR: \SR \to J^r(X)$, and a subset $\SA$ of the fibered product $\SR \times_M \HConf(TM)$.

Given a family of hyperplanes $\Xi \in \HConf(T_pM)$, we write
\[ \SA(\Xi) := \SA \cap (\SR \times_M \{\Xi\}). \]
Using the canonical identification $\SR \times_M \{\Xi\} \cong \SR_p$, we regard $\SA(\Xi)$ as a subset of the fibre $\SR_p$ lying over $p \in M$. If we are given a collection of hyperplane fields $\Xi: M \to \HConf(TM)$ instead, we will similarly write $\SA(\Xi)$ for the union of all the subsets $\SA(\Xi(p))$ as $p$ ranges over the entirety of $M$. In this case, $\SA(\Xi)$ is a subset of $\SR$. If it is a smooth submanifold, the map $\iota_\SR: \SA(\Xi) \to J^r(E)$ is an over-relation.

Given some $F \in \SR$ lying over a point $p$, we similarly denote
\[ \SA(F) := \SA \cap \left(\{F\} \times \HConf(T_pM) \right). \]
As before, we regard $\SA(F)$ as the subset of $\HConf(T_pM)$ consisting of those $\Xi$ such that $F \in \SA(\Xi)$. If $F$ is instead a section $M \to \SR$, $\SA(F)$ will be the subset of $\HConf(TM)$ given by the union of all $\SA(F(p))$, as $p$ ranges over all points in $M$.

\begin{definition} \label{def:avoidanceTemplate}
An open subset $\SA \subset \SR \times_M \HConf(TM)$ is an \textbf{(avoidance) pre-template} if the following property holds:
\begin{itemize}
\item[I.] If $\Xi' \subset \Xi \in \HConf(TM)$ is a subconfiguration, then $\SA(\Xi) \subset \SA(\Xi')$.
\end{itemize}

The pre-template $\SA$ is an \textbf{(avoidance) template} if, additionally:
\begin{itemize}
\item[II.] Given $\Xi \in \HConf(TM)$, $\SA(\Xi)$ is ample along the principal directions determined by $\Xi$.
\item[III.] Given $F \in \SR$ lying over $p \in M$, $\SA(F)$ is dense in each $\HConf_m(T_pM)$.
\end{itemize}
\end{definition}
Property (I) guarantees coherence: removing hyperplanes from $\Xi$ makes the relation $\SA(\Xi)$ bigger. In particular, if $\SA(\Xi)$ is ample along $\Xi$, then $\SA(\Xi')$ is ample along $\Xi' \subset \Xi$. 

Our main definition reads:
\begin{definition} \label{def:main}
An open over-relation $\iota_\SR: \SR \to J^r(X)$ is said to be \textbf{ample up to avoidance} if each of the over-relations 
\[ (\pi^r_{r'} \circ \iota_\SR:\, \SR \to J^{r'}(X))_{r'=1,\cdots,r} \]
admits an avoidance template.
\end{definition}
Observe that $\SR \times_M \HConf(TM)$ is an avoidance template if and only if $\SR$ is ample in all principal directions.

\subsubsection{Computability of avoidance} \label{sssec:computabilityAvoidance}

We stated in Subsection \ref{ssec:comparison} that ampleness up to avoidance is as computable as classic ampleness. There are two parts to this claim.

First we note that verifying whether a given open subset $\SA \subset \SR \times_M \HConf(TM)$ is a template boils down to pointwise checks. Property (I) is often given by construction. Property (III) is often checked together with openness and follows as soon as the complement of $\SA(F)$ is given, fibrewise, by some algebraic equality. Property (II) is the most involved, but it is no different from checking ampleness along a principal frame.

The second part of the claim is that the construction of templates is algorithmic. Indeed, we present two possible constructions in Subsection \ref{ssec:removing}. However, the reader should just think of these as rough guidelines. In practice (for instance, in the proof of Theorem \ref{thm:46}), one needs to make adjustments in order to produce a template. Still, the adjustments that need to be made are somewhat standard; see Remark \ref{rem:thinSingularitiesAreGood}.

Lastly, we observe that ampleness up to avoidance is parametric in nature, much like classic convex integration. Namely, given a template $\SA$ for $\SR$, we can define an associated foliated template for any parametric lift $\SR \times K$; see Subsection \ref{ssec:foliatedTemplates}. The parametric version of Theorem \ref{thm:main} will follow then from the non-parametric one.

\subsection{Lifted avoidance templates} \label{ssec:liftingTemplate}

Definition \ref{def:avoidanceTemplate} is intuitive conceptually but, in practice (see the proofs of Propositions \ref{prop:avoidanceRelation} and \ref{prop:main}), it is often more convenient to deal with the following notion:
\begin{definition} \label{def:liftingTemplate}
Let $\pi$ be the quotient map $\sbarHConf(TM) \to \HConf(TM)$. We write
\[ \bar\SA \subset \SR \times_M \sbarHConf(TM) \subset \SR \times_M \barHConf(TM) \]
for the preimage of a given subset
\[ \SA \subset \SR \times_M \HConf(TM). \]
\end{definition}
Given $\Xi \in \barHConf(TM)$, we write $\SA(\Xi) := \SA(\pi(\Xi))$. Similarly, given $F \in \SR$, we write $\barSA(F)$ for the preimage by $\pi$ of $\SA(F)$.

We remark:
\begin{lemma} \label{lem:liftingTemplate}
Fix a subset $\SA \subset \SR \times_M \HConf(TM)$. Then, $\SA$ is a pre-template if and only if
\begin{itemize}
\item $\barSA$ is open.
\item $\barSA$ is invariant under the action of the permutation groups $\Sigma_*$.
\item[$\overline{\text{I}}$.] Consider $\Xi', \Xi \in \barHConf(TM)$. Suppose $\pi(\Xi')$ is a subconfiguration of $\pi(\Xi)$. Then $\barSA(\Xi) \subset \barSA(\Xi')$.
\end{itemize}

Furthermore, $\SA$ is a template if and only if, additionally:
\begin{itemize}
\item[$\overline{\text{II}}$.] Given $\Xi \in \HConf(TM)$, $\barSA(\Xi)$ is ample along the principal directions determined by $\Xi$.
\item[$\overline{\text{III}}$.] Given $F \in \SR$ lying over $p \in M$, $\barSA(F)$ is dense in each $\barHConf[m](T_pM)$.
\end{itemize}
\end{lemma}
\begin{proof}
First note that $\sbarHConf(TM) \subset \barHConf(TM)$ is open. Its complement, which is an algebraic variety and thus of positive codimension, consists of all configurations that involve repetitions. The claim follows from this fact and the observation that $\pi$ is a quotient map.
\end{proof}
Conversely, any open, $\Sigma_*$-invariant subset of $\SR \times_M \sbarHConf(TM)$ is the $\barSA$ of some template $\SA$ as long as Properties ($\overline{\text{I}}$), ($\overline{\text{II}}$) and ($\overline{\text{III}}$) hold.

\subsection{Foliated templates} \label{ssec:foliatedTemplates}

We will prove in Section \ref{sec:mainProof} that the parametric analogue of Theorem \ref{thm:main} follows from Theorem \ref{thm:main} itself. Compare this to Theorem \ref{thm:convexIntegration1} and Corollary \ref{cor:convexIntegration1}. This is best implemented using the foliated setting, which we now introduce. 

Fix a foliated manifold $(N,\SF)$, a bundle $Y \to N$, and an over-relation $\iota_\SS: \SS \to J^r(Y;\SF)$. We look at subsets $\SA \subset \SS \times_N \HConf(\SF)$. We define $\SA(\Xi)$ and $\SA(F)$ in the obvious manner. Then:
\begin{definition} \label{def:avoidanceTemplateParametric}
An open subset $\SA \subset \SS \times_N \HConf(\SF)$ is a \textbf{foliated pre-template} if the following property holds:
\begin{itemize}
\item[I.] If $\Xi' \subset \Xi \in \HConf(\SF)$ is a subconfiguration, then $\SA(\Xi) \subset \SA(\Xi')$.
\end{itemize}

The pre-template $\SA$ is a \textbf{foliated template} if, additionally:
\begin{itemize}
\item[II.] Given $\Xi \in \HConf(\SF)$, $\SA(\Xi)$ is ample along the principal directions determined by $\Xi$.
\item[III.] Given $F \in \SS$ lying over $p \in N$, $\SA(F)$ is dense in each $\HConf_m(\SF_p)$.
\end{itemize}
\end{definition}

The following observation follows immediately from the leafwise nature of Definition \ref{def:avoidanceTemplateParametric}:
\begin{lemma} \label{lem:liftParametric}
Let $X \to M$ be a bundle and $\iota_\SR: \SR \to J^r(X)$ an over-relation. Fix a compact manifold $K$. Suppose $\SR$ admits a template $\SA$. Then the parametric lift $\SR \times K$ admits a foliated template $\SA \times K$.
\end{lemma}

Furthermore:
\begin{lemma} \label{lem:liftFoliated}
Let $Y \to (N,\SF)$ be a bundle over a foliated manifold. If an over-relation $\SS \to J^r(Y;\SF)$ admits a foliated template $\SA$, its non-foliated preimage $\SS^* \to J^r(Y)$ admits a template $\SA^*$.
\end{lemma}
\begin{proof}
We define $\SA^*$ as a subset of $\SS^* \times_N \HConf(TN)$. Consider the subspace $\HConf'(TN)$ of $\HConf(TN)$ consisting of those configurations $[H_1,\cdots,H_a] \in \HConf(TN)$ that satisfy:
\begin{itemize}
\item All $H_i \in [H_1,\cdots,H_a] $ intersect $\SF$ transversely.
\item For all $i \neq j$, the intersections $H_i \cap \SF$ and $H_j \cap \SF$ are distinct.
\end{itemize}
Then, the intersection with $\SF$ defines a surjection $\HConf'(TN) \to \HConf(\SF)$ which can easily be shown to be submersive. In fact, it is a proper map with compact fibres isomorphic to a product of projective spaces, showing that
\[ \pi:\, \SS^* \times_N \HConf'(TN) \,\longrightarrow\, \SS \times_N \HConf(\SF), \]
is a fibration. This allows us to define
\[ \SA^* := \, \pi^{-1}(\SA) \subset \SS^* \times_N \HConf'(TN) \subset \SS^* \times_N \HConf(TN). \]

The openness of $\SA^*$, as well as Properties (I), (II), and (III), follow from the analogous properties for $\SA$. Concretely: Property (I) follows from $\pi: \SA^* \to \SA$ being a fibration. Openess and Property (III) are a consequence of the fact that $\HConf'(TN) \subset \HConf(TN)$ is open and its fibrewise complement is an algebraic subvariety (and thus of positive codimension). Property (II) follows from Lemma \ref{lem:reductionToFoliation}.
\end{proof}

\subsection{Removing processes} \label{ssec:removing}

The most straightforward way of producing templates consists of iteratively removing those principal subspaces along which the relation is not ample.
\begin{definition}
Let $\iota_\SR: \SR \to J^r(X)$ be an over-relation. We set:
\[ \Avoid^0(\SR) := \SR \times_M \HConf(TM). \]
Inductively, we define $\Avoid^{l+1}(\SR)$ to be the complement in $\Avoid^l(\SR)$ of the closure of
\[ \{(F,\Xi)  \,\mid\, \text{for some $\tau \in \Xi$, the component of $F$ in $\Avoid^l(\SR)(\Xi)_{\tau,F}$ is not ample} \}. \]
\end{definition}
Do note that, crucially, $\Avoid^1(\SR)$ need not be a template. Indeed, upon removing elements from $\Avoid^0(\SR)$, we may have lost ampleness along subspaces that were not problematic previously. This justifies the necessity of iterating the construction.

\begin{definition} \label{def:removing}
Suppose that the process just described terminates, meaning that there is a step $l_0$ such that
\[ \Avoid^l(\SR) = \Avoid^{l_0}(\SR) \qquad \text{for every $l \geq l_0$}. \]
Then, $\Avoid^\infty(\SR) := \Avoid^{l_0}(\SR)$ is the \textbf{standard pre-template} associated to $\iota_\SR$.
\end{definition}

By construction:
\begin{lemma}
Each $\Avoid^l(\SR)$ is a pre-template. Additionally, $\Avoid^\infty(\SR)$ satisfies Property (II) in the definition of a template. 
\end{lemma}
\begin{proof}
Openness follows from the fact that we are inductively removing closed sets. For Property (I) we reason inductively as well: By induction hypothesis, $\Avoid^l(\SR)(\Xi)$ is contained in $\Avoid^l(\SR)(\Xi')$ whenever $\Xi' \subset \Xi$. Suppose $F$ is an element of both. Then, the analogous statement for the components of $F$ in $\Avoid^l(\SR)(\Xi)_{\tau,F}$ and $\Avoid^l(\SR)(\Xi')_{\tau,F}$ is also true. In particular, if the latter is not ample, neither is the former. I.e. if $(F,\Xi')$ is removed, so is $(F,\Xi)$, proving the claim.

The second statement follows by definition of the removal process terminating.
\end{proof}
As we will observe in examples, $\Avoid^\infty(\SR)$ need not satisfy Property (III); whether it does needs to be checked in each concrete application.

\subsubsection{Thinning} \label{sssec:thinning}

In applications, the following more restrictive notion can also be useful.
\begin{definition}
Let $\iota_\SR: \SR \to J^r(X)$ be an over-relation. We write $\Thin(\SR)$ for the complement in $\SR \times_M \HConf(TM)$ of the closure of
\[ \{(F,\Xi) \,\mid\, \text{for some hyperplane $\tau \in \Xi$, the complement of $\SR_{\tau,F}$ is not thin} \}. \]
\end{definition}
We denote the $l$-fold iterate of this construction by $\Thin^l(\SR)$.
\begin{definition} \label{def:thinning}
Assuming that there is a step $l_0$ in which this process stabilises, we say that $\Thin^\infty(\SR) := \Thin^{l_0}(\SR)$ is the \textbf{thinning pre-template} of $\SR$.
\end{definition}

Much like earlier:
\begin{lemma}
$\Thin^l(\SR)$ is a pre-template. Additionally, $\Thin^\infty(\SR)$ satisfies Property (II) in the definition of a template. 
\end{lemma}
One can also conceive removing pieces from $\SR$ using schemes different from those presented in Definitions \ref{def:removing} and \ref{def:thinning}. In fact, this will be necessary for our main application Theorem \ref{thm:46}; see Section \ref{sec:46}.

\subsection{Trivial pre-templates} \label{ssec:trivialPretemplates}

It is unclear to the authors whether the standard avoidance/thinning processes always terminate regardless of what $\SR$ is. One could imagine a situation where we keep removing pieces from $\SR$ but never stabilise. Furthermore, even if they terminate, they may produce uninteresting results. This is not surprising, as many relations are simply not ample up to avoidance:
\begin{lemma} \label{lem:emptyAvoid}
Fix a fibre bundle $X \to M$ and a differential relation $\SR \subset J^r(X)$. Assume that:
\begin{itemize}
\item Each $\SR_{\tau,z}$ is trivially ample or all its components are non-ample.
\item Each fibre of $\pi^r_{r-1}$ contains an element not in $\SR$.
\end{itemize}
Then, $\Avoid^\infty(\SR)(\Xi) = \emptyset$, where $\Xi$ is any tuple that includes a principal basis.
\end{lemma}
Do note that $\Avoid^\infty(\SR)$ is only interesting for those $\Xi$ that include a basis. Otherwise there are not enough directions to span the complete fibre of $\pi^r_{r-1}$.

\begin{proof}
Consider $\Xi \in \HConf(T_pM)$ including a principal basis. We work over a fixed fiber of $\pi^r_{r-1}$ lying over $p$. The Lemma follows as a consequence of the following inductive claim:
\begin{itemize}
\item Let $z_k$ differ from some $z_0 \notin \SR$ by a $\Xi$-principal path of length $k$. Then, it follows that $z_k \notin \Avoid^k(\SR)(\Xi)$.
\end{itemize}
The base case $k=0$ is definitionally true.

Consider the inductive step $k$. Given $z_k$, there is some $z_{k-1}$ such that $\tau := z_k - z_{k-1}$ is principal and $z_{k-1}$ differs from $z_0$ by a principal path $\nu$ of length $k-1$.

Due to our assumptions on $\SR$, either $\SR_{\tau,z_0}$ is empty or its components are not ample. We can then take its complement $\SR^c_{\tau,z_0}$ and note that the shift
\[ \SR^c_{\tau,z_0} + \nu \subset \Pr_{\tau,z_{k-1}} \]
is, by inductive hypothesis, disjoint from $\Avoid^{k-1}(\SR)(\Xi)_{\tau,z_{k-1}}$. It follows that $\Avoid^{k-1}(\SR)(\Xi)_{\tau,z_{k-1}}$ is empty or its components are non-ample. Therefore, $\Avoid^k(\SR)(\Xi)_{\tau,z_{k-1}}$ is empty. In particular, $z_k$ is not in $\Avoid^k(\SR)(\Xi)$.
\end{proof}
A couple of concrete instances where Lemma \ref{lem:emptyAvoid} applies are the relation defining functions without critical points (Subsection \ref{ssec:functions}) and the relation defining contact structures (Lemma \ref{lem:contact}).

Exactly the same reasoning shows:
\begin{lemma} \label{lem:emptyThin}
Let $\SR \subset J^r(X)$ be a differential relation such that:
\begin{itemize}
\item Every $\SR_{\tau,z}$ is trivially ample or has a complement that is not thin.
\item Each fibre of $\pi^r_{r-1}$ contains an element not in $\SR$.
\end{itemize}
Then $\Thin^\infty(\SR)(\Xi) = \emptyset$, where $\Xi$ is any tuple that includes a principal basis.
\end{lemma}


\section{Proof of the main Theorem} \label{sec:mainProof}

In this Section we tackle the proof of Theorem \ref{thm:main}. We restate it now in a slightly more general form that applies to over-relations:
\begin{theorem} \label{thm:mainRestate}
Fix a smooth bundle $X \to M$ and an open over-relation $\iota_\SR: \SR \to J^r(X)$. Suppose that $\iota_\SR$ is ample up to avoidance. Then, the full $C^0$-close $h$-principle applies to $\SR$.
\end{theorem}

The proof consists of two local-to-global steps. The starting point is our assumption that a template $\SA$ exists.
\begin{itemize}
\item[1.] Given any principal cover $\SC$, we use the pointwise data given by $\SA$ to produce an over-relation $\SA(\SC) \to J^r(X)$ globally on $M$. By construction, $\SA(\SC)$ will be ample with respect to $\SC$. This is the content of Proposition \ref{prop:avoidanceRelation}.
\item[2.] Given a formal solution $F: M \to \SR$, we choose a cover $\SC$ of $M$ such that $F$ is still a formal solution of $\SA(\SC)$. This follows from a jiggling-type argument that is explained in Proposition \ref{prop:main}.
\end{itemize}
Both steps are rather discontinuous in nature. This is not surprising, since covers are discontinuous objects themselves. One of the consequences of this is that the over-relation $\SA(\SC)$ may not be a fibration (even if $\SR$ and $\SA$ were).

This sketch of argument proves that all formal solutions $F$ are short for $\SR$. From this, and the parametric nature of avoidance templates, we deduce the full $h$-principle for $\SR$. We put all these pieces together in Subsection \ref{ssec:mainProof}.

\subsection{Avoidance relations associated to principal covers} \label{ssec:avoidanceRelation}

Fix a smooth bundle $X \to M$, an open over-relation $\iota_\SR: \SR \to J^r(X)$, an avoidance template $\SA$, and a principal cover $\SC$. Since the elements of $\SC$ are defined only locally, the cardinality of $\SC$ may change from point to point. This implies that we cannot regard $\SC$ as a smooth section $M \to \barHConf(TM)$.

Nonetheless, for our purposes, the following discontinuous construction is enough. To each subset of codirections $C \subset \SC$ (not necessarily a principal frame) we associate the closed set:
\[ T_C \,:=\, \bigcap_{\tau \in C} \overline{U_\tau} \subset M, \]
where $U_\tau$ is the support of the hyperplane field $\tau$. Recall that each $\tau \in \SC$ is defined as a germ along the closure $\overline{U_\tau}$ of its support. In particular, once we pick some order for the elements of $C$, we can think of $C$ as a germ of smooth section
\[ C|_{\Op(T_C)}: \Op(T_C) \to \barHConf(TM). \]
In particular, the expression $\barSA(C)|_{\Op(T_C)}$ denotes a well-defined subset of $\SR|_{\Op(T_C)}$. Here $\barSA$ is the lift of $\SA$ to $\SR \times_M \barHConf(TM)$.

\begin{definition} \label{def:avoidanceRelation}
The \textbf{avoidance over-relation} associated to $\SA$ and $\SC$ is the set
\[ \SA(\SC) := \SR \setminus \left(\bigcup_{C \subset \SC} \barSA(C)^c|_{T_C} \right), \]
where the superscript $c$ denotes taking complement. As a subset of $\SR$, the anchor of $\SA(\SC)$ into $J^r(X)$ is $\iota_\SR$. 
\end{definition}

\subsection{Proof of Theorem \ref{thm:main}} \label{ssec:mainProof}

Before we get to the proof we introduce two key auxiliary results. The first one states that avoidance relations are open and ample:
\begin{proposition} \label{prop:avoidanceRelation}
Let $\SA$ be a template and $\SC$ be a principal cover. Then, the avoidance over-relation $\iota_\SR: \SA(\SC) \to J^r(X)$ is an open over-relation ample along $\SC$.
\end{proposition}

The second one says that we can choose avoidance relations $\SA(\SC)$ adapted to a given formal datum:
\begin{proposition} \label{prop:main}
Fix an smooth bundle $X \to M$, an open over-relation $\iota_\SR: \SR \to J^r(X)$, an avoidance template $\SA$, and a formal solution $F: M \to \SR$. Then, there is a principal cover $\SC$ such that $F$ takes values in $\SA(\SC)$.
\end{proposition}
This result is proven in Subsection \ref{ssec:jiggling}.

\begin{proof}[Proof of Proposition \ref{prop:avoidanceRelation}]
Recall the three pointwise Properties in the definition of a template (Definition \ref{def:avoidanceTemplate}).

Using the openness of $\SA$ and the closedness of each $T_C$, we see that $\SA(\SC)$ is the complement in $\SR$ of a finite union of closed subsets $\barSA(C)^c|_{T_C}$. As such, it is an open subset, and thus an open over-relation with respect to $\iota_\SR$.

Ampleness can now be checked at each point $p \in M$ individually. We note that there is a maximal subset $C \subset \SC$ such that $T_C$ contains $p$. According to coherence Property (I) in Definition \ref{def:avoidanceTemplate}, $\barSA(C(p))$ is the smallest among all sets $\barSA(C'(p))$ as $C' \subset \SC$ ranges over all the subcollections satisfying $p \in T_{C'}$. It follows that 
\[ \bigcup_{C' \subset \SC} \barSA(C'(p))^c = \barSA(C(p))^c   \]
and therefore we deduce $\SA(\SC)(p) = \barSA(C(p))$. The ampleness of the former follows then from the ampleness of the latter, which is given by Property (II).

Note that we have not made use of Property (III). It only plays a role in the proof of Proposition \ref{prop:main}.
\end{proof}

\begin{proof}[Proof of Theorems \ref{thm:main} and \ref{thm:mainRestate} assuming Proposition \ref{prop:main}]

We want to be able to homotope any given compact family of formal solutions $(F_k)_{k \in K}: M \to \SR$ to a family of genuine solutions. We regard the family as a formal solution $F: M \times K \to \SR \times K$ of the parametric lift, as in Subsection \ref{sssec:parametricLift}.

Since $\SR$ is ample up to avoidance we can apply Lemmas \ref{lem:liftParametric} and \ref{lem:liftFoliated} to deduce that its parametric lift $\SR \times K$ is also ample up to avoidance. It follows that each of the foliated relations $(\pi_{r'}^r \circ \iota_{\SR \times K})_{r'= 1,\cdots,r}$ admits an avoidance template $\SA_{r'}$.

We apply Proposition \ref{prop:main} to $F$, $\pi_1^r \circ \iota_{\SR \times K}$, and $\SA_1$ to deduce that there is a principal cover $\SC_1$ of $M \times K$ such that $\SA_1(\SC_1)$ is ample along principal directions and $F$ is a formal solution. In particular, $F$ is short for $\pi_1^r \circ \iota_{\SR \times K}$.

We then apply convex integration along a principal cover (Theorem \ref{thm:convexIntegration2}). It follows that $F$ is homotopic to a formal solution 
\[ G_1: M \times K \to \SA_1(\SC_1) \subset \SR \times K \]
that is holonomic up to first order. Applying this reasoning inductively on $r'$ we produce a holonomic solution $G: M \times K \to \SR \times K$ homotopic to $F$. The section $G$ is equivalent to a family of holonomic solutions $(G_k)_{k \in K}: M \to \SR$ homotopic to $(F_k)_{k \in K}$. This concludes the non-relative proof.

For the relative case we observe that, according to Theorem \ref{thm:convexIntegration2}, the homotopy connecting $F$ and $G$ can be assumed to be constant along any closed set in which $F$ was already holonomic. Since we are working in the foliated setting, this proves the parametric nature of the $h$-principle both in parameter ($K$) and domain ($M$).
\end{proof}

\begin{corollary} \label{cor:ampleness3Integrability}
If $\SR$ is ample up to avoidance, it is ample in the sense of integrable convex hull extensions.
\end{corollary}
\begin{proof}
Fix some arbitrary formal solution $F: M \times K \to \SR \times K$, holonomic of order $r'$. During the proof of Theorem \ref{thm:main} we have shown that $F$ is a formal solution of the avoidance relation $\SA_{r'+1}(\SC_{r'+1})$, which is ample along principal frames. It follows that all formal solutions are integrably short, so Definition \ref{def:ampleness3} applies to $\SR$.
\end{proof}

\subsection{Jiggling for principal covers} \label{ssec:jiggling}

In this Subsection we prove Proposition \ref{prop:main}, completing the proof of Theorems \ref{thm:main} and \ref{thm:mainRestate}. Our goal is to find a principal cover $\SC$ compatible with a given avoidance template $\SA$ and a formal solution $F$. We construct $\SC$ using a \emph{jiggling} argument. Namely, we start with an (arbitrary) principal cover $\SC'$ which we then subdivide repeatedly. When the subdivision is fine enough, we tilt/jiggle the corresponding principal frames in order to obtain the claimed $\SC$.

This argument is (strongly) reminiscent of the classic version of jiggling due to W. Thurston \cite{Th}. For completeness, we recall it in Subsection \ref{ssec:Thurston}; its contents are not really needed for our arguments and can be skipped. Our goal with this is to highlight the similarities between the two schemes. Despite of the many parallels, it is unclear to the authors whether there is some natural generalisation subsuming both results.

\begin{remark} \label{rem:hPrincipleWOHA}
We think of both jiggling arguments (both Thurston's and ours) as \textbf{h-principles without homotopical assumptions}.

Namely, being transverse to a given distribution $\xi$ is a differential relation for submanifolds of $(N,\xi)$. It may not be possible, in general, to find solutions of this relation. However, by dropping the smoothness assumption on the submanifold (allowing it to have instead triangulation-like singularities), Thurston produces solutions. Similarly, given a formal solution $F: M \to \SR$, we can define a first order differential relation for tuples of functions $\{f_i: M \to \R\}$ by requiring $(F,\{df_i\}) \in \SA$. By allowing the functions to be defined only locally (as coordinate codirections of charts), we are effectively introducing discontinuities; this is the flexibility we need to find a suitable $\SC$.

In both cases, the main point is that, due to the presence of discontinuities, there is no formal data associated to the objects we consider. $h$-Principles without homotopical assumptions play now a central role in Symplectic and Contact Topology through the arborealisation programme \cite{Sta,AEN1,AEN2,AEN3}.
\end{remark}

Recall the setup of Proposition \ref{prop:main}: We are given a manifold $M$, a bundle $X \to M$, an over-relation $\iota_\SR: \SR \to J^r(X)$, an avoidance template $\SA$, and a formal solution $F: M \to \SR$. We want to find a principal cover $\SC$ such that $\SA(\SC)$ is ample along $\SC$ and $F$ is still a formal solution of $\SA(\SC)$.

We will assume that $M$ is compact. If not, the upcoming argument can be adapted to use an exhaustion by compacts.

\subsubsection{Picking an atlas} \label{proof:atlas}

We pick an arbitrary atlas $\SU$ of $M$. We require $\SU$ to use closed, cubical charts, i.e. each $(U,\phi) \in \SU$ has image $[-1,1]^n \subset \R^n$. Due to compactness, we may assume that $\SU$ is finite. We will still write $\phi$ to mean an arbitrary but fixed extension of $\phi$ to an open neighbourhood of $U$. We pick $\phi^{-1}(0) \in U$ as a marked point for each $(U,\phi) \in \SU$.

To each ordered pair $((U,\phi_U), (V,\phi_V))$ in $\SU \times \SU$ we associate the transition function $\phi_{UV} := \phi_V \circ \phi_U^{-1}$. Its domain and codomain are the images of $U \cap V$.

\subsubsection{Choosing principal frames} \label{proof:principalFrames}

Given $(U,\phi) \in \SU$, we pick a principal frame $\Xi_U$ with support in $U$. We write $e$ for the cardinality of this frame, which is the dimension of the fibres of $\pi^r_{r-1}$. We require $\phi_*\Xi_U$ to be invariant with respect to the translations in $\R^n$. Such an invariant principal frame is in correspondence with a principal basis at the marked point. During our arguments we think of the two interexchangeably.

The collection of all principal frames $\Xi_U$, as we range over the different $(U,\psi) \in \SU$, defines a principal cover $\SC'$ of $M$.

\subsubsection{Subdivision}  \label{proof:subdivision} 

Fix some real number $C>1$. Let $c$ be a positive integer to be fixed later on during the proof.
 
We subdivide $[-1,1]^n$ into $(2c)^n$ cubes of side $1/c$, homothetic to the original. Given $(U,\phi) \in \SU$, we apply this subdivision to $U$ using $\phi$. This yields a new collection of cubical charts, which we denote by $\SU(c)$. A cube $V \in \SU(c)$ is said to be the \emph{child} of a \emph{parent} cube $U \in \SU$ if it is obtained from $U$ by subdivision. Two children of the same parent are \emph{siblings}.

Each child $V$ inherits the parent chart $\phi$, mapping now to a small cube of side $1/c$ contained in $[-1,1]^n$. The marked point of $V$ is the preimage by $\phi$ of the center of its image. The transition function between two given cubes in $\SU(c)$ is inherited from the parents. In particular, if two cubes are siblings, the transition function between them is the identity (restricted to their overlap).

$\SU(c)$ need not be a cover, since siblings overlap along sets with empty interior. To obtain a cover $\SV(c)$, we dilate each cube in $\SU(c)$, with respect to its center, by $C$. If $c$ is sufficiently large, dilating by $C$ makes sense even for children close to the boundary of the parent. This is why we extended the charts in $\SU$ to slightly bigger opens. After $C$-dilation, each child chart $(V,\phi) \in \SV(c)$ has for image a cube of side $C/c$. The domain of the transition functions is changed accordingly. See Figure \ref{fig:Cdilation}.

\begin{figure}[ht]
		\includegraphics[scale=1.8]{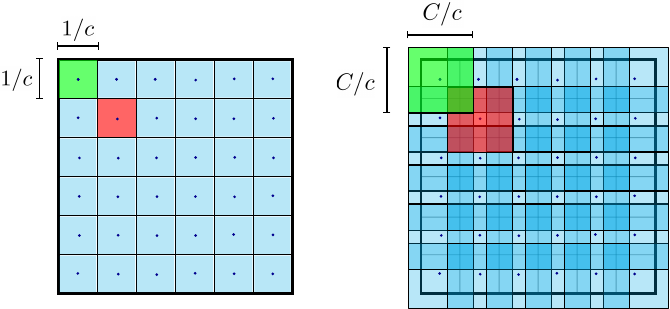}
		\centering
		\caption{The big cube on the left represents the image $[-1,1]^n$ of a chart $(U,\phi) \in \SU$, before the $C$-dilation is introduced. We show how it is subdivided into smaller cubes, as well as the (image by $\phi$ of the) marked point of each smaller cube. Two of the children are marked in green and red. On the right, we depict the $C$-scaling of each child cube. The dilated green and red cubes, which originally met at a single point, now have an intersection with non-empty interior.} \label{fig:Cdilation}
\end{figure}

Lastly, we attach to each cube in $\SV(c)$ a principal frame. It is simply the restriction of the principal frame of the parent. The collection of all these principal frames, as we range over $\SV(c)$, is a principal cover that we call $\SC'(c)$.

We now prove a number of quantitative properties for $\SV(c)$, as we take $c$ to infinity. We fix a fibrewise metric on $\NP T^*M$. This defines a fibrewise metric in $\barHConf(TM)$, since its fibres are simply products of projective spaces.

\subsubsection{A bound on the number of overlapping cubes} \label{proof:boundOverlaps}

Given $U \in \SV(c)$ we write $\SV_U(c)$ for the collection of cubes in $\SV(c)$ that intersect $U$ non-trivially.
\begin{lemma}
There is an upper bound $d_1$, independent of $c$ and $U$, for the cardinality of $\SV_U(c)$.
\end{lemma}
\begin{proof}
This is trivially true when we restrict ourselves to sibling cubes. For unrelated cubes we reason as follows: Write $P \in \SU$ for the parent of $U$ and consider the image $\phi_{PP'}(U) \subset P'$ in some other cube $P' \in \SU$, where $\phi_{PP'}$ is the transition function between $P$ and $P'$. Since $\phi_{PP'}$ is $C^1$-bounded by compactness, the diameter of $\phi_{PP'}(U)$ behaves as $O(1/c)$. The children of $P'$ form a lattice spaced $1/c$. It follows that $\phi_{PP'}(U)$ can only intersect an amount $O(1)$ of them. The claim is complete since $\SU$ has finite cardinality. 
\end{proof}

\subsubsection{Colouring} \label{proof:colours}

\begin{lemma}
There is an integer $d_2$, independent of $c$, such that we can partition $\SV(c)$ into $d_2$ colours $\{\SV(c)^{(i)}\}_{i=1,\cdots,d_2}$ with the following property: If $U, V \in \SV(c)$ belong to different colours, then they have no common neighbours (i.e. elements $W \in \SV(c)$ overlapping non-trivially with both).
\end{lemma}
\begin{proof}
This property is clear if we restrict to children of a fixed parent in $\SU$, since children are spaced uniformly as $O(1/c)$ and have size $O(1/c)$. Then, by finiteness of $\SU$, the claim follows if we use different sets of colours for each parent in $\SU$.
\end{proof}

\subsubsection{Trivialising the configuration bundles} \label{proof:trivialising}

Given $(U,\phi) \in \SV(c)$, we look at the bundle of hyperplanes $\NP T^*U$. Consider the marked point $u \in U$ and the corresponding fibre $\NP T_u^*U$. Using the parallel transport provided by the translations in the image of $\phi$, we trivialise:
\[ \NP T^*U \,\cong\, U \times \NP T_u^*U. \]
We denote the resulting projection by
\[ \pi_U: \NP T^*U \to \NP T_u^*U. \]

Similarly, we trivialise the bundle $\barHConf(TU)$ as $U \times \barHConf(T_uU)$. This produces again a projection 
\[ \barHConf(TU) \to \barHConf(T_uU). \]
We abuse notation and also call $\pi_U$; it should be apparent from context which of the two we mean.

Due to the compactness of $M$, the charts, projective space, and $\barHConf$, we have that:
\begin{lemma} \label{lem:dilations}
Fix a positive integer $j_0$. Then, there is a constant $A>1$, independent of $c$ and $(U,\phi) \in \SV(c)$, such that the following holds:

Fix a point $x \in U$. Identify $\NP T_u^*U$ with $\NP T_x^*U$ using $\pi_U$. Then, the fibrewise metrics at $x$ and $u$ bound each other from above up to a factor of $A$.

The same statement holds, for every $j \leq j_0$, for the metrics in $\barHConf_j(T_uU)$ and $\barHConf_j(T_xU)$.
\end{lemma}

\subsubsection{The diameter of a hyperplane field} \label{proof:boundHyperplane}

Given $(U,\phi) \in \SV(c)$, we look at the principal directions coming from neighbourhouring cubes. We want to show that these form a set whose diameter goes to zero as $c \to \infty$. We formalise this as follows, using the notation from the previous item.
\begin{lemma} \label{lem:boundHyperplane}
Fix a second cube $(V,\psi) \in \SV(c)$. Fix an integrable hyperplane field $\tau: V \to \NP T^*V$, invariant under the translations in $(V,\psi)$. Consider the composition $\pi_U \circ \tau: U \cap V \to \NP T_u^*M$. 

The diameter of $\image(\pi_U \circ \tau)$ behaves like $O(1/c)$ and the constants involved do not depend on $(U,\phi)$, $(V,\psi)$, or $\tau$.
\end{lemma}
\begin{proof}
First note that the claim is automatic if $U$ and $V$ are siblings. Indeed, $\tau$ is then translation-invariant for the parent and thus for $U$, so $\pi_U \circ \tau$ is constant. Otherwise, write $P$ for the parent of $U$ and $R$ for the parent of $V$. Let $\phi_{RP}$ be the transition function between the two; it restricts to the transition function between $V$ and $U$.

The Taylor remainder theorem states that
\[ d_{u+h}\phi_{RP} = d_u\phi_{RP} + O(h) \]
and the remainder is controlled by the second derivatives of $\phi_{RP}$, which are bounded independently of $c$, $U$, $V$ and $\tau$. Since the diameter of $U$ is $O(1/c)$, we have that
\[ d_x\phi_{RP}(\pi_U \circ \tau) = d_u\phi_{RP}(\pi_U \circ \tau) + O(1/c) \qquad \text{for all $x \in U \cap V$}, \]
proving the claim.
\end{proof}

\subsubsection{The diameter of a principal cover} \label{proof:boundCover}

We now look at covers instead of individual hyperplane fields. Fix $(U,\phi) \in \SV(c)$ and consider all the neighbouring cubes. For each cube $(V,\psi) \in \SV_U(c)$, suppose a principal frame $\xi_V$ is given (not necessarily the one in $\SC'(c)$ we fixed earlier). We may assume that $\xi_V$ is defined over the whole of $U$ simply by temporarily dilating $V$ (a factor of $2$ is sufficient).

The cardinality of $\SV_U(c)$ is at most $d_2$ and the cardinality of each $\xi_V$ is exactly $e$. By concatenating all the principal frames $(\xi_V)_{V \in \SV_U(c)}$, we can regard them as a section 
\[ s_U: U \to \barHConf_j(TU), \qquad \text{ for some } j \leq d_2.e. \]
Using $\pi_U$, we see $s_U$ as a map $U \to \barHConf_j(T_uU)$.
\begin{lemma} \label{lem:boundCover}
The diameter of $\image(s_U)$ behaves like $O(1/c)$. The constants involved do not depend on $(U,\phi)$ nor on $(\xi_V)_{V \in \SV_U(c)}$. 
\end{lemma}
\begin{proof}
The metric on $\barHConf(TM)$ is just the product metric inherited from the metric in $\NP T^*M$. Then the claim follows from Lemma \ref{lem:boundHyperplane} due to the finiteness of $j$.
\end{proof}

\subsubsection{Density bounds on the avoidance template} \label{proof:boundTemplate}

The discussion up to this point referred only to coverings and principal frames. The avoidance template $\SA$ and the formal solution $F: M \to \SR$ enter the proof now. We will make use of openness and Property (III) in the definition of a template. Our goal is to provide a quantitative estimate regarding the size of the balls contained in $\barSA(F)$ that one can find on a given $\varepsilon$-ball in $\barHConf(TM)$.

Fix $(U,\phi) \in \SV(c)$ with marked point $u$. Using the projection $\pi_U$ we can associate to $\barSA(F)$ the singularity:
\[ \Sigma_U \,:=\, \pi_U(\barSA(F)^c \cap \barHConf(TU)) \subset \barHConf(T_uU), \]
where the superscript $c$ denotes taking complement.

\begin{lemma} \label{lem:boundTemplate}
Let $\varepsilon > 0$ be given. Then, there exists $\delta > 0$ such that, for any sufficiently large $c$ and any $j \leq d_2.e$, the following property holds:

Fix $(U,\phi) \in \SV(c)$ with marked point $u$. Each $\varepsilon$-ball in $\barHConf_j(T_uU)$ contains a $\delta$-ball disjoint from $\Sigma_U$.
\end{lemma}
\begin{proof}
Consider $x \in M$ arbitrary but fixed. We claim that there is $\delta_x > 0$ such that every $\varepsilon$-ball in $\barHConf_j(T_xM)$ contains a $\delta_x$-ball fully contained in $\barSA(F)$; see Figure \ref{fig:Lemma512}. Indeed: suppose $B$ is a $\varepsilon/2$-ball in $\barHConf_j(T_xM)$. Since $\barSA(F)$ is fibrewise dense, there exists $\Xi \in \barSA(F) \cap \barHConf_j(T_xM)$. By openness of $\barSA(F)$ there is a $\delta_\Xi$-ball $D \subset \barSA(F) \cap \barHConf_j(T_xM)$ centered at $\Xi$ and contained in $B$. We now use the compactness of $\barHConf_j(T_xM)$ to extract a finite cover by $\varepsilon/2$-balls $\{B_i^x\}$. There are corresponding $\delta_x$-balls $\{D_i^x\}$ contained in $\barSA(F) \cap B_i^x$. Any $\varepsilon$-ball in $\barHConf_j(T_xM)$ contains one of the $B_i^x$ and thus the corresponding $D_i^x$, as claimed.

\begin{figure}[ht]
		\includegraphics[scale=0.61]{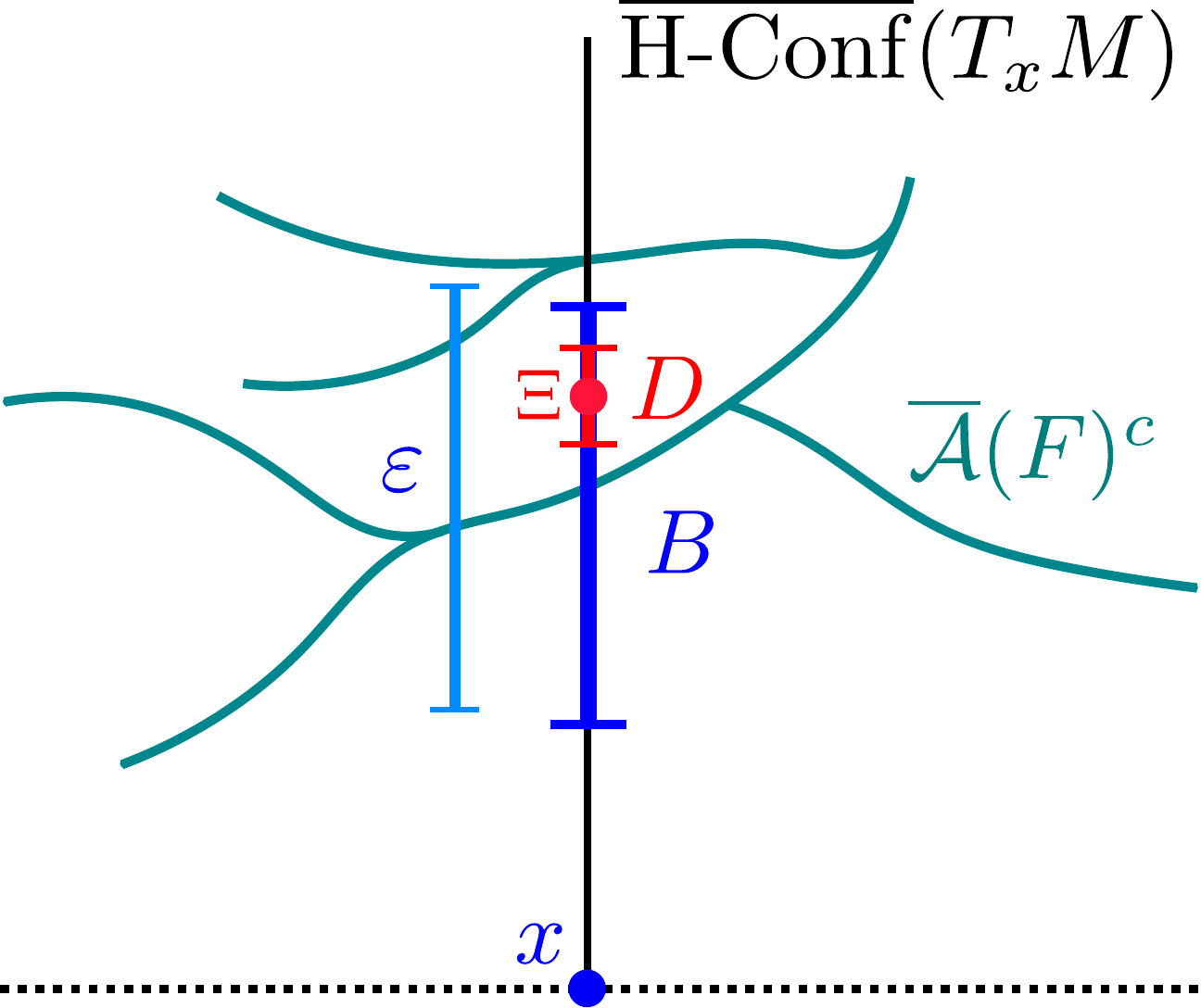}
		\centering
		\caption{The vertical line depicts the fibre of $\barHConf_j(TM)$ over a given $x \in M$. The branching set running more or less horizontally is the complement of $\barSA(F)$. Given a ball $B$ of radius $\varepsilon/2$ in the fibre, we find $\Xi \in \barHConf_j(T_xM)$ and a ball $D \subset B$ around it such that $D$ is fully contained in $\barSA(F)$. This argument uses only that $\barSA(F)$ is dense.} \label{fig:Lemma512}
\end{figure}

Before we address the statement, let us introduce some notation. Fix $(P,\phi_P) \in \SU$ and $p \in P$, not necessarily the marked point. We use $\phi_P$ to trivialise $\barHConf_j(TP) = P \times \barHConf_j(T_pP)$, allowing us to speak of the $\alpha \times \beta$-polydiscs given by such a trivialisation. These are products of an $\alpha$-disc along $P$ (measured by the euclidean metric of the chart) and a $\beta$-disc along the fibre $\barHConf_j(T_pP)$ (measured by the fibrewise metric at $p$). By definition, a $\alpha \times \beta$-polydisc is obtained from a fibrewise $\beta$-disc by parallel transport to the nearby fibres. We can now use the openness of $\barSA(F)$ to thicken the collection $\{D_i^p\}$ to a family of $\rho_p \times \delta_p$-polydiscs contained in $\barSA(F)$, for some $\rho_p > 0$. We abuse notation and still denote these thickenings by $\{D_i^p\}$.

Using the finiteness of $\SU$ and the compactness of each $(P,\phi_P) \in \SU$ we can then find constants $\rho,\delta > 0$, and lattices of points $\{p_l^P \in P\}_{l \in L, P \in \SU}$ spaced as $\rho/2$, such that $\rho < \rho_{p_l^P}/2$ and $\delta < \delta_{p_l^P}/A$, for all $l \in L$ and $P \in \SU$. Here $A$ is the dilation factor given in Lemma \ref{lem:dilations}.

\begin{figure}[ht]
		\includegraphics[scale=0.7]{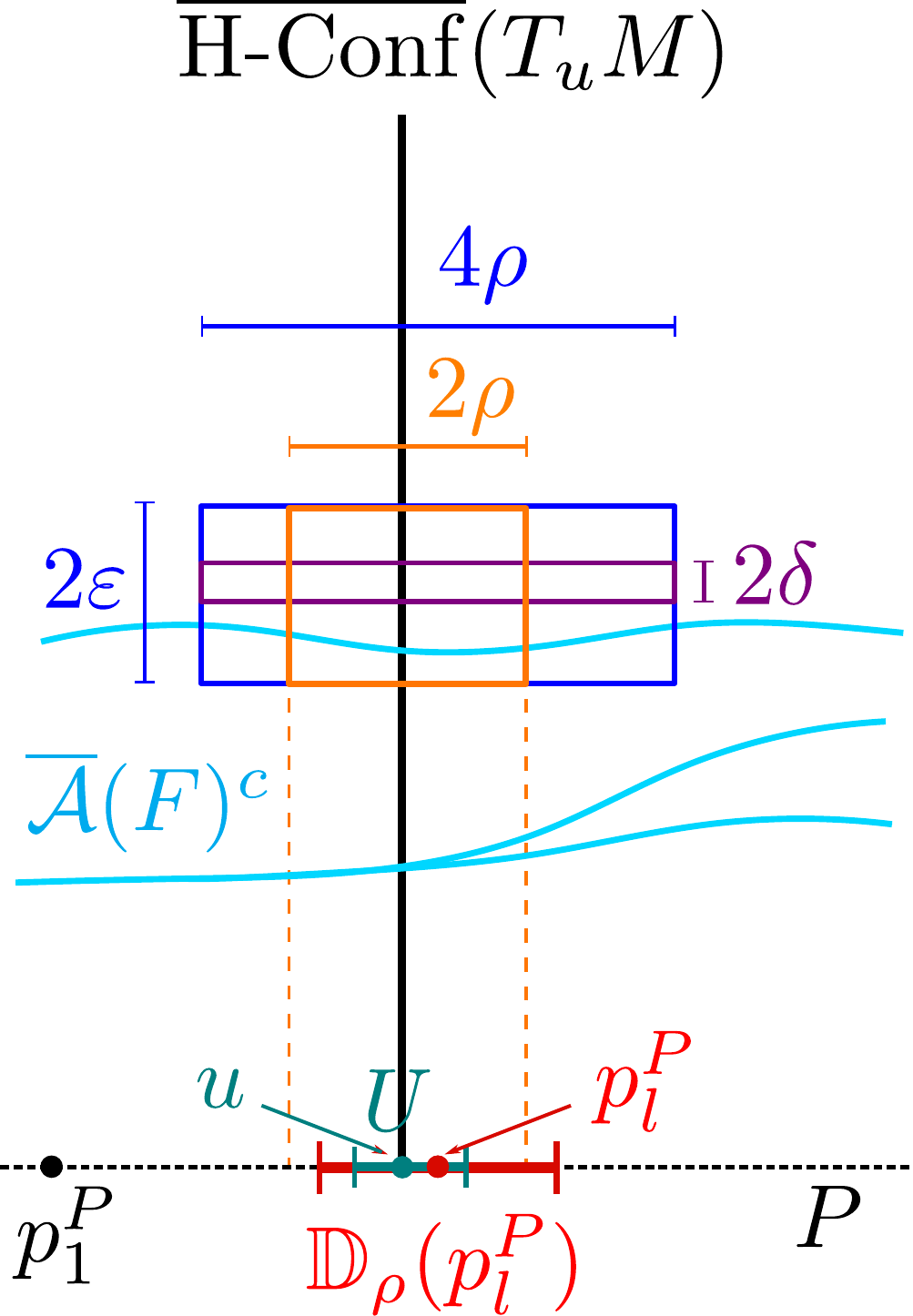}
		\centering
		\caption{The chart $U \in \SV(c)$ is fully contained in the $\rho$-ball centered at some $p_l^P$. The black vertical line depicts the fibre $\barHConf_j(T_uM)$ over its marked point $u \in U$. The branching set in light blue is the complement of $\barSA(F)$. Given a $\rho \times \varepsilon$-polydisc (with orange border) centered somewhere in $\barHConf_j(T_uM)$, there is a $2\rho \times \varepsilon$-polydisc (in blue) that contains it. In turn, the latter contains a $2\rho \times \delta$-polydisc (in purple) which is disjoint from $\barSA(F)^c$. This exhibits a $\delta$-ball disjoint from $\Sigma_U$.} \label{fig:polydisks}
\end{figure}

If $c$ is sufficiently large, any child $(U,\phi) \in \SV(c)$ of a given $(P,\phi_P) \in \SU$ is contained in the $\rho$-disc centered at some $p_l^P \in P$. In particular, any $\rho \times \varepsilon$-polydisc centered at $\barHConf_j(T_uM)$, with $u$ the marked point of $U$, is contained in a $2\rho \times \varepsilon$-polydisc $B_i^{p_l^P}$ centered at $\barHConf_j(T_{p_l^P}M)$. Then, upon projecting with $\pi_U$, the corresponding $2\rho \times \delta$-polydisc $D_i^{p_l^P}$ provides the claimed $\delta$-ball disjoint from $\Sigma_U$. See Figure \ref{fig:polydisks}.
\end{proof} 

\subsubsection{Jiggling}  \label{proof:jiggling}

The proof concludes by applying jiggling. We fix an arbitrary constant $\varepsilon_0 > 0$. By making it smaller we will be proving that the jiggling can be assumed to be as small as we want. We then define a sequence of constants (as many as colours):
\[ \varepsilon_0 > \varepsilon_1 > \cdots > \varepsilon_{d_2} > 0 \]
by iteratively applying Lemma \ref{lem:boundTemplate}. Namely, $4\varepsilon_{i+1}$ should be the ``$\delta$'' corresponding to $\varepsilon_i$. Furthermore, we impose for each $\varepsilon_i$ to be much bigger than the subsequent ones. Concretely, the following inequality should hold:
\begin{equation}\label{eq:jigglingBound}
 \varepsilon_i > 2.A \sum_{j > i} \varepsilon_j.
\end{equation}
The successive applications of Lemma \ref{lem:boundTemplate} provide us then with a lower bound for $c$.

We start with the first colour $\SV(c)^{(1)}$, working simultaneously with all its elements. Let $(U,\phi) \in \SV(c)^{(1)}$ with marked point $u$. Consider all the neighbouring $(V,\psi) \in \SV_U(c)$, each with a corresponding principal frame $\Xi_V$. Together, these define a map $s_U: U \to \barHConf(T_uM)$, as in Subsection \ref{proof:boundCover}. According to Lemma \ref{lem:boundCover}, the image of $s_U$ has diameter $O(1/c)$. In particular, if $c$ is sufficiently large (of magnitude $O(1/\varepsilon_1)$), we can assume that this diameter is smaller than $\varepsilon_1$. Using Lemma \ref{lem:boundTemplate} we can perturb each $\Xi_V$ to a nearby frame $\Xi_V'$ such that the corresponding map $s_U': U \to \barHConf(T_pM)$ satisfies:
\begin{itemize}
\item The $C^0$-distance between $s_U'$ and $s_U$ is bounded above by $\varepsilon_0/2$.
\item The $\varepsilon_1$-neighbourhood of $\image(s_U')$ is contained in $\barSA(F)$.
\end{itemize}
I.e. we have jiggled all the frames in the vicinity of $U$, producing new frames whose distance to the complement of $\barSA(F)$ is controlled.

We do the same inductively on the number of colours. At step $i$ we look at all the $(U,\phi) \in \SV^{(i)}(c)$ at once. Using Lemma \ref{lem:boundTemplate} we perturb the neighbouring frames $s_U: U \to \barHConf(T_uU)$ to a nearby section $s_U'$ satisfying:
\begin{itemize}
\item The $C^0$-distance between $s_U'$ and $s_U$ is bounded above by $\varepsilon_i$.
\item The $\varepsilon_{i+1}$-neighbourhood of $\image(s_U')$ is contained in $\barSA(F)$.
\end{itemize}
This is possible as long as $c$ is large enough; concretely, of magnitude $O(1/\varepsilon_{d_2})$. The second item gives us a lower bound for the distance to $\barSA(F)$ which, in light of Equation \ref{eq:jigglingBound}, is not destroyed in later steps thanks to the first item. After $d_2$ steps, the proofs of Proposition \ref{prop:main}, Theorem \ref{thm:main}, and Theorem \ref{thm:mainRestate} are complete. \hfill$\Box$

\subsection{Detour: Thurston's jiggling} \label{ssec:Thurston}

We invite the reader to compare the upcoming discussion to the proof of Proposition \ref{prop:main}. Let us stress once more that this Subsection is only included for the sake of pointing out the parallels between the two.

The classic jiggling procedure was introduced by W. Thurston in \cite{Th}. It allows the user to produce a triangulation whose simplices are transverse to a given distribution $\xi$.  We work with a manifold $N$ of dimension $n$. We consider triangulations $\ST$ all whose $i$-simplices $\Delta \in \ST^{(i)}$ are endowed with a parametrisation identifying them with the standard simplex in $\R^i$. Furthermore, they come with a parametrised neighbourhood germ, which is then identified with a neighbourhood of the standard simplex in $\R^i \subset \R^n$.

We say that $\Delta$ is in \textbf{general position} with respect to $\xi$ if the quotient map 
\[ \text{mod}(\xi_p): \Delta \subset \R^i \subset \R^n \to \R^{n-k} \]
that quotients by $\xi_p$ has for image a subset diffeomorphic to a simplex of dimension $\min(i,n-k)$. Here we are restricting $\xi$ to the aforementioned coordinates around $\Delta$. Do note that the general position condition is then an strengthening of the condition that $\Delta$ is transverse to $\xi$. A triangulation is in general position if all its simplices are.

\begin{proposition} \label{prop:Thurston}
Let $(N,\xi)$ be an $n$-manifold endowed with a distribution. Then, there exists a triangulation $\ST$ in \textbf{general position} with respect to $\xi$.

Furthermore, if a constant $\varepsilon > 0$ and a triangulation $\ST'$ of $N$ are given, we may assume that $\ST$ is obtained from $\ST'$ by applying finitely many cubical subdivisions and then perturbing the vertices of the resulting triangulation by an amount no larger than $\varepsilon$.
\end{proposition}
\begin{proof}
We break the argument into steps, much like we did for Proposition \ref{prop:main}. First, we fix a locally finite atlas $\SU$ of $N$.

Upon subdividing $c$ times, we may assume that $\ST'$ is subordinated to $\SU$. Cubical subdivision guarantees that the radius of the simplices goes to zero as $c \to \infty$, while the cardinality of the star of each simplex is bounded independently of $c$.

A key property, which follows from the previous paragraph, is that there is a number $d_2$, independent of $c$, such that we can colour the set of vertices into $d_2$ colours so that no two vertices of the same colour are contained in the same simplex.

We straighten out the simplices so that every $\Delta \in \ST'$ is linear with respect to some chart in $\SU$.

We tilt/jiggle the positions of the vertices of $\ST'$ in order to change how all simplices are embedded and thus achieve transversality. We do this inductively colour by colour. It follows that, at each inductive step, the vertices we tilt do not interfere with one another.

For a given vertex $p$, we consider those simplices that contain $p$ and whose other vertices are from previous colours. When we tilt $p$, the good choices are those that make said simplices transverse. As $c \to \infty$, the measure of the subset of bad choices goes to zero (since the radii of all simplices go to zero).

Given $\varepsilon_0>0$ and any $\varepsilon_1>0$ sufficiently small, we can always take $c$ large enough so that, in each $\varepsilon_0$-ball of choices, there is a $\varepsilon_1$-ball of good choices. This reasoning can be repeated with $\varepsilon_1$ and some $\varepsilon_2 > 0$. Repeating it $d_2$ times yields a sequence of positive constants $(\varepsilon_i)_{0,\cdots,d_2}$. Our tilt at step $i+1$ can then be taken to be $\varepsilon_i/d_2$ small and contained in a $\varepsilon_{i+1}$-ball of good choices as long as $c$ is sufficiently large. It follows that the transversality achieved in a given step is not destroyed in subsequent ones. The proof concludes after $d_2$ steps.
\end{proof}

\section{An example: Exact, linearly-independent differential forms} \label{sec:exactForms}

In this section we introduce a differential relation $\SR$ for which the thinning process (Subsection \ref{sssec:thinning}) terminates producing an avoidance template. We let $M$ be a $3$-dimensional manifold and we set $X := T^*M \times T^*M$. The relation $\SR \subset J^1(X)$ is of first order and Diff-invariant. Concretely, $\SR$ consists of pairs $(F_1,F_2) \in J^1(X)$, where each $F_i$ is a first-order Taylor polynomial of $1$-form, such that the 2-forms $dF_1$ and $dF_2$ are linearly independent.

In Subsection \ref{ssec:symbol} we recall some notation, involving the exterior differential, that will also be helpful in later sections. In Subsection \ref{ssec:exactForms} we formalise our claim about the thinning of $\SR$, which we then prove in Subsection \ref{ssec:exactFormsProof}.

In Subsection \ref{ssec:linearAlgebra} we will observe that the present example is a bit artificial: a linear algebra lemma due to Gromov shows that $\SR$ was already ample in all directions (but not thin). Nonetheless, we choose to include it as an easy incarnation of the avoidance/thinning approach.

\subsection{Intermezzo: The symbol of the exterior differential} \label{ssec:symbol}

Let $N$ be a manifold. We focus on the bundle $\bigwedge^k T^*N$ of $k$-forms. The corresponding space of $1$-jets $J^1(\wedge^k T^*N)$ is an affine fibration
\[ \pi_0: J^1(\wedge^k T^*N) \longrightarrow \wedge^k T^*N \]
whose model vector bundle is $\Hom(TN,\wedge^k T^*N) \cong T^*N \otimes \wedge^k T^*N$.

It follows that, given a formal datum $F \in J^1(\wedge^k T^*N)$ and a codirection $\lambda \in T_p^*N$, both based at the same point $p \in N$, the principal space associated to $F$ and $\lambda$ can be expressed explicitly as:
\[ \Pr_{\lambda,F} := \{F + (0,\lambda \otimes \beta) \,\mid\, \beta \in \wedge^k T_p^*N\}. \]

We are interested in discussing differential relations defined in terms of the exterior differential $d$. We will abuse notation and still denote its \textbf{symbol} by:
\begin{align*}
d:  J^1(\wedge^k T^*N) &\longrightarrow \wedge^{k+1} T^*N.
\end{align*}
Given $F \in J^1(\wedge^k T^*N)$, the symbol maps $F + (0,\sum_i \lambda_i \otimes \beta_i)$ to $dF + \sum_i \lambda_i \wedge \beta_i$. We also introduce the \textbf{extended symbol}:
\begin{align*}
\id \oplus d:  J^1(\wedge^k T^*N) & \longrightarrow \wedge^k T^*N \oplus \wedge^{k+1} T^*N, \\
																F & \mapsto (\pi_0(F),dF).
\end{align*}

\subsection{The statement} \label{ssec:exactForms}

We now restate the setup of our example. We fix a $3$-manifold $M$ and we focus on the first-order differential relation $\SR \subset J^1(T^*M \oplus T^*M)$ defined as:
\[ \SR := \{(F_1,F_2) \in J^1(T^*M) \oplus J^1(T^*M) \,\mid\, \text{$dF_1$ and $dF_2$ are linearly independent} \}. \]
The main result of this Section reads:
\begin{proposition} \label{prop:exactForms}
The thinning process for $\SR$ terminates in one step and the resulting thinning pre-template $\Thin(\SR)$ is an avoidance template.
\end{proposition}
In particular, $\SR$ is ample up to avoidance and thus, according to Theorem \ref{thm:main}, the $h$-principle holds for $\SR$.

The proof of Proposition \ref{prop:exactForms} requires an analysis of the structure of $\SR$ along principal subspaces (Subsection \ref{sssec:exactFormsThinning}). This will then allow us to describe $\Thin(\SR)$ explicitly and deduce that it is an avoidance template (Subsection \ref{sssec:exactFormsTemplate}).

\subsection{The proof} \label{ssec:exactFormsProof}

We introduce some auxiliary notation: Fix a point $p \in M$. Given a principal direction $\lambda \in T_p^*M$, we define the singularity
\[ \Sigma(\lambda) := \{F = (F_1,F_2) \in J^1_p(T^*M \oplus T^*M) \,\mid\, \text{ Both $dF_i$ are multiples of $\lambda$ }\}. \]
The complement of $\Sigma(\lambda)$ in $J^1_p(T^*M \oplus T^*M)$ is a first-order differential constraint, defined only at the point $p$. Nonetheless, we can still talk about it being ample; in fact, we will prove that $\Sigma(\lambda)$ is a thin singularity (see Lemma \ref{lem:nonZeroThin} below).

\subsubsection{The thinning step} \label{sssec:exactFormsThinning}

The following criterion will allow us to compute $\Thin(\SR)$:
\begin{lemma} \label{lem:exactFormsThinning}
Fix $p$, $\lambda$ and $F$ as above. The following two conditions are equivalent:
\begin{itemize}
\item $\SR_{\lambda,F} \subset \Pr_{\lambda,F}$ has thin complement.
\item $F \notin \Sigma(\lambda)$.
\end{itemize}
\end{lemma}
\begin{proof}
Recall that the principal space defined by $\lambda$ and $F$ is given by:
\[ \Pr_{\lambda,F} = \{(F_1,F_2) + ((0,\lambda \otimes \beta_1),(0,\lambda \otimes \beta_2)) \,\mid\, \beta_i \in T_p^*M\} \]
In particular, it is parametrised by the pairs $\beta_1 \times \beta_2 \in T_p^*M \times T_p^*M$. The restriction of the relation $\SR$ to $\Pr_{\lambda,F}$ reads:
\[ \SR_{\lambda,F} = \{(F_1,F_2) + ((0,\lambda \otimes \beta_1),(0,\lambda \otimes \beta_2)) \,\mid\, \text{ The forms } dF_i + \lambda \wedge \beta_i \text{ are linearly independent} \}. \]

The symbol of the exterior differential yields then a map
\[ d: \Pr_{\lambda,F} \,\longrightarrow\, \wedge^2 T_p^*M \times \wedge^2 T_p^*M \]
whose image $d\Pr_{\lambda,F}$ is $4$-dimensional. It is the product $L_1 \times L_2$ of the plane of $2$-forms 
\[ L_1 := \{dF_1 + \lambda \wedge \beta_1  \,\mid\, \beta_1 \in T_p^*M\}, \]
passing through $dF_1$, and the plane 
\[ L_2 := \{dF_2 + \lambda \wedge \beta_2 \,\mid\, \beta_2 \in T_p^*M\} \]
passing through $dF_2$. These two planes are parallel to the distinguished plane $L$ consisting of those $2$-forms proportional to $\lambda$.

There are two possible situations. The first possibility (Figure \ref{fig:RelativePositionCase1}) is that both $dF_i$ are contained in $L$ (equivalently, $L=L_1=L_2$; equivalently, $F \in \Sigma(\lambda)$). In this case, there are covectors $\nu_i$ such that $dF_i = \lambda \wedge \nu_i$. Then we can identify $d\Pr_{\lambda,F}$ with the 2-by-2 matrices:
\begin{equation*}
\begin{array}{rrcl}
\tau: & d\Pr_{\lambda,F}													& \longrightarrow & \ker(\lambda)^* \times \ker(\lambda)^* \cong \SM_{2 \times 2} \\
      & (dF_i + \lambda \wedge \beta_i)_{i=1,2}	& \mapsto         & ((\nu_i + \beta_i)|_{\ker(\lambda)})_{i=1,2}.
\end{array}
\end{equation*}
The subset $\SR_{\lambda,F}$ maps precisely to the linearly independent pairs $((\nu_i + \beta_i)|_{\ker(\lambda)})_{i=1,2}$. That is, $\tau \circ d$ is an affine submersion of $\SR_{\lambda,F}$ onto $\GL(2)$. We then conclude that the complement of $\SR_{\lambda,F}$ is not thin, because the complement of $\GL(2)$ (the zero set of the determinant, which has codimension $1$) is not. We will see in Subsection \ref{ssec:linearAlgebra} that $\GL(2)$ is nonetheless ample, and so is $\SR_{\lambda,F}$.

\begin{figure}[ht]
		\includegraphics[scale=1]{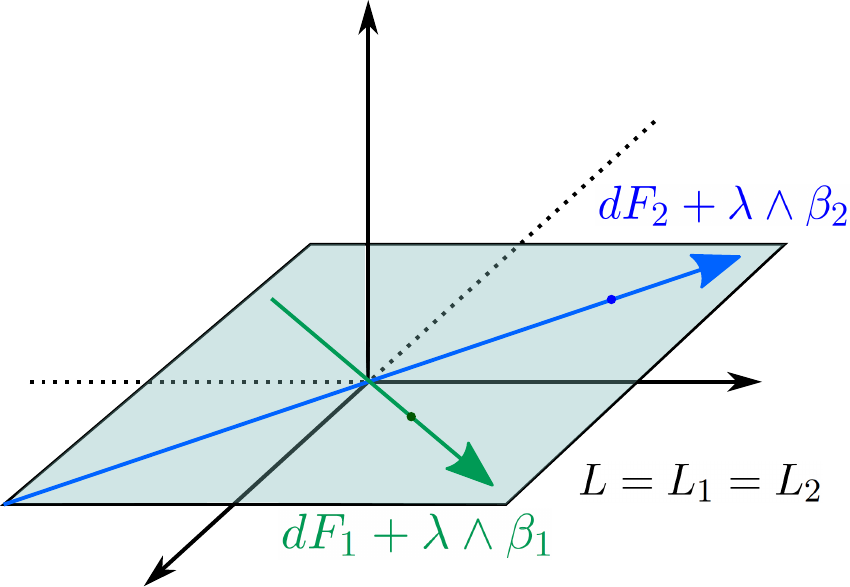}
		\centering
		\caption{The first possibility in the proof of Lemma \ref{lem:exactFormsThinning}: Both $dF_i$ are contained in $L$.} \label{fig:RelativePositionCase1}
\end{figure}

The other possibility is that one of the $dF_i$ is not proportional to $\lambda$ (equivalently, one of the $L_i$ is different from $L$; equivalently, $F \notin \Sigma(\lambda)$). Suppose, without loss of generality, that it is $dF_1$. Then, an element $dF_1 + \lambda \wedge \beta_1$ is colinear with a single element in $L_2$ (if $L_2 \neq L$) or with none (if $L_2 = L$). In the first case (Figure \ref{fig:RelativePositionCase2}), the complement of $\SR_{\lambda,F}$ is of codimension-2. In the second case (Figure \ref{fig:RelativePositionCase3}), the complement of $\SR_{\lambda,F}$ is the set $\{dF_2 = 0\}$, which is also of codimension-$2$. This completes the claim.
\end{proof}
\begin{figure}[ht]
		\includegraphics[scale=1]{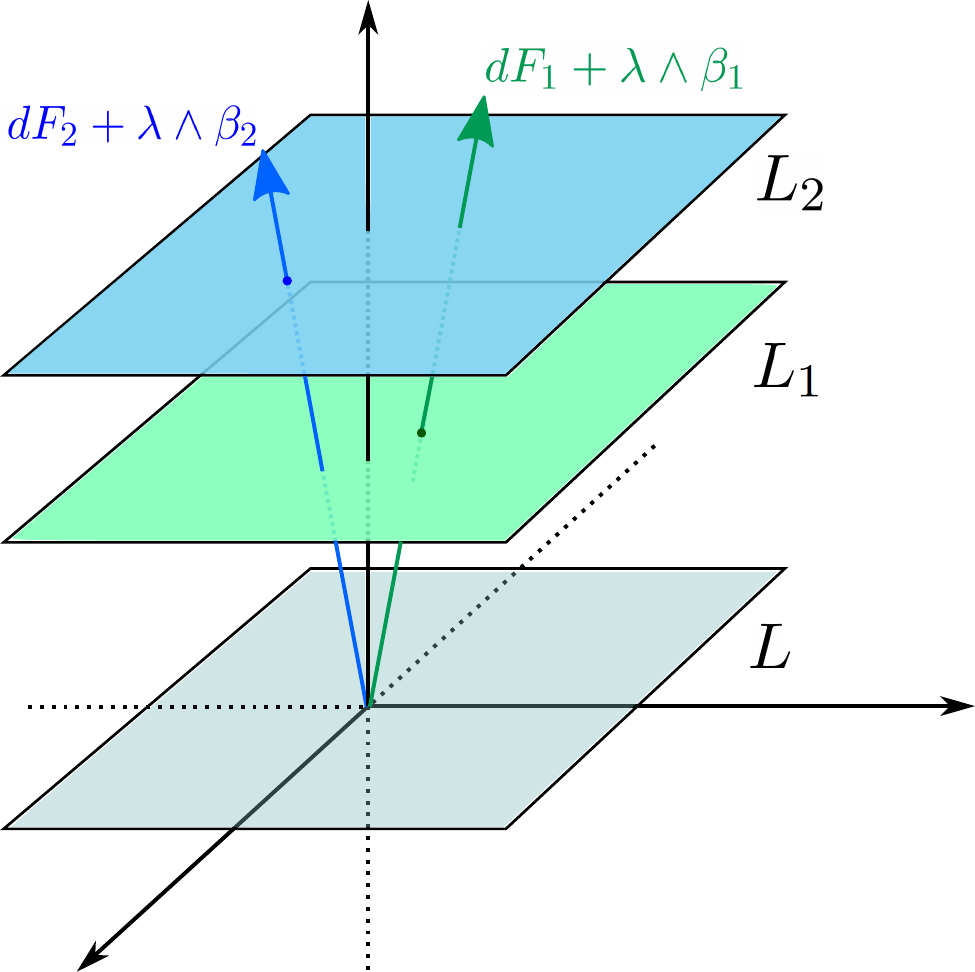}
		\centering
		\caption{The second possibility in the proof of Lemma \ref{lem:exactFormsThinning}: None of the $dF_i$ are contained in $L$.} \label{fig:RelativePositionCase2}
\end{figure}

\begin{figure}[ht]
		\includegraphics[scale=1]{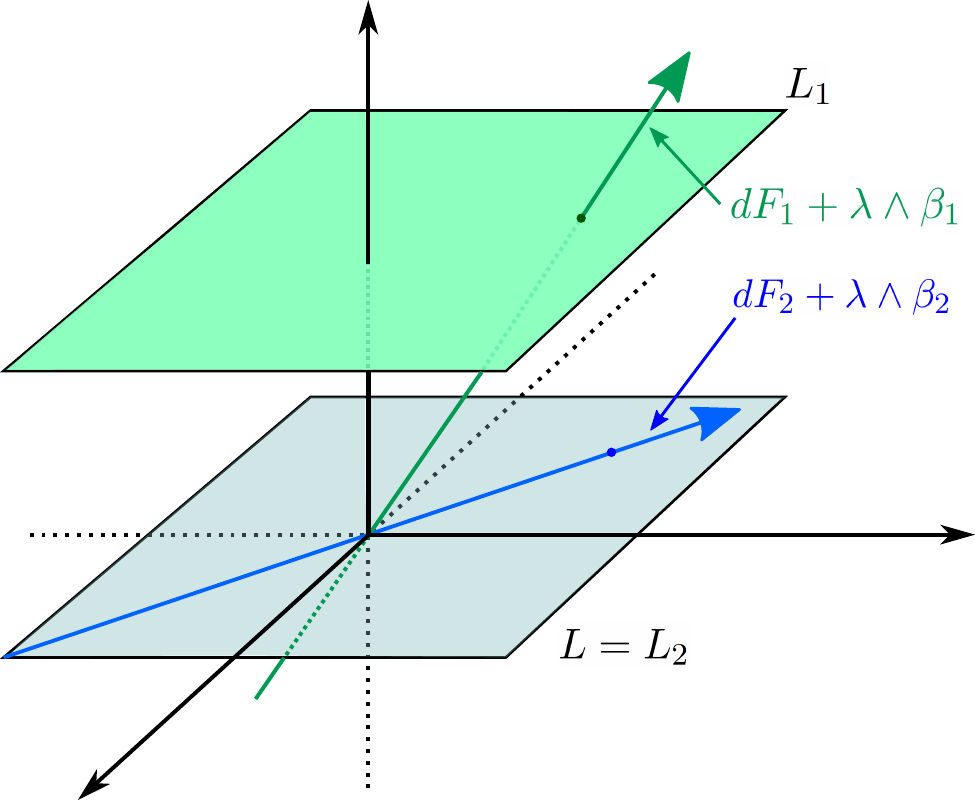}
		\centering
		\caption{The third possibility in the proof of Lemma \ref{lem:exactFormsThinning}: One of the $dF_i$ is contained in $L$ but the other one is not.} \label{fig:RelativePositionCase3}
\end{figure}

Recall the avoidance template notation from Subsection \ref{ssec:avoidanceTemplates}. We remark that the thinning process was defined as a pointwise process and Lemma \ref{lem:exactFormsThinning} indeed applies to each point $p \in M$ individually. An immediate consequence is:
\begin{corollary} \label{cor:exactFormsThinning}
The following statements are equivalent:
\begin{itemize}
\item $(F,\Xi) \in \SR \times_M \HConf(TM)$ belongs to $\Thin(\SR)$.
\item $F$ belongs to $\SR \setminus (\bigcup_{\lambda \in \Xi} \Sigma(\lambda))$.
\end{itemize}
\end{corollary}

\subsubsection{A second step is not necessary}

It turns out that the set we have removed from $\SR$ during the first step is itself thin:
\begin{lemma} \label{lem:nonZeroThin}
Fix a point $p \in M$ and a codirection $\nu \in T_p^*M$. Then, $\Sigma(\nu) \subset J^1_p(T^*M \oplus T^*M)$ is a thin singularity.
\end{lemma}
\begin{proof}
Fix a principal direction $\lambda \in T_p^*M$ and a formal datum $F \in J^1_p(T^*M \oplus T^*M) \setminus \Sigma(\nu)$. Recall the affine spaces $\Pr_{\lambda,F}$, $d\Pr_{\lambda,F}$, $L$ and $L_i$ from Lemma \ref{lem:exactFormsThinning}. We additionally consider the subspace $L'$ of $2$-forms spanned by $\nu$. The symbol $d$ maps the singularity $\Sigma(\nu) \cap \Pr_{\lambda,F} \subset \Pr_{\lambda,F}$ to the product $(L' \cap L_1) \times (L' \cap L_2) \subset d\Pr_{\lambda,F}$. There are three options:
\begin{itemize}
\item $\lambda$ and $\nu$ are proportional and $L_1 = L_2 = L = L'$. Then $\Pr_{\lambda,F} = \Sigma(\nu)$, which contradicts the fact that $F$ was in the complement.
\item $\lambda$ and $\nu$ are proportional, so $L = L'$, but at least one $L_i$, say $L_1$, is distinct from $L'$. Then $L' \cap L_1$ is empty and so is $\Sigma(\nu) \cap \Pr_{\lambda,F}$.
\item $\lambda$ and $\nu$ are not proportional. Then both intersections $L' \cap L_1$ are lines and thus $\Sigma(\nu)$ has codimension $2$ in $\Pr_{\lambda,F}$.
\end{itemize}
Only the two last possibilities can happen and the claim follows.
\end{proof}

\subsubsection{Completing the proof} \label{sssec:exactFormsTemplate}

The previous Lemmas allow us to conclude:
\begin{proof}[Proof of Proposition \ref{prop:exactForms}]
We want to show that $\Thin(\SR)$ is an thinning template. To do this, we must prove three properties. First, that for all $\Xi \in \HConf(TM)$, the subset $\Thin(\SR)(\Xi)$ has thin complement along each direction in $\Xi$. Second, that for all $F \in \SR$, the subset $\Thin(\SR)(F)$ is fibrewise dense in $\HConf(TM)$. Lastly, that $\Thin(\SR)$ is open.

According to Corollary \ref{cor:exactFormsThinning}, we have the following explicit description:
\[ \Thin(\SR)(\Xi) = \SR \setminus (\bigcup_{\lambda \in \Xi} \Sigma(\lambda)). \]
Fix $\nu \in \Xi$. According to Lemma \ref{lem:exactFormsThinning}, each subspace $\SR_{\nu,F}$ is contained in $\Sigma(\nu)$ (if $F \in \Sigma(\nu)$) or is disjoint from it and has thin complement. In the former case, $\Thin(\SR)(\Xi)_{\nu,F}$ is empty. In the latter case, we use Lemma \ref{lem:nonZeroThin} to note that all other singularities $\Sigma(\lambda) \cap \SR_{\nu,F}$ are thin singularities. This proves the first property.

For the second property, we fix $F \in \SR$.  A configuration $\Xi$ is in the complement of $\Thin(\SR)(F)$ if and only if there is $\lambda \in \Xi$ such that $F \in \Sigma(\lambda)$. Equivalently, if and only if $dF_1 \wedge \lambda = dF_2 \wedge \lambda = 0$. This is a non-trivial algebraic condition for $\Xi$, which proves the claim.

Lastly, note that $dF_1 \wedge \lambda = dF_2 \wedge \lambda = 0$ is an algebraic equality both on $dF_i$ and $\lambda$. Therefore, its complement is open (and, in fact, open in each variable upon freezing the other one), proving the third property.
\end{proof}

\subsection{Linear algebra} \label{ssec:linearAlgebra}

As we stated in the proof of Lemma \ref{lem:exactFormsThinning}, the subspace of non-degenerate matrices is an ample subset of the space of all matrices. This shows that the relation $\SR$ studied in this Section was, in fact, ample. For completeness we provide a proof:
\begin{proposition}\label{prop:amplenessGL}
The subspace $\GL(n)$ of non-singular matrices is an ample subset of the space of $(n\times n)$-matrices $\SM_{n\times n}$ if and only if $n \geq 2$.
\end{proposition}
\begin{proof}
The two components $\GL^+(n)$ and $\GL^-(n)$ are connected. We have to show that each is individually ample.

First note that every $n\times n$ matrix can be expressed as the convex combination of two non-singular matrices, since
\[ M=\frac{1}{2}(2M-2\lambda\cdot Id)+\frac{1}{2}\lambda\cdot 2Id \]
and the right hand-side is the sum of two non-singular matrices for any choice of $\lambda\notin \textrm{Spec}(M)\setminus{0}$. Therefore, it is enough to show that any matrix $M\in \GL^+(n)$ can be expressed as a convex combination of matrices in $\GL^-(n)$ (and viceversa). This readily follows by writing $M=\left(v_1,v_2,v_3,\cdots, v_n\right)$ (expressed in column vectors) as $M=\frac{1}{2}M_1+\frac{1}{2}M_2$, where
\[M_1=\left(-v_1, 3 v_2, v_3, \cdots, v_n\right)\]
\[M_2=\left(3 v_1, -v_2, v_3, \cdots, v_n\right).\]
Note that $M_1$ and $M_2$ do not belong to the same connected component of $\GL(n)$ as $M$ and, thus, the claim follows.
\end{proof}

\begin{proof}[Alternate proof]
Observe that $\SM_{n \times n}$ is convexely spanned by those matrices with a single non-zero entry. 

Then: Given a matrix $M$ and a sufficiently large constant $C$, it holds that $M$ is in the interior of the convex hull of the matrices $e_{i,j}^\pm$ whose single non-zero entry is $(i,j)$ with value $\pm C$. The matrix $e_{i,j}^\pm$ has zero determinant so it may be perturbed to yield a matrix $\tilde e_{i,j}^\pm$ with positive (resp. negative) determinant. In doing so, the convex hull is perturbed as well. However, $M$ will remain in the interior if the perturbations are small enough, by continuity. This proves the ampleness of $\GL^+(n)$ (resp. $\GL^-(n)$).
\end{proof}

We conclude:
\begin{corollary} \label{cor:linearAlgebra}
The subspace of non-singular matrices in $\SM_{n\times m}$ is ample unless $n=m=1$.
\end{corollary}
\begin{proof}
We may assume $n \neq m$; otherwise the space of singular matrices has codimension greater than $1$. Then the claim follows from the previous result.
\end{proof}

\begin{corollary}
The relation $\SR$ treated in this Section is ample in all principal directions and therefore abides by the $h$-principle.
\end{corollary}

\section{Relations involving functions} \label{sec:1dimBundle}

In this Section we restrict our attention to bundles $X \to M$ with $1$-dimensional fibres. The typical example is the trivial $\R$ bundle over $M$, whose sections are functions. Our claim is that convex integration (even if it includes avoidance) is ill-suited to address relations $\SR \subset J^r(X)$.

In Subsection \ref{ssec:functions} we prove that the the relation defining functions without critical points fails to be ample up to avoidance. In Subsection \ref{ssec:noVassiliev} we generalise this to arbitrary differential relations $\SR \subset J^r(X)$ that satisfy a mild non-triviality condition.

\begin{remark} \label{rem:GromovThin}
Let us compare these claims with \cite[Remark 1.3.4]{Gr73}. Gromov states that a generic codimension-2 singularity $\Sigma$ is thin. This is true if the fibres of $X$ have dimension at least $2$. Indeed, under genericity assumptions, $\Sigma$ intersects every principal subspace in a (maybe empty) codimension-$2$ subset.

However, if the fibres of $X$ have dimension $1$, $\Sigma$ does not intersect all the principal subspaces, only a subset of codimension-$1$. Further, these intersections are necessarily of codimension-$1$. As such, thinness and ampleness (even up to avoidance) fail.

Observe further that a differential relation given by a concrete geometric problem is, by definition, not generic. We claim that our methods can be used to generalise Gromov's statement to such non-generic situations. Suppose $\Sigma$ has codimension $2$ and the fibres of $X$ have dimension at least $2$. Even if $\Sigma$ intersects some principal subspaces in non-thin sets, it intersects most of them transversely, by Sard's theorem. One can then apply our methods to analyse the problematic subspaces.
\end{remark}

\begin{remark} \label{rem:Vassiliev}
We observe that the applicability realms of Vassiliev's $h$-principle \cite{Va} and convex integration are, in some sense, complementary. The former is most interesting when $\SR$ is the complement of a singularity of large codimension (at least $\dim(M)+2$) and $X$ has $1$-dimensional fibres. The latter is effective in the presence of much larger singularities but requires the fibres of $X$ to have dimension at least $2$.
\end{remark}

\subsection{A non-example: Functions without critical points} \label{ssec:functions}

Let $M$ be a manifold and we let $X$ be the trivial $\R$-bundle over $M$. As a differential relation in $J^1(X)$ we take 
\[ \SR = \{ F \in J^1(X) \,\mid\, dF \neq 0\}, \]
i.e. the $1$-jets of functions whose differential is non-zero. With the standard identification $J^1(X) \cong T^*M \times \R$ we see that $\SR$ is the complement of the singularity $M \times \R$.

Fix now a codirection $\lambda$ and a formal datum $F \in J^1(X)$, both based at the same point $p \in M$. The principal subspace associated to them is one-dimensional and explicitly given by:
\[ \Pr_{\lambda,F} = \{ F + (c\lambda,0) \in T^*M \times \R \,\mid\, c \in \R \}. \]
We readily see that there are two possible situations:
\begin{itemize}
\item $dF$ is proportional to $\lambda$. Then the complement of $\SR_{\lambda,F}$ is a point, which is not thin.
\item $dF$ is not proportional to $\lambda$. Then $\SR_{\lambda,F} = \Pr_{\lambda,F}$ so ampleness holds trivially.
\end{itemize}
This shows that Lemma \ref{lem:emptyAvoid} applies to $\SR$, allowing us to conclude that the standard pre-template $\Avoid^\infty(\SR)$ is empty for all configurations of codirections that form a generating set. The same applies to the thinning pre-template $\Thin^\infty(\SR)$. This was to be expected since, due to Morse theory, there is no $h$-principle for functions without critical points.

\subsection{The general case} \label{ssec:noVassiliev}

It follows immediately from Lemma \ref{lem:emptyAvoid} that:
\begin{proposition}
Let $X \to M$ be a bundle with $1$-dimensional fibres. Let $\SR \subset J^r(X)$ be the complement of a singularity $\Sigma$ that intersects non-trivially each fibre of $J^r(X) \to J^{r-1}(X)$.

Let $l_0$ be the dimension of the fibres of $J^r(X) \to J^{r-1}(X)$. Then, $\Avoid^{l_0}(\SR)$ is empty for all configurations of codirections of cardinality $l \geq l_0$ that form a principal basis.

In particular, $\SR$ is not ample nor ample up to avoidance.
\end{proposition}

\section{Flexibility of step-$2$ distributions} \label{sec:flexBG}

In this Section we will prove the $h$-principle for step-$2$ distributions (Theorem \ref{thm:step2}) and its corollary about the classification of (3,5) and (3,6) distributions (Theorem \ref{thm:35and36}). The proof can be found in Subsection \ref{ssec:flexBG}. We emphasise that the contents of this Section do not need avoidance and simply rely on classic convex integration.

Before we get to the results, and in order to set notation, we recall some of the basic theory of tangent distributions in Subsection \ref{ssec:dual}. This will allow us, in Subsection \ref{ssec:liftingToAnnihilator}, to translate our statements about distributions to statements about their annihilating forms.

In Subsection \ref{ssec:contact}, for completeness, we look at convex integration in the setting of (even-)contact structures, following the work of McDuff \cite{McD}.

\subsection{The dual picture} \label{ssec:dual}

Fix an $n$-dimensional ambient manifold $M$ and a rank-$k$ distribution $\xi$. Recall that the Lie flag and the nilpotentisation were already introduced in the introductory Subsection \ref{ssec:distributionsIntro}. 

In practice, whenever we impose (natural) differential conditions on distributions, these can be read either using a frame of vector fields or a frame of the annihilator. In this paper, the distributions we look at have greater rank than corank, so it is more convenient to use the annihilator viewpoint:
\[ \xi^{\bot} := \lbrace \alpha\in T^*M \mid \alpha(v)=0, \,\forall v\in\xi \rbrace. \]

In Subsection \ref{sssec:nilpotentisation} we discussed the curvature of $\xi$. Its first entry is a $2$-form with entries in $\xi$ and image in $TM/\xi$. Upon passing to the wedge product, it is equivalent to a bundle morphism
\[ \Omega^\xi: \xi \wedge \xi \longrightarrow TM/\xi. \]
We can then dualise it using the Cartan formula, yielding a bundle map
\begin{equation*}
		\begin{array}{rrcl}
		d^\xi:  & \xi^\bot						 & \longrightarrow &  \wedge^2 \xi^* \\
		        &\alpha                & \longmapsto     & -\alpha \circ \Omega^\xi = d\alpha|_\xi.
		\end{array}
\end{equation*}
We note that $d\alpha|_\xi$ only depends on the pointwise value of $\alpha \in \xi^\perp$. At the risk of overloading our notation, we will say that $d^\xi\alpha$ is the \textbf{curvature} associated to $\alpha \in \xi^\bot$.

In light of the Cartan formula, the kernel of $d^\xi$ is $\xi^{\bot}_2$. In particular:
\begin{lemma} \label{lem:2step}
$\xi$ is of step-$2$ if and only if any of the following equivalent conditions holds:
\begin{itemize}
\item $d_\xi$ is a monomorphism.
\item $\Omega^\xi$ is an epimorphism.
\end{itemize}
\end{lemma}
In particular, if $\xi$ is of step-$2$, the ambient dimension is at most $\rank(\xi) + \binom{\rank(\xi)}{2}$. Conversely, under this assumption on the dimension, generic distributions are of step-$2$ at a generic point.

\subsection{Step-$2$ as a differential relation}\label{ssec:liftingToAnnihilator}

We see rank-$k$ distributions as sections of the grassmannian bundle $\Gr_k(TM)$. Being bracket--generating in two steps is then a differential relation $\SR^\step \subset J^1(\Gr_k(TM))$ of first order. More concretely, we observe that any element $F \in J^1(\Gr_k(TM))$ defines a $k$-plane $j^0F$, as well as an associated curvature
\[ \Omega^F:  \wedge^2 j^0F \longrightarrow TM/j^0F, \]
simply because the curvature depends only on first order derivatives. Then:
\begin{definition} \label{def:bgDist}
The differential relation $\SR^\step \subset J^1(\Gr_k(TM))$ consists of those $F \in J^1(\Gr_k(TM))$ such that $\Omega^F$ is an epimorphism.
\end{definition}
We will write $\Dist_{(k,n)}^f(M)$ for the space of formal solutions of $\SR^\step$. The subspace of holonomic ones is denoted by $\Dist_{(k,n)}(M)$. Being bracket-generating in $l+1$ steps is similarly a differential relation of order $l$.

Definition \ref{def:bgDist} is not very practical and it is best to pass to a description in terms of forms. Namely, we consider the bundle of tuples $\formConfig$. Over the open set $\overline{\formConfig}$ consisting of linearly-independent tuples, we have a quotient map: 
\[ \pi: \overline{\formConfig} \longrightarrow \Gr_k(TM), \]
which induces a map 
\[ j^r\pi: J^r(\overline{\formConfig}) \longrightarrow J^r(\Gr_k(TM)) \]
between jet spaces. Then:
\begin{definition} \label{def:bgDist2}
A jet $\widetilde F \in J^1(\formConfig)$ is \textbf{formally bracket-generating of step-$2$} if the following conditions hold: 
\begin{itemize}
		\item $\widetilde F \in J^1(\overline{\formConfig})$ and therefore it projects to an element $F \in J^1(\Gr_k(TM))$. Denote $\xi = j^0F$.
    \item The $2$-forms $d\tilde{F}|_\xi$ are linearly independent.
\end{itemize}
The subset of such $\widetilde F$ will be denoted by $\SS^\step \subset J^1(\formConfig)$.
\end{definition}
The second condition is, according to Lemma \ref{lem:2step}, indeed equivalent to $\Omega^F$ being an epimorphism. It follows that:
\begin{lemma} \label{lem:bgDist}
$\SS^\step$ is the preimage of $\SR^\step$ under $j^1\pi$. In particular, $\SS^\step$ fibres affinely over $\SR^\step$.
\end{lemma}

\subsection{Localisation to a ball} \label{ssec:localisation}

As explained before, a formal solution in $F \in \Dist_{(k,n)}^f(M)$ defines a $k$-plane field $j^0F$. It may very well happen that the annihilator $(j^0F)^\bot$ is not trivial as a bundle. This would imply that $F$ cannot be lifted to $\widetilde F \in J^1(\overline{\formConfig})$. In particular, there may be no global lift of $F$ to $\SS^\step$.

Nonetheless, the $h$-principle for $\SR^\step$ reduces to the $h$-principle for $\SS^\step$. We could directly invoke that convex integration is local (i.e. that it is performed chart by chart). This would certainly be enough for our purposes in this Section, which rely on ampleness along all directions. However, it seems more delicate for ampleness up to avoidance.

In order to set the stage for later Sections, we follow a different approach. The following standard trick gets the job done (even in the presence of parameters and relatively): The manifold $M$ (or the product $M \times K$, in the presence of a parameter space $K$) can be triangulated and holonomic approximation \cite[Chapter 3]{EM} can be applied along the codimension-$1$ skeleton $\ST$. This homotopes the formal solution $F$ (resp. $K$-family of formal solutions) to a new formal solution $G \in \Dist_{(k,n)}^f$ (resp. $K$-family of formal solutions) that is holonomic along $\ST$ and $C^0$-close to $F$ everywhere. See \cite{Gav} for the general theory behind this.

The outcome is that now we can restrict our attention to the top dimensional cells, which are contractible. Over each ball, the annihilator of $j^0G$ is now trivial, and thus a lift to $\SS^\step$ exists. We conclude that:
\begin{lemma} \label{lem:localisation}
In order to prove the $h$-principle for $\SR^\step$, it is sufficient to prove it for $\SS^\step$.
\end{lemma}
The $h$-principle for $\SS^\step$ will follow from convex integration, as we prove next.

\subsection{h-Principle for step-2 distributions} \label{ssec:flexBG}

\begin{proof}[Proof of Theorem \ref{thm:step2}]
According to Lemma \ref{lem:localisation}, it is sufficient to check that $\SS^\step$ is ample along all codirections. Fix a formal solution $F = (F_i)_{i=1}^{n-k} \in \SS^\step$ based at some point $p \in M$. By assumption, the $j^0F = (j^0F_i)_{i=1}^{n-k}$ are linearly independent and thus annihilate a $k$-plane $\xi \subset T_pM$. Furthermore, the $dF_i|_\xi$ are linearly independent.

Fix a codirection $\lambda \in T_p^*M$. The principal space associated to $\lambda$ and $F$ reads:
\[ \Pr_{\lambda,F} =  \{(F_i + (0,\lambda \otimes \beta_i))_{i=1}^{n-k}  \,\mid\, \beta_i \in T_p^*M\}. \]
The differential of any $\widetilde F \in \Pr_{\lambda,F}$ reads $(d\widetilde F_i = dF_i + \lambda \wedge \beta_i)_{i=1}^{n-k}$. A tuple $\widetilde F$ belongs to $\SS^\step$ if and only if the tuple of two-forms $d\widetilde F|_\xi$ is linearly independent. Suppose $\lambda \in \xi^\bot$. Then we have that $d\widetilde F|_\xi = dF|_\xi$ and therefore ampleness holds (because all $\widetilde F$ are formal solutions, since $F$ was).

Otherwise, we suppose that $\lambda$ represents a non-trivial element in $\xi^*$. Then, as far as $d\widetilde F|_\xi$ is concerned, only the restriction of $\beta_i$ to $\xi \cap \ker(\lambda)$ is important. The forms $(dF_i|_{\ker(\lambda) \cap \xi})_{i=1}^{n-k}$ span a subspace, say, of dimension $l$. Up to a change of basis, we may then assume that 
\[ dF_i|_{\ker(\lambda) \cap \xi} = 0, \qquad \text{for all $i=l+1, \cdots,n-k$}. \]
Equivalently:
\[ dF_i|_\xi = \lambda \wedge \nu_i, \qquad \text{for all $i=l+1, \cdots,n-k$, for some $\nu_i \in T^*M$}. \]
Then, the tuple $d\widetilde F \in \Pr_{\lambda,F}$ is in $\SS^\step$ if and only if the forms $\{(\lambda \wedge (\beta_i + \nu_i))|_\xi\}_{i=l+1}^{n-k}$ are linearly independent. Equivalently, if and only if the forms $\{(\beta_i + \nu_i)|_{\ker(\lambda) \cap \xi}\}_{i=l+1}^{n-k}$ are linearly independent. This means that the ampleness of $\SS^\step$ along $\Pr_{\lambda,F}$ is equivalent to the ampleness of the subspace $A$ of rank-$(n-k-l)$ matrices within $\SM_{(n-k-l) \times (k-1)}$.

If $n-k-l > k-1$ (equivalently $l < n-2k+1$), the subspace $A$ is empty contradicting the fact that $F$ was a formal solution. Otherwise, $n-k-l \leq k-1$ holds and  $A$ is just the subspace of non-degenerate matrices. Due to our assumptions (step $2$ and dimension at least $4$), we have that $k \geq 3$ and thus we deduce $k-1 \geq 2$. It follows that $A$ is ample according to Lemma \ref{cor:linearAlgebra}, concluding the proof.
\end{proof}

Theorem \ref{thm:step2} proves that step-$2$ distributions are flexible as long as we do not impose any further non-degeneracy constraints. Nonetheless, according to Theorem \ref{thm:35and36}, there are two cases, (3,5) and (3,6), where the $h$-principle for maximally non-involutive distributions readily follows from the Theorem.

\begin{remark}
Do observe that the singularity associated to $\SS^\step$ is thin unless the dimension of $\wedge^2(\xi)$ is the corank of $\xi$. When the two dimensions coincide, the singularity has codimension-$1$ and, in each principal subspace, is either trivial or equivalent to the singularity of $\GL$ inside of all square matrices. This is the case for the distributions of maximal growth $(k,k+\binom{k}{2})$, which includes $(3,6)$, $(4,10)$ and so on.
\end{remark}

\subsection{The contact and even-contact cases} \label{ssec:contact}

As an appetiser for our study of maximally non-involutive distributions in Sections \ref{sec:maximalNonInvolutivity} and \ref{sec:46}, we now revisit the contact and even-contact cases. We will show that the former fails to be ample (as was to be expected, since contact structures do not abide by the $h$-principle \cite{Ben}), whereas the latter is thin along all principal directions (as proven by D. McDuff in \cite{McD}).

\subsubsection{The Pfaffian}

Let $M$ be an $n$-dimensional manifold. Once again, it is convenient, since we are dealing with hyperplane fields, to work with forms. The bundle of interest for us will be the cotangent bundle $T^*M$. In order to measure non-involutivity, we introduce the \emph{Pfaffian} map:
\begin{equation*}\label{eq:pfaffianContact}
		\begin{array}{ccccccl}
		\Gamma \colon  & J^1(T^*M) & \longrightarrow &  T^*M \oplus\wedge^2 T^*M & \longrightarrow & \wedge^{2\lfloor \frac{n-1}{2} \rfloor+1} T^*M \\
		               & F         & \longmapsto     & (j^0F, dF)          &  \longmapsto    & j^0F \wedge(dF)^{\lfloor \frac{n-1}{2} \rfloor},
		\end{array}
\end{equation*}
where the first arrow is the extended symbol of the exterior differential. The Pfaffian measures whether the formal curvature $dF|_{\ker(j^0F)}$ has maximal rank.

\begin{definition}
The (even)-contact differential relation for $1$-forms is defined as:
\[ \SR^\cont := J^1(T^*M \setminus 0) \setminus \Gamma^{-1}(0) \subset J^1(T^*M).\]
\end{definition}
Once again we emphasise that one can pass from distributions to forms locally (see Subsection \ref{ssec:localisation}), so the $h$-principle for (even-)contact structures is equivalent to the $h$-principle for $\SR^\cont$. We study its ampleness next.

\subsubsection{Checking ampleness}

Fix a coordinate direction $\lambda$ and a formal solution $F \in \SR^\cont$, both based at a point $p \in M$. The two together define the principal subspace
\[ \Pr_{\lambda,F} := \{ F + (0,\lambda \otimes \beta) \,\mid\, \beta \in T_p^*M \} \]
which maps, using the extended symbol of $d$, to:
\[ d\Pr_{\lambda,F} := \{ (j^0F, dF + \lambda \wedge \beta) \,\mid\, \beta \in T_p^*M \} \subset T_p^*M \oplus\wedge^2 T_p^*M. \]
We write $\xi = \ker(j^0F)$.

A point $\widetilde F = F + (0,\lambda \otimes \beta) \in \Pr_{\lambda,F}$ is formally (even-)contact if and only if:
\[ \Gamma(\widetilde F) = n \,  j^0F \wedge (dF)^{\lfloor \frac{n-1}{2} \rfloor - 1} \wedge (dF + \lambda \wedge \beta) \neq 0.\]
Since $F$ was a formal solution, there are four possible situations:
\begin{itemize}
\item[1.] $\lambda$ is proportional to $j^0F$. 
\item[2.] $\lambda$ is not proportional to $j^0F$ and $n$ is odd. Then $dF$ has a 1-dimensional kernel $L$ when restricted to $\xi \cap \ker(\lambda)$.
\item[3a.] $\lambda$ is not proportional to $j^0F$, $n$ is even, and $\ker(\lambda)$ contains the $1$-dimensional kernel of $dF|_\xi$. Then $dF$ has a 2-dimensional kernel $L'$ when restricted to $\xi \cap \ker(\lambda)$.
\item[3b.] $\lambda$ is not proportional to $j^0F$, $n$ is even, and $\ker(\lambda)$ is transverse to the $1$-dimensional kernel of $dF|_\xi$. Then $dF$ is non-degenerate when restricted to $\xi \cap \ker(\lambda)$.
\end{itemize}

Situation (1) means that $\Gamma(\widetilde F) = \Gamma(F) \neq 0$, so ampleness holds trivially.  

Situation (2) corresponds to the (non-trivial) contact case. Then, $\SR^\cont \cap \Pr_{\lambda,F}$ corresponds to those choices of $\beta$ that evaluate non-zero on $L$. The complement is then a hyperplane, proving that:
\begin{lemma} \label{lem:contact}
The differential relation describing contact structures is not ample. In fact, along any given principal subspace, $\SR^\cont \cap \Pr_{\lambda,F}$ is either trivially ample or not ample.
\end{lemma}

Situations (3a) and (3b) correspond to the (non-trivial) even-contact case. In (3a), $\SR^\cont_{\lambda,F}$ corresponds to those choices of $\beta$ that evaluate non-zero on $L$. Its complement is codimension-$2$. In (3b), ampleness holds trivially since $\SR^\cont_{\lambda,F} = \Pr_{\lambda,F}$. We have shown:
\begin{lemma} \label{lem:evenContact}
The differential relation describing even-contact structures has thin complement.
\end{lemma}
The $h$-principle for even-contact structures follows then from classic convex integration Theorem \ref{thm:convexIntegration1}.

\section{Maximal non-involutivity} \label{sec:maximalNonInvolutivity}

Fix two positive integers $k < n$. We want to define maximal non-involutivity for step-$2$ distributions of rank $k$ in dimension $n$. Just like in Subsection \ref{ssec:contact}, we will use the Pfaffian map to capture this in an algebraic manner (Subsection \ref{ssec:Pfaffian}). In Subsection \ref{ssec:4distributions} we will particularise the discussion to the rank-4 case. Much of what we explain in this Section we learnt from the book by R. Montgomery \cite{Mon02}.

\subsection{The Pfaffian and degenerate differential forms} \label{ssec:Pfaffian}

Let $\xi$ be a distribution of rank $k$ in an $n$-dimensional manifold $M$. We measure the non-involutivity of $\xi$ using the curvature $\Omega^\xi$. Recall that $\Omega^\xi$ is a $2$-form with entries in $\xi$ and values in $TM/\xi$. Dually, we think of it as a $\xi^\bot$-family of 2-forms. As in Subsection \ref{ssec:contact}, we can take the highest potentially-non-trivial power of the curvatures using the map:
\begin{equation*}
\begin{array}{rccl}
		p: & \wedge^2\xi^* & \longrightarrow &  \wedge^{2\lfloor \frac{k}{2} \rfloor}\xi \\
		   & \omega        & \longmapsto     &\omega^{\lfloor \frac{k}{2} \rfloor}.
\end{array}
\end{equation*}

That is, a curvature gets mapped to a top-form when the rank $k$ is even, and to a codimension-1 form when $k$ is odd.
\begin{definition}
A $2$-form $\omega \in \wedge^2\xi^*$ is \textbf{degenerate} if $p(\omega)=0$.
\end{definition}
The map $p$ is algebraic and of degree $\lfloor \frac{k}{2} \rfloor$. Its zero level set $\SC$ consists of the degenerate $2$-forms. It has codimension $1$ inside $\wedge^2 \xi^*$ if $k$ is even and codimension $k$ if $k$ is odd. Then we define:
\begin{definition} \label{def:pfaffian}
The composition $p \circ d^\xi$ is called the \textbf{Pfaffian}:
\begin{equation*}\label{eq:pfaffian}
		\begin{array}{rcclcl}
		\Pf: & \xi^\bot & \longrightarrow &  \wedge^2 \xi^* & \longrightarrow & \wedge^{2\lfloor \frac{k}{2} \rfloor}\xi^* \\
		    & \alpha & \longmapsto     & d\alpha|_\xi    & \longmapsto     &\left(d\alpha|_{\xi}\right)^{\lfloor \frac{k}{2} \rfloor}.
		\end{array}
\end{equation*}
\end{definition}
We are interested in the Diff-invariant non-degeneracy condition:
\begin{definition}
Let $k$ be even. A step-$2$ distribution $\xi$ is \textbf{maximally non-involutive} if $\Pf$ intersects $\SC$ transversely.
\end{definition}
Maximal non-involutivity may be described similarly (but differently) for $k$ odd. This is unnecessary for the purposes of this article. We now focus on rank-$4$ distributions.

\subsection{$4$-distributions}\label{ssec:4distributions}

Since $k=4$ is even, the target of $p$ is a $1$-dimensional line bundle; namely, the determinant of $\xi$. Assuming orientability of $\xi$, which we can do locally, we can fix a volume form on $\xi$ to trivialise it. This allows us to see $p$ as a quadratic form on $\wedge^2 \xi^*$ and study its signature. The signature does not depend on the choice of volume form:
\begin{lemma}
The real quadratic form $p$ has signature $(3,3)$.
\end{lemma}
\begin{proof}
Take a local frame $\xi^* = \langle \beta_1, \beta_2, \beta_3, \beta_4\rangle$ compatible with the chosen orientation. Define now the space of self-dual forms $\bigwedge^{+}(\xi^*)$ and the space of anti self-dual forms $\bigwedge^{+}(\xi^*)$ as follows:
\[ \bigwedge^{+}(\xi^*) = \langle a_1 = \beta_1\wedge\beta_2+\beta_3\wedge\beta_4,\,
                                  a_2 = \beta_1\wedge\beta_3+\beta_4\wedge\beta_2,\,
																	a_3 = \beta_1\wedge\beta_4+\beta_2\wedge\beta_3 \rangle \subset \wedge^2\xi^*. \]
\[ \bigwedge^{-}(\xi^*) = \langle b_1 = \beta_1\wedge\beta_2-\beta_3\wedge\beta_4,\,
																	b_2 = \beta_1\wedge\beta_3-\beta_4\wedge\beta_2,\,
																	b_3 = \beta_1\wedge\beta_4-\beta_2\wedge\beta_3 \rangle \subset \wedge^2\xi^*.\]
A straightforward computation shows that the matrix associated to the bilinear form $p$ with respect to the basis $\langle a_1, a_2, a_3, b_1, b_2, b_3\rangle$ consists of an upper-left $\Id_{3\times 3}$ identity block and another $-\Id_{3\times 3}$ in the right-down corner. I.e. $p$ diagonalises with the claimed signature.
\end{proof}

Since $\xi$ is bracket--generating, the exterior differential $d^\xi$ maps $\xi^\bot$ injectively into $\wedge^2 \xi^*$. We can then talk about the signature of $p$ restricted to the image. This is equivalent to:
\begin{definition}
The signature of a distribution $\xi$ is the signature of the quadratic form $\Pf: \xi^\bot \to \wedge^2(\xi^*) \cong \R$.
\end{definition}
Two remarks are in order. First: the signature is well-defined only once a volume form on $\xi$ has been chosen. Otherwise, we cannot distinguish the signatures $(i,j,k)$ and $(j,i,k)$. Furthermore, the signature of $\xi$ may vary from point to point.

We focus on the case of $(4,6)-$distributions:
\begin{definition}
A bracket--generating 4-distribution in a $6$-dimensional manifold $M$ is said to be \textbf{maximally non-involutive} if, at all points, any of the following equivalent conditions holds:
\begin{itemize}
\item The map $d^\xi$ is transverse to the locus of degenerate $2$-forms.
\item The Pfaffian is transverse to zero.
\item The Pfaffian is non-degenerate as a quadratic form.
\end{itemize}
\end{definition}

Furthermore, we can distinguish two different types of $(4,6)$-distributions:
\begin{definition} \label{def:hyperbolic}
A maximally non-involutive $(4,6)$-distribution is
\begin{itemize}
\item[i.] \textbf{elliptic} or \textbf{fat} if the signature is definite.
\item[ii.] \textbf{hyperbolic} if the signature is mixed.
\end{itemize}
\end{definition}
Since $p$ has signature $(3,3)$ and $\xi^\bot$ has dimension $2$, there are, up to changing the orientation, four possible signatures for $\xi$: $(0,0,2), (1,0,1), (1,1,0)$ and $(2,0,0)$. Only the last two cases are maximally non-involutive. They correspond, respectively, to the hyperbolic and elliptic cases.

\subsection{Formal maximally non-involutive $4$-distributions} \label{ssec:formal4distributions}

As in Section \ref{sec:flexBG}, it is more convenient not to work with the distribution itself but with its annihilating forms. We define:
\begin{definition} \label{def:formal4distributions} 
A formal datum $F = (F_i)_{i=1,2} \in J^1(T^*M \oplus T^*M)$ is said to be \textbf{formally elliptic} if:
\begin{itemize}
\item the $1$-forms $j^0F_i$ are linearly independent and thus span a $4$-plane $\xi$.
\item the $2$-forms $dF_i|_\xi$ are linearly independent and span a $2$-plane of definite signature.
\end{itemize}

The formal datum $F$ is \textbf{formally hyperbolic} if, instead:
\begin{itemize}
\item the $1$-forms $j^0F_i$ are linearly independent and thus span a $4$-plane $\xi$.
\item the $2$-forms $dF_i|_\xi$ are linearly independent and span a $2$-plane of mixed signature.
\end{itemize}
\end{definition}
A first jet of distribution is elliptic/hyperbolic if and only if any pair of forms $(F_i)_{i=1,2}$ representing it is elliptic/hyperbolic.

We write $\SR^\ellip$ for the differential relation defining elliptic $(4,6)$ distributions. Similarly, $\SR^\hyp$ denotes the differential relation consisting of formal hyperbolic $(4,6)$-distributions. Their counterparts at the level of forms are denoted by $\SS^\ellip$ and $\SS^\hyp$, respectively. As in Subsection \ref{ssec:localisation}, one can use holonomic approximation to reduce to the case of a ball, proving that:
\begin{lemma} \label{lem:reductionToFormsHyp}
The full $C^0$-close $h$-principle for $\SR^\hyp$ reduces to the full $C^0$-close $h$-principle for $\SS^\hyp$.
\end{lemma}
The upcoming final Section of the paper deals with the construction of an avoidance template for $\SS^\hyp$.

\section{h-Principle for hyperbolic $(4,6)$-distributions}\label{sec:46}

In this Section we tackle the proof of Theorem \ref{thm:46}. According to Lemma \ref{lem:reductionToFormsHyp}, the $h$-principle for $\SR^\hyp$ will follow from the $h$-principle for formally hyperbolic pairs of $1$-forms $\SS^\hyp \subset J^1(T^*M \oplus T^*M)$. Applying Theorem \ref{thm:main} we see that we just need to construct an avoidance template for $\SS^\hyp$. 

\subsection{First avoidance step}

As we advanced in Subsection \ref{ssec:resultsConvexIntegration}, $\SS^\hyp$ (and thus $\SR^\hyp$) does not intersect all principal subspaces in ample sets, so avoidance will act non-trivially. Before we provide a precise statement, we need to introduce some notation.

\subsubsection{The singularity associated to non-ampleness} \label{sssec:condition1}

We define $\Sigma^{(1)} \subset \SS^\hyp \times_M T^*M$ as the subspace of pairs $(F,\lambda)$ such that 
\begin{equation}\label{eq:condition1}
j^0F_1 \wedge j^0 F_2 \wedge \lambda \wedge dF_1
\quad\text{and}\quad j^0F_1
\wedge j^0 F_2 \wedge \lambda \wedge dF_2
\quad\text{are linearly dependent.}
\end{equation}
We write $\Sigma^{(1)}(\lambda) \subset \SS^\hyp$ for the subset of those $F$ such that $(F,\lambda) \in \Sigma^{(1)}$. Similarly, $\Sigma^{(1)}(F) \subset T^*M$ denotes those $\lambda$ such that $(F,\lambda) \in \Sigma^{(1)}$.

\begin{lemma} \label{lem:algebraic1}
The following statements hold:
\begin{itemize}
\item $\Sigma^{(1)}$ is a closed subset of $\SS^\hyp \times_M T^*M$.
\item All the fibres $(\Sigma^{(1)}_p)_{p \in M}$ are isomorphic algebraic subvarieties.
\item Fix $F \in \SS^\hyp$ lying over $p \in M$. The subspace $\Sigma^{(1)}(F)$ has positive codimension in $T_p^*M$.
\end{itemize}
\end{lemma}
\begin{proof}
$\SS^\hyp$ is Diff-invariant, and therefore all its fibres are isomorphic to one another. Furthermore, the expressions in Equation \ref{eq:condition1} are algebraic on their entries and linear dependence is itself a closed algebraic condition. These statements prove the first two claims.

For the last claim, we observe that fixing $F$ yields still an algebraic equality for $\lambda$ that is non-trivial as long as $j^0F_1 \wedge j^0 F_2 \wedge dF_1$ and $j^0F_1 \wedge j^0 F_2 \wedge dF_2$ are linearly independent. This is indeed the case if $F \in \SS^\hyp$.
\end{proof}

Write $\xi \subset T_pM$ for the $4$-plane given as the kernel of $j^0F$. We note that the following are equivalent:
\begin{itemize}
\item $(F,\lambda) \in \Sigma^{(1)}$.
\item The $3$-forms $(\lambda \wedge dF)|_\xi$ are linearly-dependent.
\item $\lambda|_\xi$ is zero or the $2$-forms $dF|_{\xi \cap \ker(\lambda)}$ are linearly-dependent.
\end{itemize}

\subsubsection{Main statement}

We claim that $\Sigma^{(1)}$ is precisely the set to be removed in order to carry out the first avoidance step.
\begin{proposition} \label{prop:firstAvoidance}
Let $F \in \SS^\hyp$ and $\lambda \in T_p^*M$, both based at the same point $p \in M$. Write $\xi \subset T_pM$ for the $4$-plane defined by $j^0F$. Then:
\begin{itemize}
\item[i.]   $\SS^\hyp_{\lambda,F}$ is non-trivially ample if and only if $(F,\lambda) \notin \Sigma^{(1)}$.
\item[ii.]  $\SS^\hyp_{\lambda,F}$ is trivially ample if and only if $\lambda|_\xi = 0$. In particular, $(F,\lambda) \in \Sigma^{(1)}$.
\item[iii.] $\SS^\hyp_{\lambda,F}$ is not ample otherwise. I.e. if $(F,\lambda) \in \Sigma^{(1)}$ but $\lambda|_\xi \neq 0$.
\end{itemize}
\end{proposition}

Let us provide some geometric insight before we get into the proof. An element $\widetilde F$ in the principal subspace $\Pr_{\lambda,F}$ maps under the exterior differential to a pair $(d\widetilde F_i = dF_i + \lambda \wedge \beta_i)_{i=1,2}$, where the $\beta_i$ range over $T_p^*M$.

According to Definition \ref{def:formal4distributions}, the pair $dF|_\xi$ spans a plane of $2$-forms $L$. The restriction $p|_L$ is a bilinear form of mixed signature, due to hyperbolicity. Now consider the subspace 
\[ K := \{(\lambda \wedge \beta)|_\xi \,\mid\, \beta \in T_p^*M\} \subset \wedge^2\xi^*. \]
By definition, given any other element $\widetilde F \in \Pr_{\lambda,F}$, the pair $d\widetilde F|_\xi$ is obtained from $dF|_\xi$ by shifting each form $dF_i|_\xi$ along $K$. As such, when $K$ and $L$ are transverse, the pair $d\widetilde F|_\xi$ will span a plane $\widetilde L$ that is a graph over $L$ in the direction of $K$. If transversality fails, it may very well happen that the pair $d\widetilde F|_\xi$ is linearly dependent; however, we still think of its span $\widetilde L$ as a degenerate plane.

We further note that being a graph over $L$ in the direction of $K$ is an intrinsic characterisation of the planes $\widetilde L$ associated to elements $\widetilde F \in \Pr_{\lambda,F}$. I.e. the set of all such planes does not depend on the concrete basis $dF$ of $L$. We furthermore note that $\widetilde F \in \SS_{\lambda,F}$ if and only if $p|_{\widetilde L}$ is non-degenerate of mixed signature. This means that all relevant properties of $\widetilde F$ can be read from $\widetilde L$. We conclude that we are allowed to choose a convenient basis of $L$ in order to simplify our computations.

\begin{proof}[Proof of Proposition \ref{prop:firstAvoidance}]
Consider $\widetilde F \in \Pr_{\lambda,F}$ and restrict $p$ to its (possibly degenerate) span $\widetilde L$. This restriction can be represented by the $2$-by-$2$ matrix
\begin{align}\label{determinanteMetrica}
\begin{pmatrix}
g_{11} & g_{12} \\
g_{12} & g_{22} \\
\end{pmatrix}
\end{align} 
whose coefficients read:
\begin{align*}
g_{11}(\beta_1,\beta_2) & = (dF_1 + \lambda \wedge \beta_1)^2 
         = dF_1^2 + 2\lambda \wedge \beta_1\wedge dF_1, \\
g_{22}(\beta_1,\beta_2) & = (dF_2 + \lambda \wedge \beta_2)^2
				 = dF_2^2 + 2\lambda\wedge\beta_2\wedge dF_2, \\
g_{12}(\beta_1,\beta_2) & = (dF_1 + \lambda \wedge \beta_1)\wedge(dF_2 + \lambda \wedge \beta_2)
         = dF_1\wedge dF_2 + \lambda\wedge(\beta_2\wedge dF_1 + \beta_1\wedge dF_2).
\end{align*} 
Each of these expressions can be identified with a real function by fixing a volume form on $\xi$. We fix such a volume; all our upcoming computations do not depend on this auxiliary choice.

We have effectively defined an affine map that takes values in the space of symmetric $2$-by-$2$ matrices, which we think of as $\R^3$:
\begin{equation*}\label{eq:cuadratic}
\begin{array}{rccl}
	\Psi \colon & \xi^* \oplus \xi^* & \longrightarrow &  \R^3\\
		          & (\beta_1, \beta_2)  & \longmapsto & \left(g_{11}(\beta_1,\beta_2), g_{22}(\beta_1,\beta_2), g_{12}(\beta_1,\beta_2)\right)
\end{array}
\end{equation*}
It is now convenient to introduce the determinant, which we see as a quadratic form in the space of symmetric $2$-by-$2$ matrices:
\begin{equation*}\label{eq:det}
		\begin{array}{rccl}
		\det \colon & \R^3   & \longrightarrow &  \R \\
		            & (x,y,z)& \longmapsto     & xy-z^2.
		\end{array}
\end{equation*}
We saw back in Example \ref{ex:symmetricMatrices} that the signature of $\det$ was $(1,2,0)$, so its zero set $\SC$ is a cone. See Figure \ref{fig:cone}. The cone $\SC$ divides the space in $3$ components: The two positive ones we called $\SH^+$; they are convex and thus not ample. The third component $\SH^-$ is the exterior of the cone; it corresponds to the matrices with negative determinant and it is ample. In particular, hyperbolicity is equivalent to $\det \circ \Psi(-) < 0$ and thus equivalent to $\Psi(-) \in \SH^-$. Similarly, ellipticity is equivalent to $\det \circ \Psi(-) > 0$ and thus equivalent to $\Psi(-) \in \SH^+$. We will come back to ellipticity in Lemma \ref{lem:ellipticity} below; for now we focus on hyperbolicity. We want to check ampleness; there are various possibilities, depending on what the image of $\Psi$ is.

\begin{figure}[ht]
		\includegraphics[scale=0.81]{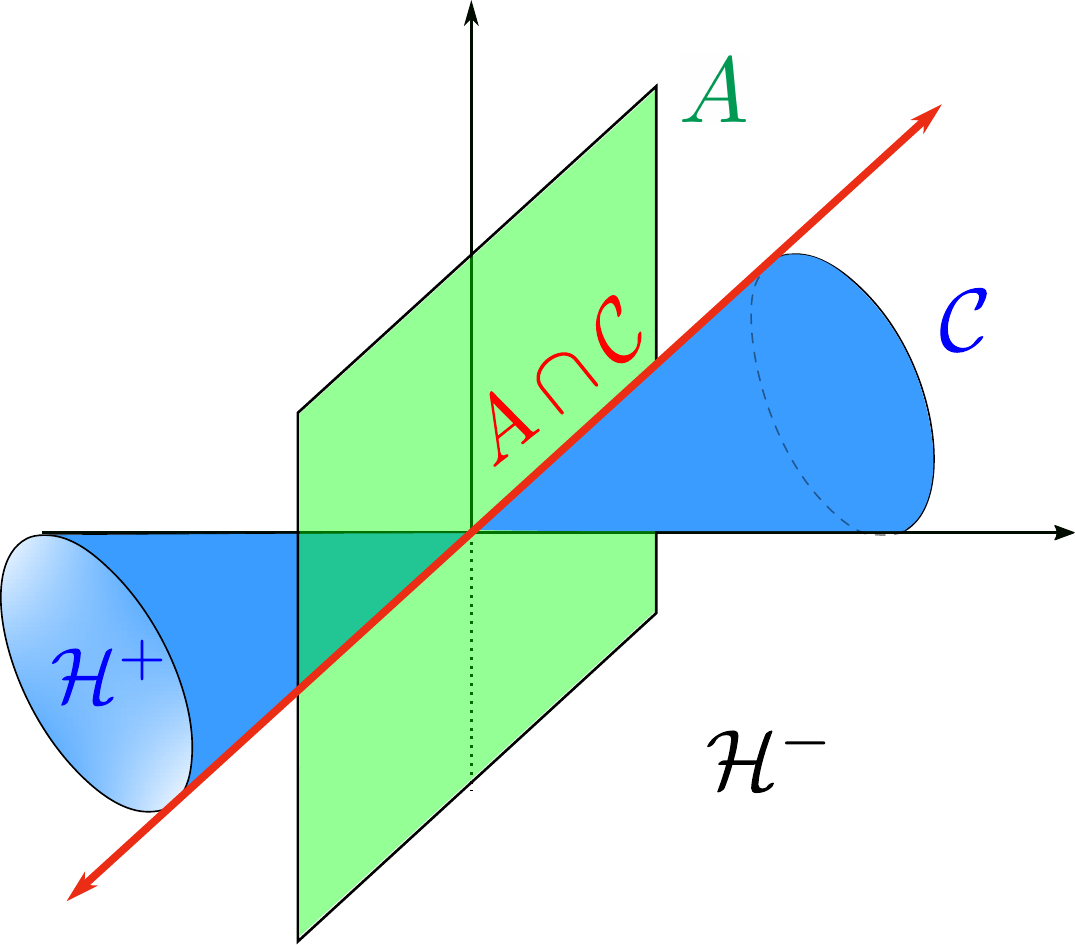}
		\centering
		\caption{The space of 2-by-2 symmetric matrices. In blue, the cone of degenerate matrices. The outside of the cone $\SH^-$ corresponds to matrices with negative determinant. It has a single ample component. The subspace of matrices with positive determinant $\SH^+$ has two non-ample components. The image of $\Psi$ can be a single point, the whole space, or a vertical plane $A$, shown in green. In the latter case, the intersection of $A$ with the cone is a single line, cutting $A$ in two non-ample components.} \label{fig:cone}
\end{figure}

Suppose that $\lambda|_\xi$ is zero. Then the subspace $K$ defined before the proof is zero as well. In particular, $\Psi$ is constant and its image must be in $\SH^-$, since $F$ is a formal solution. It follows that $\SS^\hyp_{\lambda,F} = \Pr_{\lambda,F}$, so ampleness holds trivially. Situation (ii.) holds. We henceforth assume $\lambda|_\xi \neq 0$ and thus $K \neq 0$.

Suppose that the forms $dF|_{\xi \cap \ker(\lambda)}$ are both zero. This is equivalent to $L \subset K$. This means that both $dF_i|_\xi$ are proportional to $\lambda$. However, this readily implies that $dF_i \wedge dF_i = dF_1 \wedge dF_2 = 0$, contradicting the fact that $F$ was a formal solution.

Suppose that the forms $(dF_i|_{\xi \cap \ker(\lambda)})_{i=1,2}$ are linearly independent. This amounts to transversality of $L$ and $K$. It follows that $\Psi$ is surjective. Then, ampleness of $\SS^\hyp$ along $\Pr_{\lambda,F}$ is equivalent to the ampleness of the subspace $\SH^-$ of symmetric matrices with negative determinant (and thus, of mixed signature). This space is indeed ample, but not trivially. Situation (i.) holds.

Lastly, suppose that the $dF|_{\xi \cap \ker(\lambda)}$ are linearly dependent but not identically zero; i.e. $L \cap K$ is $1$-dimensional. Up to changing basis we may assume that $dF_1|_{\xi \cap \ker(\lambda)} = 0$; i.e. $dF_1$ spans the line $L \cap K$. In this case, $dF_2|_{\xi \cap \ker(\lambda)} \neq 0$. Then $dF_1 \wedge dF_1 = 0$, so the image of $\Psi$ is a $2$-dimensional plane $A$ through the origin, tangent to the cone. The restriction $\SH^- \cap A$ consists of two half-spaces, separated by the line $A \cap \SC$. Since this complement $A \cap \SC$ is linear and of codimension-$1$, it is not a thin singularity. Situation (iii.) holds.
\end{proof}

\subsubsection{Conclusion of the first avoidance step}

We now describe $\Avoid(\SS^\hyp)$ using Proposition \ref{prop:firstAvoidance}.
\begin{corollary} \label{cor:firstAvoidance}
$\Avoid(\SS^\hyp)$ consists of those pairs $(F,\Xi) \in \SS^\hyp \times_M \HConf(TM)$ such that 
\[ (F,\lambda) \notin \Sigma^{(1)} \quad \text{for every codirection $\lambda \in \Xi$}. \]
\end{corollary}
\begin{proof}
Write $\xi$ for the kernel of $j^0F$. Recall Situations (i.), (ii.) and (iii.) from Proposition \ref{prop:firstAvoidance}. We define $\Delta_3 \subset \SS^\hyp \times_M \HConf(TM)$ as the subspace of pairs $(F,\Xi)$ for which Situation (iii.) holds for $F$ and some $\lambda \in \Xi$. We define $\Delta_2 \subset \SS^\hyp \times_M \HConf(TM)$ as the subspace of those $(F,\Xi) \notin \Delta_3$ such that Situation (ii.) holds for some $\lambda \in \Xi$.

According to Proposition \ref{prop:firstAvoidance}, $\Delta_3$ consists exactly of those pairs $(F,\Xi)$ such that $\SS^\hyp_{\lambda,F}$ is not ample, for some $\lambda \in \Xi$. By definition, it follows that:
\[ \Avoid(\SS^\hyp) = \SS^\hyp \times_M \HConf(TM) \,\setminus\, \overline{\Delta_3}. \]
We claim that the closure $\overline{\Delta_3}$ is exactly $\Delta_2 \cup \Delta_3$.

We first observe that the closure is contained in the union. Indeed, the pairs $(F,\lambda)$ satisfying Situation (ii.) or (iii.) are precisely the elements of $\Sigma^{(1)}$. Lemma \ref{lem:algebraic1} states that this is a closed subset.

We then prove $\Delta_2 \subset \overline{\Delta_3}$. For fixed $F$, the set of $\nu \in T^*M$ such that Situation (iii.) holds for $(F,\nu)$ is non-empty, invariant under scalings of $\nu$, and depends only on the restriction $\nu|_\xi$. Take $\lambda \in \Xi$ such that $\lambda|_\xi = 0$. Then there is a neighbourhood of $\lambda$ in $T^*M$ that submerses onto a neighbourhood of the zero section in $\xi^*$. This implies that any neighbourhood of $\lambda$ contains codirections $\nu$ such that Situation (iii.) holds for $(F,\nu)$. This proves the claim and concludes the proof.
\end{proof}

\subsubsection{Ampleness fails for (4,6) elliptic distributions}

For completeness, we observe:
\begin{lemma} \label{lem:ellipticity}
Let $F \in \SS^\ellip$ and $\lambda \in T_p^*M$, both based at the same point. Write $\xi \subset T_pM$ for the $4$-plane defined by $j^0F$. Then:
\begin{itemize}
\item $\SS^\ellip_{\lambda,F}$ is trivially ample if and only if $\lambda|_\xi = 0$.
\item $\SS^\ellip_{\lambda,F}$ is not ample otherwise.
\end{itemize}
In particular, for fixed $F$, the set of $\lambda \in T_p^*M$ such that $\SS^\ellip_{\lambda,F}$ is not ample is open and dense.
\end{lemma}
\begin{proof}
We reason as in the proof of Proposition \ref{prop:firstAvoidance}. Using the same notation as there, we have that the image of $\Psi$ is either a point, a plane through the origin tangent to the cone $\SC$, or the whole of $\R^3$. The first case corresponds to trivial ampleness and to $\lambda|_\xi = 0$. The second case cannot happen, as the plane would be disjoint from $\SH^+$. The last case, corresponding to $(F,\lambda) \notin \Sigma^{(1)}$, is not ample as $\SH^+$ consists of two convex components. Density follows from the fact that $\lambda|_\xi = 0$ cuts out a linear subspace.
\end{proof}

Reasoning as in Corollary \ref{cor:firstAvoidance} implies that:
\begin{corollary} \label{cor:ellipticity}
$\Avoid(\SS^\ellip)$ is empty.
\end{corollary}
Which is exactly the same situation as in the contact case (Lemma \ref{lem:contact}). This leads us to conjecture that there is no full $h$-principle for elliptic (4,6) distributions.

\subsection{Second avoidance step}

We will prove now that $\Avoid(\SS^\hyp)$ is not ample along all principal subspaces, so further elements have to be removed. This is different from the example in Section \ref{sec:exactForms}, where a single thinning step was sufficient to produce a thinning template.

Applying standard avoidance would lead us to study $\Avoid^2(\SS^\hyp)$. It turns out that it is difficult to determine whether $\Avoid^2(\SS^\hyp)$ is an avoidance template. The reason is that we do not have a explicit description of the elements removed from $\Avoid(\SS^\hyp)$ to yield $\Avoid^2(\SS^\hyp)$.

Due to this, it is more fruitful to ignore $\Avoid^2(\SS^\hyp)$ altogether and instead construct an avoidance template $\SA \subset \Avoid(\SS^\hyp)$ with more transparent properties.

\subsubsection{The singularity of interest} \label{sssec:condition2}

We define a new singularity
\[ \Sigma^{(2)} \subset \SS^\hyp \times_M T^*M \times_M T^*M \]
as the subspace of pairs $(F,\lambda_1,\lambda_2)$ such that
\begin{equation}\label{eq:condition2}
j^0F_1 \wedge j^0 F_2 \wedge \lambda_1 \wedge \lambda_2 \wedge dF_1 = j^0F_1 \wedge j^0 F_2 \wedge \lambda_1 \wedge \lambda_2 \wedge dF_2 = 0.
\end{equation}
It is also convenient to denote $\Sigma^{(2)}(\lambda_1,\lambda_2) \subset \SS^\hyp$ for the subset of elements $F$ such that $(F,\lambda_1,\lambda_2) \in \Sigma^{(2)}$. Similarly we define $\Sigma^{(2)}(F)$.

We note again that the expressions in Equation \ref{eq:condition2} are algebraic on their entries. Furthermore, these expressions are non-trivial on the $\lambda_i$ as long as $j^0F_1 \wedge j^0F_2 \wedge dF \neq 0$. This proves:
\begin{lemma} \label{lem:algebraic2}
The following statements hold:
\begin{itemize}
\item $\Sigma^{(2)}$ is a closed subset of $\SS^\hyp \times_M T^*M \times_M T^*M$.
\item All the fibres $(\Sigma^{(2)}_p)_{p \in M}$ are isomorphic algebraic subvarieties.
\item Fix $F \in \SS^\hyp$ lying over $p$. The subspace $\Sigma^{(2)}(F)$ has positive codimension in $T_p^*M \oplus T_p^*M$.
\end{itemize}
\end{lemma}

Write $\xi \subset T_pM$ for the $4$-plane given as the kernel of $j^0F$. We note that $(F,\lambda_1,\lambda_2) \in \Sigma^{(2)}$ if and only if at least one of the following conditions holds:
\begin{itemize}
\item $(\lambda_1 \wedge \lambda_2)|_\xi$ is zero.
\item  Both $2$-forms $(dF_i)|_{\xi \cap \ker(\lambda_1) \cap \ker(\lambda_2)}$ are zero. Equivalently, both $(dF_i)|_{\xi}$ are proportional to $(\lambda_1 \wedge \lambda_2)|_\xi$.
\end{itemize}

\subsubsection{Main statement}

We now study the ampleness of $\Avoid(\SS^\hyp)$. This amounts to the following: Given a collection of codirections $\Xi$ and a codirection $\lambda \in \Xi$, we try to determine whether 
\[ \SS^\hyp_{\lambda,F} \setminus (\bigcup_{\nu \in \Xi} \Sigma^{(1)}(\nu)) \]
is an ample subset of $\Pr_{\lambda,F}$, for each formal solution $F$.
\begin{proposition} \label{prop:secondAvoidance}
Fix codirections $\lambda, \nu \in T_p^*M$ and a formal datum $F \in \SS^\hyp$, based also at $p$. Write $\xi \subset T_pM$ for the $4$-plane cut out by $j^0F$.

The following statements hold:
\begin{itemize}
\item Suppose $F \notin \Sigma^{(2)}(\lambda,\nu)$. Then $\Pr_{\lambda,F} \cap \Sigma^{(1)}(\nu)$ is a thin singularity.
\item Suppose $F \in \Sigma^{(2)}(\lambda,\nu)$ but $\lambda \wedge \nu|_\xi \neq 0$. Then $\Pr_{\lambda,F} \setminus \Sigma^{(1)}(\nu)$ is ample but its complement is of codimension $1$.
\end{itemize}
\end{proposition}
\begin{proof}
In both situations we are assuming that the forms $\lambda|_\xi$ and $\nu|_\xi$ are linearly independent. This allows us to define the restriction map
\[ \Phi: \Pr_{\lambda,F} \longrightarrow \wedge^2 (\xi \cap \ker(\nu))^* \oplus \wedge^2 (\xi \cap \ker(\nu))^* \]
that sends $\widetilde F \in \Pr_{\lambda,F}$ to $d\widetilde F|_{\xi \cap \ker(\nu)}$. Recall that $\widetilde F \in \Sigma^{(1)}(\nu)$ if and only if the pair $d\widetilde F|_{\xi \cap \ker(\nu)}$ is linearly dependent. The upcoming argument follows the proof of Lemma \ref{lem:exactFormsThinning} (Section \ref{sec:exactForms}), since the linear dependence problems under consideration are exactly the same.

Write $L \subset \wedge^2 (\xi \cap \ker(\nu))^*$ for the subspace of $2$-forms proportional to $\lambda|_{\xi \cap \ker(\nu)}$. Due to the linear independence of $\lambda|_\xi$ and $\nu|_\xi$, $L$ is a $2$-dimensional plane. We write $L_i$ for the plane passing through $dF_i$ parallel to $L$. Then, the image of $\Pr_{\lambda,F}$ under $\Phi$ is the sum $L_1 \oplus L_2$, which is $4$-dimensional.

The condition $F \in \Sigma^{(2)}(\lambda,\nu)$ is equivalent to $\Phi(F) = 0$, which in turn is equivalent to $L_1=L_2=L$. Ampleness of $\Pr_{\lambda,F} \setminus \Sigma^{(1)}(\nu)$ is thus equivalent to ampleness of $\GL_2 \subset M_{2 \times 2}$, proving the second claim.

Similarly, $F \notin \Sigma^{(2)}(\lambda,\nu)$ means that $\Phi(F) \neq 0$. I.e. at least one $L_i$ is different from $L$ and therefore does not pass through the origin. We deduce the singularity $\Sigma^{(1)}(\nu)$ has codimension $2$.
\end{proof}

\begin{remark} \label{rem:thinSingularitiesAreGood}
Consider the following elementary facts:
\begin{itemize}
\item Removing a thin singularity from an ample set still yields an ample set.
\item The intersection of two ample sets need not be ample. For instance, the union of two codimension-1 singularities, both having ample complement individually, may separate the space into non-ample pieces. See Figure \ref{fig:cubicas}.
\end{itemize}
Our claim is that these statements largely determine the steps to be taken during avoidance.

Indeed, suppose in our current example that $\SS^\hyp_{\lambda,F}$ is ample, for some $\lambda$ and $F$. Then, the condition $F \notin \Sigma^{(2)}(\lambda,\nu)$ implies that $\SS^\hyp_{\lambda,F} \setminus \Sigma^{(1)}(\nu)$ is ample, according to the first fact. This may still be true even if $F \in \Sigma^{(2)}(\lambda,\nu)$, but the second fact tells us that it need not be. This is even more delicate when there are several codirections involved (which will always be the case in the construction of a template). Avoiding the uncertainty of the second situation effectively forces us to consider the singularity $\Sigma^{(2)}$, and thus prescribes what the second avoidance step must be. 

Our claim is that this type of analysis, which is algorithmic in nature, is not specific to $\SS^\hyp$. Indeed, it must guide the avoidance process of any given differential relation.
\end{remark}

\begin{figure}[ht]
		\includegraphics[scale=0.12]{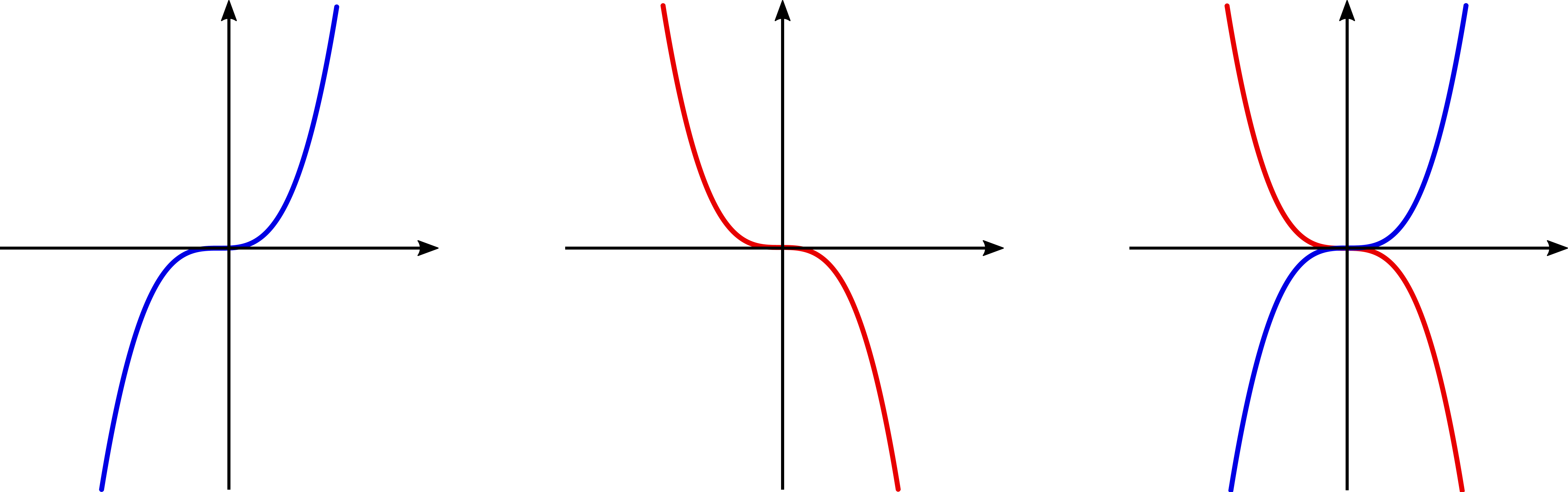}
		\centering
		\caption{On the left and middle, two cubics in $\R^2$. Their complements (as well as themselves) are ample subsets. On the right, we intersect the complements, yielding four components, none of which is ample. The intersection of the two cubics is a point, also not ample.} \label{fig:cubicas}
\end{figure}

\subsubsection{Conclusion of the second step}

Using Proposition \ref{prop:secondAvoidance} we now define:
\[ \SA := \{ (F,\Xi) \in \Avoid(\SS^\hyp) \,\mid\, (F,\lambda_1,\lambda_2) \notin \Sigma^{(2)} \text{ for all } \lambda_1 \neq \lambda_2 \in \Xi \}. \]
\begin{lemma} \label{lem:secondAvoidance}
$\SA$ is a pretemplate.
\end{lemma}
\begin{proof}
As we noted in Lemma \ref{lem:algebraic2}, the condition $(F,\lambda_1,\lambda_2) \in \Sigma^{(2)}$ is closed, smooth in its entries, and algebraic over each given point $p \in M$. Furthermore, for a given $F$ formal solution of $\SS^\hyp$, the condition is non-trivial on $\lambda_1$ and $\lambda_2$. This proves the claim. 
\end{proof}

\subsection{End of the proof}

The proof of Theorem \ref{thm:46} will be complete once we show that:
\begin{theorem} \label{thm:avoidanceAchieved}
$\SA$ is a template for $\SS^\hyp$.
\end{theorem}
This statement requires the following auxiliary result:
\begin{proposition} \label{prop:thirdAvoidance}
Fix codirections $\lambda, \nu_1, \nu_2 \in T_p^*M$ and a formal datum $F \notin \Sigma^{(2)}(\nu_1, \nu_2)$, based also at $p$. Write $\xi \subset T_pM$ for the $4$-plane cut out by $j^0F$.

Then $\Pr_{\lambda,F} \cap \Sigma^{(2)}(\nu_1,\nu_2)$ is a thin singularity.
\end{proposition}
\begin{proof}
Recall that $\Sigma^{(2)}$ is defined by Equation \ref{eq:condition2}. The condition $F \notin \Sigma^{(2)}(\nu_1,\nu_2)$ implies that the forms $\nu_i|_\xi$ are linearly independent. As such, Equation \ref{eq:condition2} reads:
\[ (dF_1)|_{\xi \cap \ker(\nu_1) \cap \ker(\nu_2)} = (dF_2)|_{\xi \cap \ker(\nu_1) \cap \ker(\nu_2)} = 0. \]
Suppose first that $\lambda$ is in the span $\langle \nu_1|_\xi, \nu_2|_\xi\rangle$. Any given element $\widetilde F \in \Pr_{\lambda,F}$ then satisfies 
\[ d \widetilde F|_{\xi \cap \ker(\nu_1) \cap \ker(\nu_2)} = dF|_{\xi \cap \ker(\nu_1) \cap \ker(\nu_2)}, \]
proving that $\Pr_{\lambda,F} \cap \Sigma^{(2)}(\nu_1,\nu_2)$ must be empty.

Suppose instead that $\lambda$ is linearly independent from the other two forms. We can write out $d \widetilde F|_{\xi \cap \ker(\nu_1) \cap \ker(\nu_2)}$ as a pair of $1$-forms 
\[ (dF_i + \lambda \wedge \beta_i)|_{\xi \cap \ker(\nu_1) \cap \ker(\nu_2)}. \]
Equation \ref{eq:condition2} provides then two independent, affine, codimension-$1$ constraints, one for each $\beta_i$. The locus cut out by both has then codimension-$2$, proving thinness.
\end{proof}

\begin{proof}[Proof of Theorem \ref{thm:avoidanceAchieved}]
First we prove that $\SA$ satisfies Property (II) in the definition of template. Fix $(F,\Xi) \in \SA$ and $\lambda \in \Xi$. By construction, such a pair is characterised by the following properties:
\begin{itemize}
\item $F \in \SS^\hyp$.
\item $F \notin \Sigma^{(1)}(\nu)$ for all $\nu \in \Xi$.
\item $F \notin \Sigma^{(2)}(\nu_1, \nu_2)$ for all pairs $\nu_1 \neq \nu_2 \in \Xi$.
\end{itemize}

We need to show that $\SA(\Xi)$ is ample along $\Pr_{\lambda,F}$. An explicit description reads:
\[ \SA(\Xi)_{\lambda,F} = \SS^\hyp_{\lambda,F} \setminus \left[\bigcup_{\nu \in \Xi} \Sigma^{(1)}(\nu) \,\cup\, \bigcup_{\nu_1 \neq \nu_2 \in \Xi} \Sigma^{(2)}(\nu_1, \nu_2) \right]. \]
We then observe:
\begin{itemize}
\item Using condition $F \notin \Sigma^{(1)}(\lambda)$, we apply Proposition \ref{prop:firstAvoidance} to deduce that $\SS^\hyp_{\lambda,F}$ is ample.
\item Using condition $F \notin \Sigma^{(2)}(\lambda, \nu)$, we apply Proposition \ref{prop:secondAvoidance} and deduce that each $\Sigma^{(1)}(\nu)$ is thin.
\item According to Proposition \ref{prop:thirdAvoidance}, the singularities $\Sigma^{(2)}(\nu_1, \nu_2)$ are all thin.
\end{itemize}
The claim follows because ampleness is preserved upon removal of thin singularities.

Secondly, we prove that $\SA$ satisfies Property (III) in the definition of template. We need to show that, for each $F \in \SS^\hyp$ lying over $p \in M$, the subspace $\SA(F)$ is dense in $\HConf(T_pM)$. Its complement can be explicitly written down as
\[ \SA(F)^c \cap \HConf(T_pM) = \{ \Xi \in \HConf(T_pM) \,\mid\, F \in \Sigma^{(1)}(\nu_1) \cup \Sigma^{(2)}(\nu_1, \nu_2) \text{ for some } \nu_1 \neq \nu_2 \in \Xi \}. \]
The conditions $\nu_1 \in \Sigma^{(1)}(F)$ or $(\nu_1, \nu_2) \in \Sigma^{(2)}(F)$ are given, respectively, by Equations \ref{eq:condition1} and \ref{eq:condition2}. Both of them are non-trivial on their entries, since $F$ is in $\SS^\hyp$; see Lemmas \ref{lem:algebraic1} and \ref{lem:algebraic2}. We deduce that $\SA(F)^c \cap \HConf(T_pM)$ is a finite union of positive-codimension subvarieties, proving the claim.
\end{proof}

\begin{proof}[Proof of Theorem \ref{thm:46}]
According to Lemma \ref{lem:reductionToFormsHyp}, the $h$-principle for $\SR^\hyp$ reduces to the $h$-principle for $\SS^\hyp$. Theorem \ref{thm:main} says that the $h$-principle for $\SS^\hyp$ follows from the existence of an avoidance template. Theorem \ref{thm:avoidanceAchieved} yields such a template $\SA$.
\end{proof}

\end{document}